\documentclass[12pt]{article}

\usepackage{aurical}
\usepackage[T1]{fontenc}

\usepackage{amssymb,amsmath,amsthm,bm}
\usepackage{setspace,mathrsfs,nicefrac}
\usepackage{bbm,url,color,wasysym}
\usepackage[margin=1in]{geometry}
\usepackage{paralist}\setdefaultitem{$-$}{\tiny $\bullet$}{\tiny$\circ$}{$\star}

\DeclareMathOperator{\Var}{Var}

\newcommand{\Z}{\mathbb{Z}}

\newcommand{\R}{\mathbb{R}}
\newcommand{\C}{\mathbb{C}}

\renewcommand{\P}{\mathrm{P}}
\newcommand{\E}{\mathrm{E}}

\renewcommand{\L}{\mathbb{L}}

\newcommand{\1}{\mathbbm{1}}
\renewcommand{\d}{{\rm d}}

\newcommand{\e}{{\rm e}}
\renewcommand{\geq}{\geqslant}
\renewcommand{\leq}{\leqslant}
\renewcommand{\ge}{\geqslant}
\renewcommand{\le}{\leqslant}
\newcommand{\RNum}[1]{\uppercase\expandafter{\romannumeral #1\relax}}

\author{ Weicong Su, University of Utah}
\title{\bf\Large{On the Peaks of a Stochastic Heat Equation on a Sphere with a Large Radius}}
\date{}
\newtheorem{stat}{Statement}[section]

\newtheorem{theorem}[stat]{Theorem}

\newtheorem{lemma}[stat]{Lemma}

\theoremstyle{definition}
\newtheorem{definition}[stat]{Definition}
\newtheorem{remark}[stat]{Remark}

\numberwithin{equation}{section}

\begin{document}%\onehalfspacing
\maketitle
\begin{abstract}%\\
	For every $R>0$, consider the stochastic heat equation $\partial_{t} u_{R}(t\,,x)=\tfrac12 \Delta_{S_{R}^{2}}u_{R}(t\,,x)+\sigma(u_{R}(t\,,x)) \xi_{R}(t\,,x)$ on $S_{R}^{2}$, where $\xi_{R}=\dot{W_{R}}$ are centered Gaussian noises with the covariance structure given by $\E [\dot{W_{R}}(t,x)\dot{W_{R}}(s,y)]=h_{R}(x,y)\delta_{0}(t-s)$, where $h_{R}$ is symmetric and semi-positive definite and there exist some fixed constants $-2< C_{h_{up}}< 2$ and $\frac{1}{2}C_{h_{up}}-1 <C_{h_{lo}}\le C_{h_{up}}$ such that for all $R>0$ and $x\,,y \in S_{R}^{2}$, $(\log R)^{C_{h_{lo}}/2}=h_{lo}(R)\leq h_{R}(x,y) \leq h_{up}(R)=(\log R)^{C_{h_{up}}/2}$, $\Delta_{S_{R}^{2}}$ denotes the Laplace-Beltrami operator defined on $S_{R}^{2}$ and $\sigma:\R \mapsto \R$ is Lipschitz continuous, positive and uniformly bounded away from $0$ and $\infty$. Under the assumption that $u_{R,0}(x)=u_{R}(0\,,x)$ is a nonrandom continuous function on $x \in S_{R}^{2}$ and the initial condition that there exists a finite positive $U$ such that $\sup_{R>0}\sup_{x \in S_{R}^{2}}\vert u_{R,0}(x)\vert \le U$, we prove that for every finite  positive $t$, there exist finite positive constants $C_{low}(t)$ and $C_{up}(t)$ which only depend on $t$ such that as $R \to \infty$, $\sup_{x \in S_{R}^{2}}\vert u_{R}(t\,,x)\vert$ is asymptotically bounded below by $C_{low}(t)(\log R)^{1/4+C_{h_{lo}}/4-C_{h_{up}}/8}$ and asymptotically bounded above by $C_{up}(t)(\log R)^{1/2+C_{h_{up}}/4}$ with high probability.

\end{abstract}
\vskip.5in %\newpage

\section{Introduction}
Suppose $\left\{\left(\Omega_{R}\,,\mathscr{F}_{R}\,,\P_{R}\right)\right\}_{R>0}$ is a collection of probability spaces. For each $R>0$, let $\E_{R}$ denote the expectation with respect to $\P_{R}$. For each $R>0$, let $\xi_{R}$ denote time-white space-colored noise on $S_{R}^{2} \times [0 \,, \infty)$, with $S_{R}^{2}$ being a sphere of radius $R$, defined on the probability space  $\left(\Omega_{R}\,,\mathscr{F}_{R}\,,\P_{R}\right)$.  The covariance structure of $\xi_{R}=\dot{W_{R}}$ is given by
\begin{equation}\label{covariancestructure}
\E_{R} \left[\dot{W_{R}}(t\,,x)\dot{W_{R}}(s\,,y)\right]=h_{R}(x,y)\delta_{0}(t-s)\,,
\end{equation}
where $h_{R}$ is a symmetric, semi-positive definite function on $S_{R}^{2} \times S_{R}^{2}$ and there exist some fixed constants $-2< C_{h_{up}}< 2$ and $\frac{1}{2}C_{h_{up}}-1 <C_{h_{lo}}\le C_{h_{up}}$ such that for all $R>0$ and $x\,,y \in S_{R}^{2}$, 
\begin{align*}
\displaystyle (\log R)^{C_{h_{lo}}/2}=h_{lo}(R)\leq h_{R}(x,y) \leq h_{up}(R)=(\log R)^{C_{h_{up}}/2}. 
\end{align*}
For $0<C_{\sigma_{lo}}<C_{\sigma_{up}}<\infty$, let $\sigma: \R \mapsto [ C_{\sigma_{lo}},  C_{\sigma_{up}} ]$ be Lipschitz continuous with the Lipschitz constant $0< \L_{\sigma} <\infty$.
Consider a collection of stochastic heat equations, each of which is defined on $[0,\infty) \times S_{R}^{2} \times \Omega_{R}$, 
\begin{equation}\label{SHES}
\partial_{t} u_{R}(t\,,x)=\tfrac12 \Delta_{S_{R}^{2}}u_{R}(t\,,x)+\sigma(u_{R}(t\,,x)) \xi_{R}(t\,,x),
\end{equation}
$0 \le t <\infty$, $x \in S_{R}^{2}$, subject to the initial value condition,
\[
	u_{R}(0\,,x)=u_{R,0}(x)\qquad\text{for all $x \in S_{R}^{2}$},
\]
where $\Delta_{S_{R}^{2}}$ is the Laplace-Beltrami operator on $S_{R}^{2}$ and the initial function $u_{R,0}(\cdot)$ is nonrandom and continuous. 
The mild solution to \eqref{SHES} is defined to be a process $u_{R}(\cdot\,,\cdot\,,\cdot):[0,\infty) \times S_{R}^{2} \times \Omega_{R} \mapsto \R$ which for each $0 \leq t<\infty$, $0<R<\infty$, $x \in S_{R}^{2}$, $\P_{R}$-almost surely satisifies the equation
\begin{equation}\label{mild}
u_{R}(t\,,x)=\int_{S_{R}^{2}}p_{R}(t\,,x\,,y)u_{R,0}(y) \d y+\int_{0}^{t}\int_{S_{R}^{2}}p_{R}(t-s\,,x\,,y)\sigma(u_{R}(s\,,y))W_{R}(\d s\,, \d y)\,,
\end{equation}
where $p_{R}$ is the heat kernel on $S_{R}^{2}$ and $\Omega_{R}$ is a probability space which depends on $R$.
\begin{remark}
Whenever it is clear from the context, we write $\Omega$ for $\Omega_{R}$, $\P$ for $\P_{R}$ and $\E$ for $\E_{R}$ for brevity. For example, we can rewrite \eqref{covariancestructure} as
\[\E \left[\dot{W_{R}}(t\,,x)\dot{W_{R}}(s\,,y)\right]=h_{R}(x,y)\delta_{0}(t-s)\,,\] whenever there is no confusion.
\end{remark}

The goal of this paper is to give an asymptotic estimate of $\sup_{x \in S_{R}^{2}}\vert u_{R}(t\,,x)\vert$ as $R \to \infty$. The following is the main theorem of this paper. 
\begin{theorem}\label{maintheorem}
If there exists a finite positive $U$ such that $\sup_{R>0}\sup_{x \in S_{R}^{2}}\vert u_{R,0}(x)\vert \le U$, then for any $0<t<\infty$, there exist constants $0<C_{low}(t) \le C_{up}(t) < \infty$,  which only depend on $t$, such that
\begin{equation}
\lim_{R \to \infty}\P\left( C_{low}(t)\left(\log R\right)^{\alpha_{l}} \le \sup_{x \in S_{R}^{2}}\vert u_{R}(t\,,x)\vert \le C_{up}(t)\left(\log R\right)^{\alpha_{u}}\right)=1,
\end{equation}
where $\alpha_{l} = 1/4+C_{h_{lo}}/4-C_{h_{up}}/8$ and $\alpha_{u} =1/2+C_{h_{up}}/4$.
\end{theorem}
The stochastic heat equation \eqref{SHES} provides a model of the heat flow on a large sphere. In this model, Theorem \ref{maintheorem} gives an estimate of the highest temperature on a large heated sphere. The result of this paper offers a potential explanation for the existence of solar flares on a large-sized star and estimates the temperatures of the solar flares relative to the radius of the star. While a majority of papers in the theory of SPDE  focus on SPDEs on Euclidean spaces, there are a smaller number of published works that study SPDEs on Riemannian manifolds. We find seven papers related to SPDEs on Riemannian manifolds:  Gy\" ongy \cite{Gyo1} \cite{Gyo2}, Funaki \cite{Funaki}, Lang, Schwab \cite{LangSchwab}, Dalang, L\'ev\^eque \cite{Dalang} \cite{Leveque} and Elliott, Hairer, Scott \cite{Eli}. These papers though focus on more general theories of SPDEs on spheres or Riemannian manifolds in general instead of investigating a specific quantative property of a SPDE such as giving an asymptotic estimate of the peaks of a SPDE, which is the main goal of our paper. 
\par
The challenge in finding an accurate asymptotic estimate on the peaks, as given in Theorem \ref{maintheorem}, is to unveil the effect of the curvature of a sphere on the heat flow on its surface under a noisy environment modeled by \eqref{SHES}. Unlike its Euclidean counterpart, the heat kernel of on a sphere does not have a compact form. The series expansion of the heat kernel on a Riemannian manifold is well-developed via the spectral theory of Laplace-Beltrami operator (See \cite{Rosenberg}). The technique to estimate of the maximal temperature of peaks, $\sup_{x \in S_{R}^{2}}\vert u_{R}(t\,,x)\vert$, relies on finding sufficiently-many ``independent'' points on a large sphere in the sense that heat flows originate from these points will not interact with each other in a short amount of time. This idea was introduced in \cite{ConusJosephDavar}. While there always exist sufficiently-many ``independent'' points in a Euclidean space as done in \cite{ConusJosephDavar}, cleverly fitting in these ``independent'' points on a sphere is the key to achieving the goal of this paper. This fitting requirement poses strong restrictions on the choices of various variables used to define an underlying coupling process. Successful coordination on the choice of these variables makes everything fall into the right place. In addtion to having to circumvent the ``dependence'' among points, we will need access to accurate estimations on the heat kernel on a sphere. Among various works on heat kernel estimations such as  Li, Yau \cite{LiYau}, Varadhan \cite{Varadhan}, and Molchanov \cite{Molchanov}, we will use Molchanov's result to prove the main theorem of this paper. Molchanov \cite{Molchanov} gives a uniform estimation on a compact subset  of the sphere excluding the South pole. 
\par
Before moving to the more technical details and the long series of calculations, an outline of our paper is given. This paper is organized as follows. In Section 2, we recall the Laplacian-Beltrami operator \cite{Rosenberg} and Molchanov's heat kernel estimates \cite{Molchanov}, and develop some preliminary estimates associated with the spherical heat kernels which will be frequently used throughout this paper. In Section 3, we show that the mild solution \eqref{mild} exists uniquely and prove that it is jointly measurable. In Section 4, we show that the mild solution has spatial continuity. In Section 5,  we follow the method in \cite{ConusJosephDavar} to give an asymptotic upper bound of the supremum of the mild solution by noting that there exist sufficiently many ``independent'' points on a sphere of large radius. In Section 6, necessary tail probability estimates are developed which will be used to give an asymptotic lower bound of the supremum of the mild solution. In Section 7, we use a discretization technique as in \cite{DalangMuellerZambotti} along with spatial continuity to give an asymptotic lower bound of the supremum of the mild solution, thus finishing the proof of the main Theorem \ref{maintheorem} of the paper. In the appendix, we follow the argument in \cite{DavarCBMS} to give the proof of the spherical version of Garsia's Lemma that is used in Section 4.
\par
Throughout this paper, the following notations will be used. Let $S^2$ denote $S_{1}^{2}$ the unit sphere, as usual. For each $k\,,R>0$, ``$\|\cdot\|_{k,R}$'' denotes the $\|\cdot\|_{L^{k}(\Omega_{R})}$-norm. Denote $x/R$ to be $\tilde{x}$ for each $x \in S_{R}^{2}$, $R>0$. When there is no confusion as to which probability space $\left(\Omega_{R}\,,\mathscr{F}_{R}\,,\P_{R}\right)$ is involved, we write $\|\cdot\|_{k}$ instead of $\|\cdot\|_{k,R}$ for brevity.  For real-valued functions $f$ and $g$, which are defined on $[0\,,\infty)$,
we write ``$f(t) \sim_{t} g(t)$'' to mean that there exist a constant $0<\epsilon_{0}<1$ such that
$1-\epsilon_{0}\le \liminf_{t \to 0} \vert f(t)/g(t)\vert \le \limsup_{t \to 0} \vert f(t)/g(t)\vert \le 1+\epsilon_{0}$. For real-valued functions $f$ and $g$, which are defined on $[M\,,\infty)$ for some finite positive $M$,
we write ``$f(R) \asymp_{R} g(R)$'' to mean that there exist constants $0<C_{1}\le C_{2}<\infty$ such that
$C_{1}\le \liminf_{R \to \infty} \vert f(R)/g(R)\vert \le \limsup_{R \to \infty} \vert f(R)/g(R)\vert \le C_{2}$.
\newpage

\section{The heat kernels on spheres and some preliminary estimates}
We use a similar but slightly different definition of the heat kernel than the definition in \cite{Rosenberg} (with the $\frac12$ in front of the Laplace-Beltrami operator).
\begin{definition}\label{defheatkernel}
The heat kernel on a Riemannian manifold $M$ is a function $p(t\,,x\,,y) \in C^{\infty}(R^{+}\times M \times M)$ such that
\begin{compactenum}
\item it satisfies the heat equation
\begin{equation}
\partial_{t}p(t\,,x\,,y)=\tfrac12 \Delta_{M,x} p(t\,,x\,,y),
\end{equation}
where $\Delta_{M,x}$ is the Laplace-Beltrami operator acting on $x$,
\item for every continuous function $f$ with compact support in $M$ and every $x \in M$,
\begin{equation}\label{Convergence}
 \lim_{t \to 0} \int_{M} p(t,x,y)f(y)\d y=f(x).
\end{equation}
\end{compactenum}
\end{definition}
\par
It is well known that $\Delta_{S_{R}^{2}}=R^{-2}\Delta_{S^2}$ \cite{Rosenberg} and that the spherical harmonics \\$\{Y_{lm}\}_{l=0, \cdots , \infty\,; -l \le m \le l}$ are eigenfunctions of $\Delta_{S^2}$ which form an orthonormal basis in $L^{2}(S^2)$ with the relations \cite{Rosenberg}
\[\Delta_{S^{2}}Y_{lm}=-l(l+1)Y_{lm},\]
for every $l \ge 0$ and $-l \le m \le l$.
Define the collection of functions $Y_{lm;R}(\cdot)=Y_{lm}(\cdot /R)$ on $S_{R}^{2}$ for every $l \ge 0$ and $-l \le m \le l$, then for all $x\in S_{R}^{2}$,
\begin{align*}
\Delta_{S_{R}^{2}}Y_{lm;R}(x)=-\frac{l(l+1)}{R^{2}}Y_{lm;R}(x).
\end{align*}
The orthogonality of $\{Y_{lm;R}\}_{l \ge 0\,,-l \le m \le l }$ inherits from that of $\{Y_{lm}\}_{l \ge 0\,,-l \le m \le l }$ 
and that for every $l \ge 0$ and $-l \le m \le l$, every $R>0$,
\begin{align*}
\frac{1}{R^{2}}\int_{S_{R}^{2}}\vert Y_{lm;R}(x) \vert^{2}\d x =1.
\end{align*}
Hence, for every $R>0$, $\{R^{-1}Y_{lm;R}\}_{l \ge 0\,,-l \le m \le l }$ form an orthonormal basis of $L^{2}(S_{R}^{2})$. By Proposition 3.1 in \cite{Rosenberg}, and Proposition 3.29 in \cite{Marinucci}, for every $t\,,R > 0$, $x\,,y \in S_{R}^{2}$,
\begin{align}\label{exact heat kernel on spheres}
p_{R}(t,x,y)&=\frac{1}{R^{2}}\sum_{l=0}^{\infty}\sum_{m=-l}^{l}\e^{-l(l+1)t/2R^2}Y_{lm;R}(x)\overline{Y_{lm;R}(y)}\nonumber\\
&=\frac{1}{R^{2}}\sum_{l=0}^{\infty}\sum_{m=-l}^{l}\e^{-l(l+1)t/2R^2}Y_{lm}(x/R)\overline{Y_{lm}(y/R)} ,                    
\end{align}
where \eqref{exact heat kernel on spheres} holds in the  sense of pointwise convergence and $L^{2}(S^{2}(R))$-convergence.
By the well-known summation formula of spherical harmonics \cite{Marinucci},
\begin{equation}\label{Spherical Legendre}
\sum_{m=-l}^{l}Y_{lm}(x)\overline{Y_{lm}(y)}=\frac{2l+1}{4\pi}P_{l}(x \cdot y) \qquad \text{for each $l \ge 0$ and any $x\,,y \in S^2$ },
\end{equation}
where $P_{l}$ denotes the $l-th$ Legendre polynomial and ``$\cdot$'' is the inner product for vectors, i.e., for every $x\,,y \in S^{2}$ whose Cartesian coordinates are given by $x=(x_{1}\,,x_{2}\,,x_{3})$ and $y=(y_{1}\,,y_{2}\,,y_{3})$ respectively, $x \cdot y=x_{1}y_{1}+x_{2}y_{2}+x_{3}y_{3}$.
Denote $d:S_{R}^2 \times S_{R}^{2} \mapsto [0,\infty)$ to be the geodesic distance on $S_{R}^{2}$ and $\theta(\cdot \,, \cdot)=d(\cdot \,, \cdot)/R$ the angle formed by two points on $S_{R}^{2}$. Then, by \eqref{exact heat kernel on spheres} and \eqref{Spherical Legendre}, 
\begin{align}\label{exact heat kernel on spheres version 2}
p_{R}(t\,,x\,,y)&:=p_{R}(t\,,\theta(x,y))\nonumber\\
&=\sum_{l=0}^{\infty}\frac{(2l+1)\e^{-l(l+1)t/2R^2}}{4\pi R^{2}}P_{l}(\cos \theta(x\,,y))\nonumber\\
&=\frac{1}{R^{2}}p_{1}(t/R^{2}\,,\theta(x,y)).
\end{align}
It has been proved in \cite{Molchanov} that for every $\theta_{0} \in (0,\pi)$,
\begin{equation}
p_{1}(t\,,x\,,y)=p_{1}(t,\theta(x,y)) \sim_{t} 
                         \frac{\e^{-\theta(x,y)^{2}/2t}}{2\pi t}\sqrt{\frac{\theta(x,y)}{\sin\theta(x,y)}}\,,
\end{equation}
uniformly for all $0 \le \theta(x,y) \le \theta_{0}$. This together with the scaling property \eqref{exact heat kernel on spheres version 2} gives the following.
\begin{lemma}\label{heat kernel estimate on spheres}
For every $t>0$, $0<\theta_{0}<\pi$, $0<\epsilon_{0}<1$, there exists $0<R_{mol}(t,\theta_{0},\epsilon_{0})<\infty$ such that for all $R>R_{mol}(t,\theta_{0},\epsilon_{0})$,
\begin{equation}
p_{R}(t,\theta(x,y)) = C(t/R^{2}\,, \theta(x,y))\frac{\e^{-R^{2}\theta(x,y)^{2}/2t}}{2 \pi t}\sqrt{\frac{\theta(x,y)}{\sin\theta(x,y)}}\,,
\end{equation}
where $1-\epsilon_{0} \le \inf_{0 \le \theta \le \theta_{0}}C(t/R^{2}\,, \theta) \le \sup_{0 \le \theta \le \theta_{0}}C(t/R^{2}\,, \theta) \le 1+\epsilon_{0}$.
\end{lemma}
The fact that the heat kernel is a transition density function gives 
\begin{lemma}\label{kernelaspdf}
For all $R,t>0$, $x \in S_{R}^{2}$,
\begin{equation}
\int_{S_{R}^{2}}p_{R}(t\,,\theta(x\,,y))\d y =1.
\end{equation}
\end{lemma}

%The following Varadhan's asymptotic relation holds uniformly for any pair of points on a complete Riemannian manifold. See Theorem 5.2.1 in \cite{Hsu} or \cite{Varadhan} for details.
%\begin{theorem}\label{Varadhan}
%Let $M$ be a complete Riemannian manifold and $p_{M}(t\,,x\,,y)$ the minimal heat kernel on $M$. Then uniformly on every compact subset of $M \times M$, we have 
%\begin{equation}
%\lim_{t \to 0}t\ln p_{M}(t\,,x\,,y) = -\frac{1}{2}d(x\,,y)^{2}.
%\end{equation}
%\end{theorem}
%The following Lemma is derived from Theorem \ref{Varadhan}.
%\begin{lemma}\label{uniformsphere}
%For every finite positive $t$, $0<\epsilon_{0}<1$, there exist a real-valued function $A(\cdot)$ and  a finite positive $R_{mol,A}(t\,,\epsilon_{0})$ such that for all $R>R_{mol,A}(t,\epsilon_{0})$,
%\begin{equation}
%p_{1}(s/R^{2}\,,x\,,y)=\e^{-\frac{A(s/R^{2})R^{2}\theta^{2}(x\,,y)}{2s}}\,,
%\end{equation}
%where 
%\begin{equation}
%1-\epsilon_{0} \le A(s/R^{2}) \le 1+\epsilon_{0}
%\end{equation}
%for all $x\,,y \in S^{2}$ and all $s \in (0,t)$.
%\end{lemma}
The following three quantities will be useful in the upcoming chapters. \\
For every nonnegative $\alpha$, $\beta$, $t,R>0$, let $B_{R}(x,\sqrt{\beta t})$ be the geodesic ball centered at $x$ with radius $\sqrt{\beta t}$ on $S_{R}^{2}$ and define
\begin{align}\label{fear}
f_{\e}(\alpha\,,R\,,t)=\int_{0}^{t}\d s \int_{S_{R}^{2} \times S_{R}^{2}}\e^{-2\alpha s}p_{R}(s\,, \theta(x\,,y_{1}))p_{R}(s\,, \theta(x\,,y_{2}))h_{R}(y_{1},y_{2})\d y_{1} \d y_{2}\,,
\end{align}
and
\begin{align}\label{febar}
&f_{\e,\beta}(\alpha\,,R\,,t)\nonumber\\
&=\int_{0}^{t}\d s \int_{B_{R}(x,\sqrt{\beta t}) \times B_{R}(x,\sqrt{\beta t})}\e^{-2\alpha s}p_{R}(s\,, \theta(x\,,y_{1}))
p_{R}(s\,, \theta(x\,,y_{2}))h_{R}(y_{1},y_{2})\d y_{1} \d y_{2}\,,
\end{align}
and
\begin{flalign}\label{tfebar}
&\tilde{f}_{\e,\beta}(\alpha\,,R\,,t) \\
&=\int_{0}^{t}\d s \int_{S_{R}^{2}\setminus B_{R}(x,\sqrt{\beta t}) \times S_{R}^{2}\setminus B_{R}(x,\sqrt{\beta t})}\e^{-2\alpha s}p_{R}(s\,, \theta(x\,,y_{1}))
p_{R}(s\,, \theta(x\,,y_{2}))h_{R}(y_{1},y_{2})\d y_{1} \d y_{2}\nonumber.
\end{flalign}
For notational convenience, denote $\displaystyle B_{R}(x,\sqrt{\beta t}) \times B_{R}(x,\sqrt{\beta t})$ by $\displaystyle T_{1}(\beta\,,x\,,R\,,t)$ and\\ $\displaystyle S_{R}^{2}\setminus B_{R}(x,\sqrt{\beta t}) \times S_{R}^{2}\setminus B_{R}(x,\sqrt{\beta t})$ by $\displaystyle T_{2}(\beta\,,x\,,R\,,t)$.  The following estimates will be used later.
\begin{lemma}\label{upperbound of fear}
For every $0<t\,,R\,,\alpha<\infty$, $f_{\e}(\alpha\,,R\,,t) \le (2 \alpha)^{-1}h_{up}(R)$.
\end{lemma}
\begin{proof}
By \eqref{fear} and Lemma \ref{kernelaspdf}, for every $0<t\,,R\,,\alpha<\infty$,
\begin{align*}
f_{\e}(\alpha\,,R\,,t)&\leq h_{up}(R)\int_{0}^{t}\d s \int_{S_{R}^{2} \times S_{R}^{2}}
\e^{-2\alpha s}p_{R}(s\,, \theta(x\,,y_{1}))p_{R}(s\,, \theta(x\,,y_{2}))\d y_{1} \d y_{2} \\
&= h_{up}(R)\int_{0}^{t}\e^{-2\alpha s}\d s \left(\int_{S_{R}^{2}}
p_{R}(s\,, \theta(x\,,y))\d y\right)^{2}\\
&\le \frac{h_{up}(R)}{2 \alpha}.
\end{align*}
\end{proof}

\begin{lemma}\label{upperbound of tfebar}
For every $0<t<\infty$, there exists a finite positive $R_{mol}(t)$ such that for $R\ge R_{mol}(t)$,  
\begin{equation}
\tilde{f}_{\e,\beta}(\alpha\,,R\,,t) \le 2h_{up}(R)t\e^{-2\sqrt{\alpha \beta t}}\,,
\end{equation}
provided that $\alpha \asymp_{R} \beta \asymp_{R} (\log R)^{c}$ where $0<c<1$ is a constant.
\end{lemma}
\begin{proof}
By checking the details in \cite{Molchanov}, for every $0<t<\infty$, $0<\theta_{0}<\pi$, there exist finite positive $\delta\,,c_{0}\,, R_{mol}(t,\delta,c_{0},\theta_{0})$ such that for all $R \ge R_{mol}(t,\delta,c_{0},\theta_{0})$ and $0<s<t$, 
\begin{equation}
\inf_{0\le \theta \le \theta_{0}}p_{1}(s/R^{2},\theta) \ge \left(1-\e^{-R^{2}\delta/s}\right)\left(1-c_{0}\sqrt{s}/R\right)\frac{\e^{-R^{2}\theta^{2}/2s}}{2\pi s}\sqrt{\theta/\sin\theta}.
\end{equation}
This together with \eqref{exact heat kernel on spheres version 2} and the elementary inequality $\sqrt{\theta\sin \theta}\ge \theta\sqrt{1-\theta^{2}/6}$ (for all $0 \le \theta \le \pi$) implies that for all finite positive $t,\beta$, there exists a finite positive $R_{mol}(t,\beta)$ such that for all finite positive $\alpha$ and $R\ge R_{mol}(t,\beta)$,
\begin{align*}
\tilde{f}_{\e,\beta}(\alpha\,,R\,,t)
%&=\int_{0}^{t}\d s \int_{T_{2}(\beta,x,R,t)}\e^{-2\alpha s}
%p_{R}(s\,, \theta(x\,,y_{1}))p_{R}(s\,, \theta(x\,,y_{2}))h_{R}(y_{1},y_{2})\d y_{1} \d y_{2}\\
%&\le h_{up}(R)\int_{0}^{t}\e^{-2\alpha s} \d s \left(1-\int_{B_{R}(x\,,\sqrt{\beta t})} p_{R}(s\,,\theta(x,y))\d y\right)^{2}\\
%&\le h_{up}(R)\int_{0}^{t}\e^{-2\alpha s} \d s \left(1-2\pi\int_{0}^{\sqrt{\beta t}/R} p_{1}(s/R^{2}\,,\theta(\tilde{x},y))\d y\right)^{2}\\
%&= h_{up}(R)\int_{0}^{t}\e^{-2\alpha s} \d s \left(1-2\pi\int_{0}^{\sqrt{\beta t}/R}R^{2}C(s/R^{2})\frac{\e^{-R^{2}\theta^{2}/2s}}{2\pi s}\sqrt{\theta \sin \theta}\d \theta\right)^{2}\\
%&\le h_{up}(R)\int_{0}^{t}\e^{-2\alpha s} \d s \left(\int_{S^{2}\setminus B(\tilde{x}\,,\sqrt{\beta t}/R)} p_{1}(\frac{s}{R^{2}}\,,\theta(\tilde{x},y))\d y\right)^{2}\\
%&= 4\pi^{2}h_{up}(R)\int_{0}^{t}\e^{-2\alpha s} \d s \left(\int_{\sqrt{\beta t}/R}^{\pi} \e^{-\frac{A(s/R^{2})R^{2}\theta^{2}}{2s}}\sin \theta \d \theta\right)^{2}\\
&\le h_{up}(R)\int_{0}^{t}\e^{-2\alpha s} \d s \Bigg(1-2\pi\int_{0}^{\sqrt{\beta t}/R}\left(1-\e^{-R^{2}\delta/s}\right)\left(1-c_{0}\sqrt{s}/R\right)\\ &\qquad \times R^{2}\theta\frac{\e^{-R^{2}\theta^{2}/2s}}{2\pi s}\sqrt{1-\theta^{2}/6}\d \theta\Bigg)^{2}\\
&\le h_{up}(R)\int_{0}^{t}\e^{-2\alpha s} \left(1-\left(1-\e^{-\frac{R^{2}\delta}{t}}\right)\left(1-\frac{c_{0}\sqrt{t}}{R}\right)\sqrt{1-\frac{\beta t}{6R^{2}}}\left(1-\e^{-\beta t/2s}\right)\right)\d s \\
&\le h_{up}(R)\Bigg(\int_{0}^{t}\e^{-2\alpha s-\beta t/ 2s}\d s+\int_{0}^{t}\e^{-2\alpha s}\left(1-\sqrt{1-\frac{\beta t}{6R^{2}}}\right)\d s\\ &\qquad +\int_{0}^{t}\e^{-2\alpha s}\left(\e^{-R^{2}\delta/t}+c_{0}\sqrt{t}/R\right)\d s\Bigg)\\
&\le h_{up}(R)t\e^{-2\sqrt{\alpha \beta t}}+\frac{h_{up}(R)\beta t}{12 \alpha R^{2}}+\frac{h_{up}(R)}{2\alpha}\left(\e^{-R^{2}\delta/t}+c_{0}\sqrt{t}/R\right).
\end{align*}
This implies for every $0<t<\infty$, there exists a finite positive $R_{mol}(t)$ such that for $R\ge R_{mol}(t)$,  
\begin{equation}
\tilde{f}_{\e,\beta}(\alpha\,,R\,,t) \le 2h_{up}(R)t\e^{-2\sqrt{\alpha \beta t}}\,,
\end{equation}
provided that $\alpha \asymp_{R} \beta \asymp_{R} (\log R)^{c}$ where $0<c<1$ is a constant.
\end{proof}

\begin{lemma}\label{lowerbound of feb0r}
For every $0<t\,,\beta<\infty$, $0<\epsilon_{0}<1$, there exists a finite positive $R_{mol}(t,\pi/4,\epsilon_{0})$ such that for all $R\ge\max\{R_{mol}(t,\pi/4,\epsilon_{0})\,,4\sqrt{\beta t}/\pi\}$, 
\begin{equation}
f_{\e,\beta}(0\,,R\,,t)\ge 2\pi^{2}th_{lo}(R)(1-\epsilon_{0})^{2}\left(1-\e^{-\beta/2}\right)^{2}.
\end{equation}
\end{lemma}
\begin{proof}
By \eqref{febar} and Lemma \ref{heat kernel estimate on spheres}, for every $0<t\,,\beta<\infty$, $0<\epsilon_{0}<1$, there exists a finite positive $R_{mol}(t,\pi/4,\epsilon_{0})$ such that for all $R\ge\max\{R_{mol}(t,\pi/4,\epsilon_{0})\,,4\sqrt{\beta t}/\pi\}$,
\begin{align*}
f_{\e,\beta}(0\,,R\,,t)
%&=\int_{0}^{t}\d s\int_{T_{1}(\beta,x,R,t)}
%p_{R}(s\,, \theta(x\,,y_{1}))p_{R}(s\,, \theta(x\,,y_{2}))h_{R}(y_{1},y_{2})\d y_{1} \d y_{2}\\
%&\geq \frac{h_{lo}(R)C_{l}^{2}(\frac{\pi}{4})}{4\pi^{2}}\int_{0}^{t}s^{-2}\d s  \int_{T_{1}(\beta,x,R,t)}
%\e^{-R^{2}(\theta(x,y_{1})^{2}+\theta(x,y_{2})^{2})/2s}
%\sqrt{\frac{\theta(x,y_{1})}{\sin \theta(x,y_{1})}}\sqrt{\frac{\theta(x,y_{2})}{\sin \theta(x,y_{2})}}\d y_{1} \d y_{2}\\
&\geq h_{lo}(R)(1-\epsilon_{0})^{2}\int_{0}^{t}s^{-2}\d s 
\left(\int_{B_{R}(x,\sqrt{\beta t})} \e^{-R^{2}\theta(x,y_{1})^{2}/2s}\sqrt{\frac{\theta(x,y_{1})}{\sin \theta(x,y_{1})}} \d y_{1}\right)^{2}\\
%&= \frac{h_{lo}(R)C_{l}^{2}(\frac{\pi}{4})R^{4}}{4\pi^{2}}\int_{0}^{t}s^{-2}\d s
 %\left(\int_{B(\tilde{x},\sqrt{\beta t}/R)} \e^{-R^{2}\theta(\tilde{x},y_{1})^{2}/2s}\sqrt{\frac{\theta(\tilde{x},y_{1})}{\sin %%\theta(\tilde{x},y_{1})}} \d y_{1}\right)^{2}\\
&=4\pi^{2}h_{lo}(R)(1-\epsilon_{0})^{2}R^{4}\int_{0}^{t}s^{-2}\d s
 \left(\int_{0}^{\sqrt{\beta t}/R} \e^{-R^{2}\theta^{2}/2s}\sqrt{\theta \sin \theta}\d\theta\right)^{2}\\
&\ge 2\pi^{2}h_{lo}(R)(1-\epsilon_{0})^{2}R^{4}\int_{0}^{t}s^{-2}\d s
\left(\int_{0}^{\sqrt{\beta t}/R}\theta \e^{-R^{2}\theta^{2}/2s}\d\theta\right)^{2}\\
%&=\frac{h_{lo}(R)C_{l}^{2}(\frac{\pi}{4})}{2}\int_{0}^{t}\left(1-\e^{-\frac{\beta t}{2s}}\right)^{2} \d s \\
%&>\frac{h_{lo}(R)C_{l}^{2}(\frac{\pi}{4})}{2}\int_{0}^{t}\left(1-\e^{-\frac{\beta}{2}}\right)^{2} \d s\\
&=2\pi^{2}th_{lo}(R)(1-\epsilon_{0})^{2}\left(1-\e^{-\frac{\beta}{2}}\right)^{2}\,,
\end{align*}
where in the second inequality, the assumption $R\ge 4\sqrt{\beta t}/\pi$ comes into play. It implies $\sqrt{\beta t}/R\le \pi/4$ and hence $\sqrt{\theta\sin\theta}\ge\sqrt{2}\theta/2$ for all $0<\theta \le \sqrt{\beta t}/R$.
\end{proof}

\section{Existence, Uniqueness, and measurability}
Following the development in the Section 1, we establish in this section the existence and uniqueness of the mild solution. Moreover, we apply Doob's separability theory \cite{Doob} to show that the mild solution is jointly measurable. This along with certain integrabiltiy conditions, justifies the application of Fubini's theorem whenever there presents measurability issues. We begin with the following crucial Existence and Uniqueness theorem.
\begin{theorem}\label{existenceuniqueness}
For every $0 < T\,,R < \infty$, and each $0 \le t \le T$, $x \in S_{R}^{2}$, the mild solution to Equation \eqref{mild} exists and is unique up to a modification independent of $t\,,x$.
\end{theorem}
\begin{proof}
Define the initial step of iteration to be
\begin{equation}
u_{R}^{(0)}(t\,,x)=u_{R}(0,x)=u_{R,0}(x)\,,
\end{equation}
and inductively define
\begin{equation}\label{iter}
u_{R}^{(n+1)}(t\,,x)=\int_{S_{R}^{2}}p_{R}(t\,,x\,,y)u_{R}^{(n)}(0\,,y)\d y+ \int_{0}^{t} \int_{S_{R}^{2}}p_{R}(t-s\,,x\,,y)\sigma(u_{R}^{(n)}(s\,,y))W(\d s\,, \d y).
\end{equation}
It is well-known that 
\begin{equation}\label{distributionsense}
\delta(1-x)=\tfrac12 \sum_{l=0}^{\infty}(2l+1)P_{l}(x)\,,
\end{equation}
for $-1\le x \le 1$, where the $P_{l}$'s are Legendre polynomials, $\delta$ denotes the Dirac-Delta function and \eqref{distributionsense} is understood in the sense of distribution. To be more specific, the sum in \eqref{distributionsense} converges to zero pointwisely for $-1 \le x <1$ and diverges to infinity for $x=1$. Moreover,
\begin{align*}
\lim_{n \to \infty}\int_{-1}^{1} \tfrac12 \sum_{l=0}^{n}(2l+1)P_{l}(x)f(x)\d x = f(1)
\end{align*}
for every continuous funtion $f$ defined on $[-1,1]$.
Taking $t = 0$ gives us
\begin{align}
u_{R}^{(n+1)}(0\,,x)&=\lim_{t \to 0}\int_{S_{R}^{2}}p_{R}(t\,,x\,,y)u_{R}^{(n)}(0\,,y)\d y\nonumber\\
&=\sum_{l=0}^{\infty}\int_{S_{R}^{2}}\frac{(2l+1)}{4 \pi R^{2}}P_{l}(\cos \theta (x\,,y))u_{R}^{(n)}(0\,,y)\d y\nonumber\\
&=\frac{1}{2\pi R^{2}} \int_{S_{R}^{2}}\delta(1-\cos \theta(x\,,y))u_{R}^{(n)}(0\,,y)\d y\nonumber\\
&= u_{R}^{(n)}(0\,,x).
\end{align}
By induction,
\begin{equation}\label{firstterm}
u_{R}^{(n)}(0\,,x)= u_{R\,,0}(x) \qquad \text{for all $n$}.
\end{equation}
From \eqref{iter},
\begin{flalign}\label{diffiter}
&u_{R}^{(n+1)}(t\,,x)-u_{R}^{(n)}(t\,,x)\nonumber\\
&\qquad\qquad\qquad= \int_{0}^{t} \int_{S_{R}^{2}}p_{R}(t-s\,,x\,,y)\left(\sigma(u_{R}^{(n)}(s\,,y))-\sigma(u_{R}^{(n-1)}(s\,,y))\right)W(\d s\,, \d y).
\end{flalign}
For notational brevity, denote for all $0<s<t<\infty$, $0<R<\infty$, positive integer $n$, $x\,,y_{1}\,,y_{2} \in S_{R}^{2}$,
\begin{align*}
V_{1}(t\,,s\,,R\,,n\,,x\,,y_{1}\,,y_{2})&=p_{R}(t-s\,,x\,,y_{1})p_{R}(t-s\,,x\,,y_{2})
\left(\sigma(u_{R}^{(n)}(s\,,y_{1}))-\sigma(u_{R}^{(n-1)}(s\,,y_{1}))\right)\nonumber\\
&\qquad \cdot \left(\sigma(u_{R}^{(n)}(s\,,y_{2}))-\sigma(u_{R}^{(n-1)}(s\,,y_{2}))\right)\,,
\end{align*}
and
\begin{align*}
V_{2}(t\,,s\,,R\,,n\,,x\,,y_{1}\,,y_{2})&=p_{R}(t-s\,,x\,,y_{1})p_{R}(t-s\,,x\,,y_{2})\\
&\qquad \cdot \e^{-2\alpha s}\Big\vert u_{R}^{(n)}(s\,,y_{1})-u_{R}^{(n-1)}(s\,,y_{1})\Big\vert \cdot \Big\vert u_{R}^{(n)}(s\,,y_{2})-u_{R}^{(n-1)}(s\,,y_{2})\Big\vert.
\end{align*}
By Carlen-Kr\'ee's bound \cite{Carleen} for Burkholder-Gundy-Davis inequality, and a similar argument in \cite{FoondunDavar}, and \eqref{firstterm}, we have for any $k \ge 2$, $0 \le t \le T<\infty$, $0<\alpha \,, R <\infty$ and $x \in S_{R}^{2}$ that
\begin{align}
&\e^{-\alpha t}\Big\|u_{R}^{(n+1)}(t\,,x)-u_{R}^{(n)}(t\,,x)\Big\|_{k}\nonumber\\
&=  \Big\|\e^{-\alpha t}\int_{0}^{t} \int_{S_{R}^{2}}p_{R}(t-s\,,x\,,y)
\left(\sigma(u_{R}^{(n)}(s\,,y))-\sigma(u_{R}^{(n-1)}(s\,,y))\right)W(\d s\,, \d y)\Big\|_{k}\nonumber\\
&\leq 2\sqrt{k}\Bigg\| \e^{-\alpha t}\sqrt{\int_{[0\,,t]\times S_{R}^{2} \times S_{R}^{2}}h_{R}(y_{1},y_{2})V_{1}(t\,,s\,,R\,,n\,,x\,,y_{1}\,,y_{2})\d s \d y_{1} \d y_{2}}\Bigg\|_{k}\nonumber\\
&\leq 2\L_{\sigma}\sqrt{k}\Bigg\| \sqrt{\int_{[0\,,t]\times S_{R}^{2} \times S_{R}^{2}} h_{R}(y_{1},y_{2})\e^{-2\alpha (t-s)}
V_{2}(t\,,s\,,R\,,n\,,x\,,y_{1}\,,y_{2}) \d s\d y_{1} \d y_{2}}\Bigg\|_{k}\nonumber\\
&\leq 2\L_{\sigma}\sqrt{kf_{\e}(\alpha\,,R\,,t)}
\sup_{0 \le t \le T}\sup_{x \in S_{R}^{2}}\e^{-\alpha t}\Big\|u_{R}^{(n)}(t\,,x)-u_{R}^{(n-1)}(t\,,x)\Big\|_{k}.
\end{align} 
Along Lemma \ref{upperbound of fear}, this implies for any $k \ge 2$, $0  \le T<\infty$, $0<\alpha \,, R <\infty$ and $x \in S_{R}^{2}$ that
\begin{flalign}\label{lastinequalityintheorem2.1}
&\sup_{0 \le t \le T}\sup_{x \in S_{R}^{2}}\e^{-\alpha t}\Big\|u_{R}^{(n+1)}(t\,,x)-u_{R}^{(n)}(t\,,x)\Big\|_{k} \nonumber\\ 
&\qquad\qquad\qquad\qquad\qquad  \le \frac{\L_{\sigma}\sqrt{2h_{up}(R)k}}{\sqrt{\alpha}} \sup_{0 \le t \le T}\sup_{x \in S_{R}^{2}}\e^{-\alpha t}\Big\|u_{R}^{(n)}(t\,,x)-u_{R}^{(n-1)}(t\,,x)\Big\|_{k}.
\end{flalign}
Define the norm $\|\cdot\|_{\alpha\,,k}$ for the collection of random fields on $[0\,,T] \times S_{R}^{2} \times \Omega_{R}$ by 
\begin{equation}
\|X\|_{\alpha\,,k}=\sup_{0 \le t \le T}\sup_{x \in S_{R}^{2}}\e^{-\alpha t}\Big\|X(t\,,x\,,\omega)\Big\|_{k}
\end{equation}
where $X:[0\,,T]\times S_{R}^{2} \times \Omega_{R} \mapsto \R$.\\
Choose $\alpha=k^{2} > \max\{4\,,2L_{\sigma}^{2}h_{up}(R)\}$ then $\frac{\L_{\sigma}\sqrt{2h_{up}(R)k}}{\sqrt{\alpha}}=\frac{\L_{\sigma}\sqrt{2h_{up}(R)}}{k}<1$. \\
The contraction mapping principle implies that $u_{R}^{(n)}(\cdot \,,\cdot)$ converges to a unique limit in $\|\cdot\|_{k^{2}\,,k}$-norm for $k > \max\{2\,,L_{\sigma}\sqrt{2h_{up}(R)}\}$. We denote this limit by $u_{R}(\cdot\,,\cdot)$.  By Markov's inequality, for each $0\le t\le T$, $x \in S_{R}^{2}$, $u_{R}(t\,,x)$ is the $\P_{R}$-limit of $u_{R}^{(n)}(t\,,x)$ and is hence $\P_{R}$-measurable. $u_{R}(t\,,x)$ is unique up to a modification independent of $t\,,x$ since if \\$\sup_{0 \le t \le T}\sup_{x \in S_{R}^{2}}\|u_{R}(t,x)-\tilde{u}_{R}(t,x)\|_{k^{2},k}=0$ then almost surely $u_{R}(t\,,x)=\tilde{u}_{R}(t\,,x)$ for all $0 \le t \le T$, $x \in S_{R}^{2}$, for every $0 < T\,,R < \infty$.
\end{proof}
Next, we want to show the joint measurability of the mild solution as mentioned at the beginning of this section. To do this we develope three lemmas, which state the mild solution is space-continuous and time-continuous in $L^{k}(\Omega)$ for each $k \ge 2$ and is a uniform limit in probability of its Picard iterations, independent of space and time.
\begin{lemma}\label{Lpspacecontinuous}
The solution is spatial-continuous in the $L^{k}$ sense. More  precisely, for any $k \ge 2$, any $0<t<\infty$, $0<\epsilon_{0}<1$ there exists a finite positive $R_{mol}(t,\epsilon_{0})$ such that for all $R\ge R_{mol}(t,\epsilon_{0})$, and any $x\,,x' \in S_{R}^{2}$ such that $\theta(x\,,x')<t^{3/2}(1+\epsilon_{0})^{-1}R^{-4}$,
\begin{equation}\label{diff6}
\Big\|u_{R}(t\,,x)-u_{R}(t\,,x')\Big\|_{k}^{k} \leq \left(4\sqrt{2}C_{\sigma_{up}}\sqrt{kh_{up}(R)}(1+\epsilon_{0})^{1/3}R^{4/3}\theta(x\,,x')^{1/3}\right)^{k}.
\end{equation}
\end{lemma}
\begin{proof}
Assume throughout the proof that $k \ge 2$. Denote for every positive integer $n$, $0<s\,,R<\infty$, $x\,,x'\,,y_{1}\,,y_{2} \in S_{R}^{2}$,
\begin{align*}
Q_{n,\sigma}(s,R,x,x',y_{1},y_{2})&=\big[p_{R}(s,\theta(x,y_{1}))\sigma(u^{(n)}(s\,,y_{1}))-p_{R}(s,\theta(x',y_{1}))\sigma(u^{(n)}(s\,,y_{1}))\big]\\
&\qquad \times \big[p_{R}(s,\theta(x,y_{2}))\sigma(u^{(n)}(s\,,y_{2}))-p_{R}(s,\theta(x',y_{2}))\sigma(u^{(n)}(s\,,y_{2}))\big]\,,
\end{align*}
and 
\begin{align*}
Q(s,R,x,x',y_{1},y_{2})=\big\vert p_{R}(s,\theta(x,y_{1}))-p_{R}(s,\theta(x',y_{1}))\big\vert \cdot \big\vert p_{R}(s,\theta(x,y_{2}))-p_{R}(s,\theta(x',y_{2}))\big\vert\,,
\end{align*}
for notational brevity.
By Carlen-Kr\'ee's optimal bound (\cite{Carleen}) on the Burkholder-Gundy-Davis inequality, for every $0<t\,,R<\infty$ and $x\,,x' \in S_{R}^{2}$,
\begin{align}\label{diff000}
	&\left\|u^{(n+1)}_{R}(t,x)-u^{(n+1)}_{R}(t,x')\right\|_{k}^{k}\nonumber\\
          %&=\Bigg\|\int_{0}^{t}\int_{S_{R}^{2}}p_{R}(t-s,\theta(x,y))\sigma(u^{(n)}(s\,,y))-p_{R}(t-s,\theta(x',y))%\sigma(u^{(n)}(s\,,y))W(\d s\,,\d y)\Bigg\|_{k}^{k}\nonumber\\
           &\qquad\qquad\qquad\leq (2\sqrt{k})^{k}\Bigg\|\sqrt{\int_{0}^{t}\d s\int_{S_{R}^{2}\times S_{R}^{2}}\d y_{1}\d y_{2}h_{R}(y_{1},y_{2})
	Q_{n,\sigma}(s,R,x,x',y_{1},y_{2})}\Bigg\|_{k}^{k}\nonumber\\
           &\qquad\qquad\qquad\leq (2\sqrt{k})^{k} \Big(h_{up}(R)C_{\sigma_{up}}^{2}\int_{0}^{t}ds\int_{S_{R}^{2}\times S_{R}^{2}}\d y_{1}\d                y_{2}Q(s,R,x,x',y_{1},y_{2})\Big)^{k/2}.
\end{align}
For every $0<\delta<t$, $0<R<\infty$, $x\,,x'\in S_{R}^{2}$,
\begin{align}\label{diff00}
&\int_{0}^{t}\int_{S_{R}^{2} \times S_{R}^{2}}Q(s,R,x,x',y_{1},y_{2})\d s \d y_{1}\d y_{2}\nonumber\\
&=\int_{0}^{\delta}\int_{S_{R}^{2} \times S_{R}^{2}}Q(s,R,x,x',y_{1},y_{2})\d s \d y_{1}\d y_{2}+\int_{\delta}^{t}\int_{S_{R}^{2} \times S_{R}^{2}}Q(s,R,x,x',y_{1},y_{2})\d s \d y_{1}\d y_{2}.
\end{align}
Since $p_{R}$ is a transition density function, for every $0<\delta<t$, $0<R<\infty$, $x\,,x'\in S_{R}^{2}$,
\begin{align}\label{diff0}
&\int_{0}^{\delta}\int_{S_{R}^{2} \times S_{R}^{2}}Q(s,R,x,x',y_{1},y_{2})\d s \d y_{1}\d y_{2}\nonumber\\
&\leq \int_{0}^{\delta}\int_{S_{R}^{2} \times S_{R}^{2}}\left(p_{R}(s,\theta(x,y_{1}))+p_{R}(s,\theta(x',y_{1}))\right)\left(p_{R}(s,\theta(x,y_{2}))+p_{R}(s,\theta(x',y_{2}))\right)\d s \d y_{1}\d y_{2}\nonumber\\
%&\leq 4\int_{0}^{\delta}\d s \Bigg(\int_{S_{R}^{2}}p_{R}(s\,,\theta(x\,,y))\d y\Bigg)^{2}\nonumber\\
&= 4\delta.
\end{align}
Denote for any $0<\delta<t<\infty$, $0<R<\infty$, $x\,,x'\,,y \in S_{R}^{2}$,
\begin{align}
&S(\delta,t,R,x,x',y)\nonumber\\
&=\sum_{l=1}^{\infty}\frac{2l+1}{2\pi l(l+1)}\left(\e^{-l(l+1)\delta/2R^{2}}-\e^{-l(l+1)t/2R^{2}}\right)
\big[P_{l}(\cos \theta(x,y))-P_{l}(\cos \theta(x',y))\big]\,,
\end{align}
for notational brevity. Then for any $0<\delta<t$, $0<R<\infty$, $x\,,x'\in S_{R}^{2}$,
\begin{flalign}\label{diff1}
\int_{\delta}^{t}\int_{S_{R}^{2} \times S_{R}^{2}}Q(s,R,x,x',y_{1},y_{2})\d s \d y_{1}\d y_{2}
=\int_{S_{R}^{2} \times S_{R}^{2}}S(\delta,t,R,x,x',y_{1})S(\delta,t,R,x,x',y_{2}) \d y_{1}\d y_{2}.
\end{flalign}
By uniform convergence and that $ \sup_{-1 \le a \le 1}\vert P'_{l}(a) \vert \le l(l+1)/2$, we have
\begin{align}
\Bigg\vert \sum_{l=1}^{\infty}\frac{2l+1}{2\pi l(l+1)}(\e^{-\frac{l(l+1)\delta}{2R^{2}}}-\e^{-\frac{l(l+1)t}{2R^{2}}})\left(P_{l}(a)-P_{l}(b)\right) \Bigg\vert 
%&\leq  \left(\sum_{l=1}^{\infty}\frac{2l+1}{4\pi}(\e^{-\frac{l(l+1)\delta}{2R^{2}}}-\e^{-\frac{l(l+1)t}{2R^{2}}})\right)\vert a-b \vert \nonumber\\
\leq  \left(\sum_{l=1}^{\infty}\frac{2l+1}{4\pi}\e^{-\frac{l(l+1)\delta}{2R^{2}}}\right)\vert a-b \vert \,,
\end{align}
for any $0<\delta<t$ and any  $-1 \leq a\,, b \leq 1$.\\
Note that $\sum_{l=1}^{\infty}\frac{2l+1}{4\pi}\e^{-\frac{l(l+1)\delta}{2R^{2}}}=p_{1}(\frac{\delta}{R^{2}}\,,z\,,z)$ for any $0<\delta,R<\infty$ and any $z \in S_{R}^{2}$. By Molchanov's heat kernel estimate (Lemma \ref{heat kernel estimate on spheres}), for every $0<\delta<t$,  $0<\epsilon_{0}<1$ there exists $0<R_{mol}(t,\epsilon_{0})<\infty$ such that for any $z \in S^{2}$ and any $R\ge R_{mol}(t,\epsilon_{0})$,
\begin{align*}
p_{1}(\delta/R^{2}\,,z\,,z) \leq \frac{(1+\epsilon_{0})R^{2}}{2\pi \delta}.
\end{align*} 
Hence, for any $-1 \leq a\,,b \leq 1$,  $0<\delta<t$, $0<\epsilon_{0}<1$ and $R\ge R_{mol}(t,\epsilon_{0})$,%(we will see this condition on $\delta$ can be met after a few lines in the upcoming \eqref{delta}), 
\begin{equation}\label{diff2}
\Bigg\vert \sum_{l=1}^{\infty}\frac{2l+1}{2\pi l(l+1)}(\e^{-\frac{l(l+1)\delta}{2R^2}}-\e^{-\frac{l(l+1)t}{2R^2}})\left(P_{l}(a)-P_{l}(b)\right)\Bigg\vert \le \frac{(1+\epsilon_{0})R^{2}}{2\pi \delta}\vert a-b\vert.
\end{equation}
Use \eqref{diff2}, the triangle inequality and the trignometric inequality $\vert \cos \alpha - \cos \beta \vert \le \vert \alpha -\beta\vert$  in \eqref{diff1} to get for any $0<\delta<t$, $0<\epsilon_{0}<1$ and $R>R_{mol}(t, \epsilon_{0})$, and any $x\,,x' \in S_{R}^{2}$, 
\begin{align}\label{diff3}
\int_{\delta}^{t}\int_{S_{R}^{2} \times S_{R}^{2}}Q(s,R,x,x',y_{1},y_{2})\d s \d y_{1}\d y_{2}
\leq \frac{4(1+\epsilon_{0})^{2}R^{8}}{\delta^{2}}\theta(x\,,x')^{2}.
\end{align}
Use \eqref{diff0} and \eqref{diff3} in \eqref{diff00} to get, for any $0<\delta<t$, $0<\epsilon_{0}<1$ and $R>R_{mol}(t,\epsilon_{0})$, and any $x\,,x' \in S_{R}^{2}$, 
\begin{align}\label{diff4}
\int_{0}^{t}\int_{S_{R}^{2} \times S_{R}^{2}}Q(s,R,x,x',y_{1},y_{2})\d s \d y_{1}\d y_{2}
\leq 4\delta + \frac{4(1+\epsilon_{0})^{2}R^{8}}{\delta^{2}}\theta(x\,,x')^{2}.
\end{align}
Take 
\begin{equation}\label{delta}
\delta = (1+\epsilon_{0})^{2/3}R^{8/3}\theta(x,x')^{2/3}.
\end{equation}\label{delta0}
Then for any $0<t<\infty$, $0<\epsilon_{0}<1$ and $R>R_{mol}(t,\epsilon_{0})$, and any $x\,,x' \in S_{R}^{2}$ such that 
\begin{equation}
\theta(x\,,x')<\frac{t^{3/2}}{(1+\epsilon_{0})R^{4}},
\end{equation}
we have $\delta<t$.
\eqref{diff2} and hence \eqref{diff4} can be applied to give that for any $0<t<\infty$, $0<\epsilon_{0}<1$ and $R>R_{mol}(t,\epsilon_{0})$, and any $x\,,x' \in S_{R}^{2}$ such that 
$\theta(x\,,x')<t^{3/2}(1+\epsilon_{0})^{-1}R^{-4}$,
\begin{align}\label{diff4.0}
\int_{0}^{t}\int_{S_{R}^{2} \times S_{R}^{2}}Q(s,R,x,x',y_{1},y_{2})\d s \d y_{1}\d y_{2}
\leq 8(1+\epsilon_{0})^{2/3}R^{8/3}\theta(x\,,x')^{2/3}.
\end{align}
Use \eqref{diff4.0} in \eqref{diff000} to get for any $0<t<\infty$, $0<\epsilon_{0}<1$, $R>R_{mol}(t,\epsilon_{0})$, and $x\,,x' \in S_{R}^{2}$ such that 
$\theta(x\,,x')<t^{3/2}(1+\epsilon_{0})^{-1}R^{-4}$,
\begin{align}\label{diff5}
\Big\| u^{(n+1)}_{R}(t\,,x)-u^{(n+1)}_{R}(t\,,x')\Big\|_{k}^{k}
\leq \left(4\sqrt{2}C_{\sigma_{up}}\sqrt{kh_{up}(R)}(1+\epsilon_{0})^{1/3}R^{4/3}\theta(x\,,x')^{1/3}\right)^{k}.
\end{align}
Let $n \to \infty$ to get for any $0<t<\infty$, $0<\epsilon_{0}<1$ and $R>R_{mol}(t,\epsilon_{0})$, and any $x\,,x' \in S_{R}^{2}$, such that $\theta(x\,,x')<t^{3/2}(1+\epsilon_{0})^{-1}R^{-4}$,
\begin{equation}\label{diff6}
\Big\|u_{R}(t\,,x)-u_{R}(t\,,x')\Big\|_{k}^{k} \leq \left(4\sqrt{2}C_{\sigma_{up}}\sqrt{kh_{up}(R)}(1+\epsilon_{0})^{1/3}R^{4/3}\theta(x\,,x')^{1/3}\right)^{k}.
\end{equation}
\end{proof}

\begin{lemma}\label{Lptimecontinuous}
The solution is time-continuous in $L^{k}$ sense. More precisely, for any $k \ge 2$, $0<t_{1}<t_{2}<\infty$, $R>0$, 
\begin{equation}
\sup_{x \in S_{R}^{2}}\Big\|u_{R}(t_{1}\,,x)-u_{R}(t_{2}\,,x)\Big\|_{k}^{k}\leq (2\sqrt{k})^{k}\left(h_{up}(R)C_{\sigma_{up}}^{2}(t_{2}-t_{1})\right)^{k/2}.
\end{equation}
\end{lemma}
\begin{proof}
By Carlen's optimal bound (\cite{Carleen}) on Burkholder-Gundy-Davis inquality and Lemma \ref{kernelaspdf}, for every $0<t_{1}<t_{2}<\infty$, $x\in S_{R}^{2}$,
\begin{align}\label{diff7}
&\Big\|u^{(n)}_{R}(t_{1}\,,x)-u^{(n)}_{R}(t_{2}\,,x)\Big\|_{k}^{k}\nonumber\\
&\hskip 0.4 in \qquad=\Big\|\int_{t_{1}}^{t_{2}} p_{R}\left(s\,,\theta(x,y)\right)\sigma\left(u^{(n)}_{R}(s\,,y)\right)\Big\|_{k}^{k}\nonumber \\
&\hskip 0.4 in \qquad\leq (2\sqrt{k})^{k}\left(h_{up}(R)C_{\sigma_{up}}^{2}\int_{t_{1}}^{t_{2}}\int_{S_{R}^{2}\times S_{R}^{2}}p_{R}(s\,,\theta(x,y_{1}))p_{R}(s\,,\theta(x,y_{1}))\d s \d y_{1} \d y_{2}\right)^{k/2}\nonumber\\
&\hskip 0.4 in \qquad \leq (2\sqrt{k})^{k}\left(h_{up}(R)C_{\sigma_{up}}^{2}(t_{2}-t_{1})\right)^{k/2}.
\end{align}
Let $n \to \infty$ to finish.
\end{proof}

\begin{lemma}\label{unifcontinuity}
For every $k \ge 2$, $0<T<\infty$,  $0<\theta_{0}<\pi$, there exists a finite positive $R_{mol}(T,\epsilon_{0})$ such that for all $R>R_{mol}(T,\epsilon_{0})$, there exists a full probability space $\Omega_{T,R}$ on which $\sup_{0 \le t \le T}\sup_{x \in S_{R}^{2}} \big\vert u_{R}^{(n)}(t\,,x)-u_{R}(t\,,x)\big\vert$ is $\P_{R}$-measurable. Moreover, for all nonnegative integer $n$ and $\sup_{0 \le t \le T}\sup_{x \in S_{R}^{2}} \big\vert u_{R}^{(n)}(t\,,x)-u_{R}(t\,,x)\big\vert$  converges to zero almost surely as $n \to \infty$.
\end{lemma}
\begin{proof}
For each positive integer $n$. Define 
\begin{equation}
T_{n}=\{T\cdot2^{-n} \,, 2T\cdot2^{-n} \,, 3T\cdot2^{-n}\,, \cdots \,, T\}\,,
\end{equation}
and
\begin{align}
&G_{R,n}=\Big\{x \in S_{R}^{2}: x=\big(R\sin(i_{1}\pi 4^{-n})\cos(2i_{2}\pi 4^{-(n+1)})\,,R\sin(i_{1}\pi 4^{-n})\sin(2i_{2}\pi 4^{-(n+1)})\,,\nonumber\\
&R\cos(i_{1}\pi 4^{-n})\big) \text{ for some $i_{1}\,,i_{2} \in \Z$}\Big\}.
\end{align}
By Doob's separability theory, Theorem 2.4  in \cite{Doob} specifically (since $[0\,,T] \times S_{R}^{2}$ can be parametrized by $t\,,\theta\,,\phi$ each of which is linear), for each $n$ there exists a version of $ u_{R}^{(n)}(t\,,x)-u_{R}(t\,,x)$ such that there exists a countable subset of $[0\,,T] \times S_{R}^{2}$, denoted by $D_{n}(T,R)$ such that $\sup_{0 \le t \le T}\sup_{x \in S_{R}^{2}} \big\vert u_{R}^{(n)}(t\,,x)-u_{R}(t\,,x)\big\vert = \sup_{(t,x) \in D_{n}(T,R)}\big\vert u_{R}^{(n)}(t\,,x)-u_{R}(t\,,x)\big\vert$ and hence
$\sup_{0 \le t \le T}\sup_{x \in S_{R}^{2}} \big\vert u_{R}^{(n)}(t\,,x)-u_{R}(t\,,x)\big\vert$ is measurable with respect to $P_{R}$.\\  By throwing away the bad sets for each $n$ where $\sup_{0 \le t \le T}\sup_{x \in S_{R}^{2}} \big\vert u_{R}^{(n)}(t\,,x)-u_{R}(t\,,x)\big\vert$ is non-measurable with respect to $P_{R}$, we get a full probability subset $\Omega_{T,R}$ of $\Omega$ on which $\sup_{0 \le t \le T}\sup_{x \in S_{R}^{2}} \big\vert u_{R}^{(n)}(t\,,x)-u_{R}(t\,,x)\big\vert$ is $\P_{R}$-measurable for each $n$. For the rest of the proof, we redefine for all $0 \le t\le T$, $R>0$ and $x \in S_{R}^{2}$,
\begin{equation}
u_{R}(t \,,x\,,\omega)= u_{R}(t \,,x\,,\omega)\1_{\{\omega \in \Omega_{T,R}\}}\,,
\end{equation}
and for each nonnegative $n$
\begin{equation}
u_{R}^{(n)}(t \,,x\,,\omega)= u_{R}^{(n)}(t \,,x\,,\omega)\1_{\{\omega \in \Omega_{T,R}\}}.
\end{equation}
For every $\epsilon >0$, $0<\theta_{0}<\pi$, $0<\epsilon_{0}<1$, $k \ge 2$, $0<T\,,R<\infty$, positive integer $n$, 
\begin{align}\label{ururn}
&\P\left(\sup_{0 \le t \le T}\sup_{x \in S_{R}^{2}} \Big\vert u_{R}(t\,,x)-u_{R}^{(n)}(t\,,x)\Big\vert > \epsilon\right)\nonumber\\
&\leq \epsilon^{-k} \E\left(\sup_{0 \le t \le T}\sup_{x \in S_{R}^{2}} \Big\vert u_{R}^{(n)}(t\,,x)-u_{R}(t\,,x)\Big\vert^{k}\right)\nonumber\\
&\leq \epsilon^{-k} 2^{n}4^{2n+1} 5^{k-1}\sup_{t \in T_{n}}\sup_{\vert t' -t\vert \le T\cdot 2^{-n}} \sup_{x \in G_{R,n}}\sup_{\theta(x'\,,x) \le \pi \cdot 4^{-n}}\Big(Mo_{1}+
Mo_{2}+Mo_{3}+
Mo_{4}+Mo_{5}\Big).
\end{align} 
where
\begin{align*}
&Mo_{1}=\E\left(\Big\vert u_{R}^{(n)}(t\,,x)-u_{R}^{(n)}(t\,,x')\Big\vert^{k}\right),
Mo_{2}=\E\left(\Big\vert u_{R}^{(n)}(t\,,x')-u_{R}^{(n)}(t'\,,x')\Big\vert^{k}\right),\\
&Mo_{3}=\E\left(\Big\vert u_{R}^{(n)}(t'\,,x')-u_{R}(t'\,,x')\Big\vert^{k}\right),
Mo_{4}=\E\left(\Big\vert u_{R}(t'\,,x')-u_{R}(t\,,x')\Big\vert^{k}\right),\\
&Mo_{5}=\E\left(\Big\vert u_{R}(t\,,x')-u_{R}(t\,,x)\Big\vert^{k}\right).
\end{align*}
Similar to \eqref{lastinequalityintheorem2.1} we will get for any $0<\alpha\,,T\,,R<\infty$, $k\ge 2$ and $m>n$,
\begin{align*}
&\sup_{0 \le t \le T} \sup_{x \in S_{R}^{2}}\e^{-\alpha t} \Big\| u_{R}^{(m+1)}(t\,,x)-u_{R}^{(n)}(t\,,x)\Big\|_{k}\\
&\qquad\qquad\qquad\qquad\qquad\qquad\qquad \le \left[\left(L_{\sigma}\sqrt{2h_{up}(R)k/\alpha}\right)^{m}+\cdots+\left(L_{\sigma}\sqrt{2h_{up}(R)k/\alpha}\right)^{n}\right]\\
&\qquad\qquad\qquad\qquad\qquad\qquad\qquad\qquad \times \sup_{0 \le t \le T} \sup_{x \in S_{R}^{2}}\e^{-\alpha t} \Big\| u_{R}^{(1)}(t\,,x)-u_{R}^{(0)}(t\,,x)\Big\|_{k}.
\end{align*}
Let $m \to \infty$ to get for any $0<\alpha<\infty$,
\begin{align*}
&\sup_{0 \le t \le T} \sup_{x \in S_{R}^{2}}\e^{-\alpha t} \Big\| u_{R}(t\,,x)-u_{R}^{(n)}(t\,,x)\Big\|_{k} \\
&\qquad\qquad\qquad\qquad\qquad\leq \frac{\left(L_{\sigma}\sqrt{2h_{up}(R)k/\alpha}\right)^{n}}{1-\left(L_{\sigma}\sqrt{2h_{up}(R)k/\alpha}\right)}\sup_{0 <t \le T} \sup_{x \in S_{R}^{2}}\e^{-\alpha t} \Big\| u_{R}^{(1)}(t\,,x)-u_{R}^{(0)}(t\,,x)\Big\|_{k}.
\end{align*}
Choose $\alpha=8\L_{\sigma}^{2}h_{up}(R)k$ to get for every $0\le t \le T <\infty$, $0<R<\infty$, $k\ge 2$ and $x\in S_{R}^{2}$,
\begin{equation}
\Big\| u_{R}(t\,,x)-u_{R}^{(n)}(t\,,x)\Big\|_{k}^{k} \leq \e^{8\L_{\sigma}^{2}h_{up}(R)k^{2}T}2^{-(n-1)k}\sup_{0 \le t\le T}\sup_{x \in S_{R}^{2}}\Big\| u_{R}^{(1)}(t\,,x)-u_{R}^{(0)}(t\,,x)\Big\|_{k}^{k}.
\end{equation}
By taking the supremum on the left, we have for $0<R<\infty$, $k\ge 2$ and $x\in S_{R}^{2}$,
\begin{align*}
&\sup_{0\le t\le T}\sup_{x \in S_{R}^{2}} \Big\| u_{R}(t\,,x)-u_{R}^{(n)}(t\,,x)\Big\|_{k}^{k} \leq \e^{8\L_{\sigma}^{2}h_{up}(R)k^{2}T}2^{-(n-1)k}\sup_{0\le t \le T}\sup_{x \in S_{R}^{2}}\Big\| u_{R}^{(1)}(t\,,x)-u_{R}^{(0)}(t\,,x)\Big\|_{k}^{k}.
\end{align*}
Use the above upper bound, Lemma \ref{Lpspacecontinuous} and  \eqref{diff5}, Lemma \ref{Lptimecontinuous} and \eqref{diff7} in  \eqref{ururn} to get for any fixed $k\ge 2$, and every $0<T<\infty$, $0<\epsilon_{0}<1$, there exists a finite positive $R_{mol}(T,\epsilon_{0})$ such that for all $R\ge R_{mol}(T,\epsilon_{0})$ and any $n$ such that $\pi 4^{-n}<(2^{-n}T)^{3/2}(1+\epsilon_{0})^{-1}R^{-4}$ (so Lemma \ref{Lpspacecontinuous} and  \eqref{diff5} can be applied),
\begin{align*}
&\P\left(\sup_{0\le t \le T}\sup_{x \in S_{R}^{2}} \Big\vert u_{R}(t\,,x)-u_{R}^{(n)}(t\,,x)\Big\vert > \epsilon\right)\\
&\leq \epsilon^{-k} 2^{5n+2} 5^{k-1}\Bigg(\e^{8\L_{\sigma}^{2}h_{up}(R)k^{2}T}2^{-(n-1)k}\sup_{0\le t \le T}\sup_{x \in S_{R}^{2}}\Big\| u_{R}^{(1)}(t\,,x)-u_{R}^{(0)}(t\,,x)\Big\|_{k}^{k}\\
&+2\left(4\sqrt{2}C_{\sigma_{up}}\sqrt{k h_{up}(R)}(1+\epsilon_{0})^{1/3}R^{4/3} (\pi \cdot 4^{-n})^{1/3} \right)^{k}+2(2\sqrt{k})^{k}\left(h_{up}(R)C_{\sigma_{up}}^{2}2^{-n}T\right)^{k/2}\Bigg).
\end{align*}
\par
Choose $\epsilon=2^{-n/4}$ to get for any fixed $k\ge 2$, and every $0<T<\infty$, $0<\epsilon_{0}<1$, there exists a finite positive $R_{mol}(T,\epsilon_{0})$ such that for all $R\ge R_{mol}(T,\epsilon_{0})$ and any $n$ such that $\pi 4^{-n}<(2^{-n}T)^{3/2}(1+\epsilon_{0})^{-1}R^{-4}$,
\begin{align*}
&\P\left(\sup_{0\le t \le T}\sup_{x \in S_{R}^{2}} \Big\vert u_{R}(t\,,x)-u_{R}^{(n)}(t\,,x)\Big\vert > 2^{-n/4}\right)\\
&\leq  2^{(5+k/4)n+2} 5^{k-1}\Bigg(\e^{8\L_{\sigma}^{2}h_{up}(R)k^{2}T}2^{-(n-1)k}\sup_{0\le t \le T}\sup_{x \in S_{R}^{2}}\Big\| u_{R}^{(1)}(t\,,x)-u_{R}^{(0)}(t\,,x)\Big\|_{k}^{k}\\
& +2\left(4\sqrt{2}C_{\sigma_{up}}\sqrt{k h_{up}(R)}(1+\epsilon_{0})^{1/3}R^{4/3} (\pi \cdot 4^{-n})^{1/3} \right)^{k}+2(2\sqrt{k})^{k}\left(h_{up}(R)C_{\sigma_{up}}^{2}2^{-n}T\right)^{k/2}\Bigg).
\end{align*}
Choose any $k>20$ so $2^{(5+k/4)n+2}\cdot 2^{-(n-1)k}=2^{(5-\frac{3k}{4})n+k+2}$ and $2^{(5+k/4)n+2}\cdot 2^{-2nk/3}=2^{(5-\frac{5k}{12})n+2}$ and $2^{(5+k/4)n+2}\cdot 2^{-nk/2}=2^{(5-\frac{k}{4})n+2}$ all decay exponentially fast as $n \to \infty$. Hence,
\begin{align*}
\sum_{n=1}^{\infty}\P\left(\sup_{0\le t \le T}\sup_{x \in S_{R}^{2}} \Big\vert u_{R}(t\,,x)-u_{R}^{(n)}(t\,,x)\Big\vert > 2^{-n/4}\right)<\infty.
\end{align*}
Borel-Cantelli's lemma implies there almost surely exists a finite $N(\omega)$ such that for $n \ge N(\omega)$,
\begin{equation}
\sup_{0 \le t \le T}\sup_{x \in S_{R}^{2}} \Big\vert u_{R}(t\,,x)-u_{R}^{(n)}(t\,,x)\Big\vert \le 2^{-n/4}.
\end{equation}
\end{proof}
We are now ready to show that the mild solution is jointly measurable.
\begin{theorem}
For every $0<T\,,R<\infty$, there is a version of the mild solution such that $u_{R}(\cdot\,,\cdot \,, \cdot): [0\,,T] \times S_{R}^{2} \times \Omega \mapsto \R$ is measurable.
\end{theorem}
\begin{proof}
First, make the modification that for all $0 \le t\le T$ and $x \in S_{R}^{2}$
\begin{align*}
u_{R}(t \,,x\,,\omega)= u_{R}(t \,,x\,,\omega)\1_{\{\omega \in \Omega_{T,R}\}} \,,
\end{align*}
and for each nonnegative $n$
\begin{align*}
u_{R}^{(n)}(t \,,x\,,\omega)= u_{R}^{(n)}(t \,,x\,,\omega)\1_{\{\omega \in \Omega_{T,R}\}}\,,
\end{align*}
where $\Omega_{T,R}$ is given in Lemma \ref{unifcontinuity}.
For notational brevity, define for each real number $\alpha$ random sets
\begin{align*}
M_{1}(\alpha)&=\Big\{t \in [0,T] \,,x \in S_{R}^{2}\,,\omega \in \Omega \Big\vert u_{R}(t\,,x\,,\omega) \ge \alpha \Big\}\\
&\qquad \bigcap \Big\{t \in [0,T] \,,x \in S_{R}^{2}\,,\omega \in \Omega \Big\vert  u_{R}(t\,,x\,,\omega)= \limsup_{n \to \infty}u_{R}^{(n)}(t\,,x\,,\omega)\Big\}\,,
\end{align*}
and
\begin{align*}
M_{2}(\alpha)&=\Big\{t \in [0,T] \,,x \in S_{R}^{2}\,,\omega \in \Omega \Big\vert u_{R}(t\,,x\,,\omega) \ge \alpha \Big\}\\
&\qquad \bigcap \Big\{t \in [0,T] \,,x \in S_{R}^{2}\,,\omega \in \Omega \Big\vert  u_{R}(t\,,x\,,\omega)\ne \limsup_{n \to \infty}u_{R}^{(n)}(t\,,x\,,\omega)\Big\}.
\end{align*}
Then
\begin{align*}
M_{1}(\alpha) \bigcap M_{2}(\alpha) =\Big\{t \in [0,T] \,,x \in S_{R}^{2}\,,\omega \in \Omega \Big\vert u_{R}(t\,,x\,,\omega) \ge \alpha \Big\}.
\end{align*}
It suffices to show for each real number $\alpha$, $M_{1}(\alpha)$
is measurable with respect to the product measure in the measurable space $[0\,,T] \times S_{R}^{2} \times \Omega$ since $M_{2}(\alpha)$
has product measure zero by Lemma \ref{unifcontinuity}. Note that
\begin{align*}
M_{1}(\alpha)&=\Big\{t \in [0,T] \,,x \in S_{R}^{2}\,,\omega \in \Omega \Big\vert \limsup_{n \to \infty}u_{R}^{(n)}(t\,,x\,,\omega) \ge \alpha \Big\}\\
&\qquad \bigcap \Big\{t \in [0,T] \,,x \in S_{R}^{2}\,,\omega \in \Omega \Big\vert  u_{R}(t\,,x\,,\omega)= \limsup_{n \to \infty}u_{R}^{(n)}(t\,,x\,,\omega)\Big\}.
\end{align*}
The set
\begin{align*}
\Big\{t \in [0,T] \,,x \in S_{R}^{2}\,,\omega \in \Omega \Big\vert  u_{R}(t\,,x\,,\omega)= \limsup_{n \to \infty}u_{R}^{(n)}(t\,,x\,,\omega)\Big\}
\end{align*}
has full product measure. Moreover, 
\begin{align*}
&\Big\{t \in [0,T] \,,x \in S_{R}^{2}\,,\omega \in \Omega \Big\vert \limsup_{n \to \infty}u_{R}^{(n)}(t\,,x\,,\omega) \ge \alpha \Big\}\\
%&=\bigcap_{k=1}^{\infty}\Big\{t \in [0,T] \,,x \in S_{R}^{2}\,,\omega \in \Omega \Big\vert \sup_{n \ge k}u_{R}^{(n)}(t\,,x\,,\omega) \ge \alpha-\frac{1}{k} \Big\}\\
&=\bigcap_{k=1}^{\infty}\bigcup_{n \ge k}\Big\{t \in [0,T] \,,x \in S_{R}^{2}\,,\omega \in \Omega \Big\vert u_{R}^{(n)}(t\,,x\,,\omega) \ge \alpha-\frac{1}{k} \Big\}
\end{align*}
is jointly measurable since each of $u_{R}^{(n)}$ is by iteration.
This finishes the proof.
\end{proof}
For the rest of the paper, we will use the time-space-probability jointly measurable version of the mild solution without stating this hidden information explicitly.
\section{Spatial Continuity}
In this section, we apply a version of Kolmogorov's continuity theorem to show the mild solution is spatial-continuous almost surely. Time continuity can also be obtained by a similar method but we do not prove it in the paper since time continuity is not used to prove our main results. 
\newline We follow the developments in \cite{DavarCBMS} to prove a spherical version of Kolmogorov's continuity theorem by setting up the Garsia's theorem. Since we are working on spheres, some necessary arguments for the sphrical versions of Garsia's theorems and Kolmogorov's continuity theorem will be given, which will be similar to the arguments in \cite{DavarCBMS}. Further details are given in the appendix. \\
We begin by setting up some necessary notations and terminologies.  Suppose $\big\{ \mu_{k} \big\}_{k \ge 2}$ is a sequence of subadditive measure. 
Fix $k \ge 2$ and $r_{0}(k)=1$, define iteratively,
\begin{equation}\label{rn}
r_{n+1}(k) = \sup \Big\{r>0: \mu_{k}(r)=\frac{1}{2}\mu_{k}(r_{n}(k))\Big\}.
\end{equation}
Define for every $x \in S_{R}^{2}$,
\begin{equation}\label{fnk}
\bar{f}_{n,k}(x)=\frac{1}{\vert B_{R}(x,r_{n}(k))\vert}\int_{B_{R}(x,r_{n}(k))}f(z)\d z\,,
\end{equation}
and
\begin{equation}
C_{\mu_{k}} = \sup_{r >0} \frac{\mu_{k}(2r)}{\mu_{k}(r)}\,,
\end{equation}
where $B_{R}(x\,,r)$ is the geodesic ball centered at $x$ with radius $r$ in $S_{R}^{2}$ and $\vert \cdot \vert$ denotes the surface measure. 
For notational convenience, denote
\begin{equation}
B_{R}(r)=B_{R}(N\,,r)\,,
\end{equation}
where $N$ is the North Pole of $S_{R}^{2}$.\\
Define the operator $\tilde{+}$ on spheres by assigning for any $x\,,z \in S_{R}^{2}$ the point $ x \tilde{+} z$ to be the isotropic image of $z$ by rotating $S_{R}^{2}$ along the great circle which contains $x\,,z\,,N$ from $x$ to $N$ if $x \ne N$ and $x \tilde{+} z =z$ if $x =N$. We can rewrite $\bar{f}_{n,k}(x)$ as
\begin{equation}
\bar{f}_{n,k}(x)=\frac{1}{\vert B_{R}(r_{n}(k))\vert}\int_{B_{R}(r_{n}(k))}f(x \tilde{+} z)\d z.
\end{equation}
Define Garsia's integral
\begin{equation}\label{ik}
I_{k}=\int_{S_{R}^{2}}\d x \int_{S_{R}^{2}} \d y \Bigg\vert \frac{f(x)-f(y)}{\mu_{k}(R\theta(x\,,y))}\Bigg\vert^{k}.
\end{equation}
The spherical version of Kolmogorov's continuity theorem is based on the following lemma and theorem of Garsia's theory. The proofs on the results of Garsia's theory are omitted in this section and are given in the appendix.
\newpage
\begin{lemma}\label{GarsiaLemma1.1}
Suppose $f \in L^{1}(S_{R}^{2})$, $\bar{f}_{n,k}$ is defined as in \eqref{fnk}, $I_{k}$ is defined as in \eqref{ik} and there exists $1\leq k < \infty$ such that
\begin{enumerate}
\item $I_{k} <\infty$ and 
\item $\int_{0}^{1}\vert B_{R}(r)\vert^{-2/k} \d \mu_{k}(r)<\infty$.
\end{enumerate}
Then $\bar{f}_{k}=\lim_{n \to \infty}\bar{f}_{n,k}$ exists and for each integer $l \ge 0$,
\begin{equation}\label{diffavg}
\sup_{x \in S_{R}^{2}}\vert \bar{f}_{k}(x) - \bar{f}_{l,k}(x) \vert \leq 4C_{\mu_{k}} I_{k}^{1/k} \int_{0}^{r_{l+1}(k)} \vert B_{R}(r)\vert^{-2/k}\d \mu_{k}(r).
\end{equation}
Moreover, $\bar{f}_{k}=f \text{a.e.}$ .
\end{lemma}
\begin{theorem}\label{GarsiaTheorem1.2}
Suppose $f \in L^{1}(S_{R}^{2})$ and $\bar{f}_{k}$ is defined for some $1\leq k <\infty$ as in Lemma \ref{GarsiaLemma1.1}. Then for all $x\,,x' \in S_{R}^{2}$ such that $R\theta(x\,,x')\leq 1$,
\begin{align}
\vert \bar{f}_{k}(x)-\bar{f}_{k}(x') \vert \leq 4C_{\mu_{k}}(2+C_{\mu_{k}})I_{k}^{1/k} \int_{0}^{R\theta(x,x')}\vert B_{R}(r)\vert^{-2/k}\d \mu_{k}(r).
\end{align}
\end{theorem}
Our spherical version of Kolmogorov's theorem will be a consequence of the following two theorems. They will be of use again later when we give an asymptotic upper bound of $\sup_{x \in S_{R}^{2}}\vert u_{R}(t\,,x)\vert$ as $R \to \infty$.
\begin{theorem}\label{Holder Continuity}
For every $0<t<\infty$, $0<\epsilon_{0}<1$, $0<a<2$, $0<q<1/3$, there exist finite positive $K(a,q)$ and $R_{mol}(t,\epsilon_{0})$ such that for all $k \ge \max\{2\,,K(a,q)\}$, $R\ge \max \Big\{\left(3(2\pi)^{-1}(1+\epsilon_{0})^{-1}t^{3/2}\right)^{1/4}\,,R_{mol}(t,\epsilon_{0})\Big\}$ and $n$ such that $\pi R 2^{-n} \leq 1$ and $\pi 2^{-n}<t^{3/2}(1+\epsilon_{0})^{-1}R^{-4}$,
\begin{align}
&\E\left[\sup_{0<\theta(x,x') \leq \pi 2^{-n}}\Bigg\vert \frac{u_{R}(t\,, x)-u_{R}(t\,, x')}{\left(R\theta(x,x')\right)^{q}}\Bigg\vert^{k}\right] \nonumber\\
&\leq \pi^{a-4}\left(12288\sqrt{2h_{up}(R)}C_{\sigma_{up}}(2-a)^{-1}(1+\epsilon_{0})^{1/3}\pi^{13/3-a-q}2^{q}R^{1/3-q}\right)^{k}
k^{k/2}.
\end{align}
\end{theorem}
\begin{theorem}\label{probabilityofsmalldifference}
For every $0<t<\infty$, $0<\epsilon_{0}<1$, $0<a<2$, $0<q<1/3$, there exist finite positive $K(a,q)$ and $R_{mol}(t,\epsilon_{0})$ such that for all $R\ge \max \Big\{\left(\frac{3t^{3/2}}{2\pi(1+\epsilon_{0})}\right)^{1/4}\,,R_{mol}(t,\epsilon_{0})\Big\}$, $n \ge \max\Big\{2\,,K(a,q)\,,\log_{2}\left(\pi R\right)\,,\big\lfloor \log_{2}\left(\pi (1+\epsilon_{0})R^{4}\right)-3(\log_{2}t)/2 \big\rfloor +1\Big\}$, $0<\gamma<\infty$,
\begin{align}
&\P\left( \sup_{\theta(x,x') \le \pi 2^{-n}}\vert u_{R}(t\,, x)-u_{R}(t\,, x')\vert \ge  \pi R2^{-n\gamma}\right)\nonumber\\
&\leq \pi^{a-4}\left(12288\sqrt{2h_{up}(R)}C_{\sigma_{up}}(2-a)^{-1}(1+\epsilon_{0})^{1/3}\pi^{10/3-a}2^{q}R^{-2/3}\right)^{n}n^{n/2}2^{(\gamma-q)n^{2}}.
\end{align}
\end{theorem}
We postpone the proofs of the above two theorems but state and prove the spatial continuity theorem.
\begin{theorem}\label{spacecontinuityfixR}
For every $0<t<\infty$, there exists a finite positive $R(t)$ such that for all $R\ge R(t)$ and $0<\gamma<1/3$, there exists a finite positive $n(R,t,\gamma) > \max\Big\{2\,,K(\gamma)\,,\log_{2}\left(\pi R\right)\,, \\ \log_{2}\left(3\pi R^{4}/2\right)-3(\log_{2}t)/2 \Big\}$ where $K(\gamma)$ is a finite positive number such that for all positive integer $n \ge n(R,t,\gamma)$, 
\begin{equation}
\sup_{\theta(x,x') \leq \pi 2^{-n}}\vert u_{R}(t\,, x)-u_{R}(t\,, x')\vert \le \pi R2^{-\gamma n}.
\end{equation}
Moreover, $n(R,t,\gamma)$ is increasing in $R$.
\end{theorem}
\begin{proof}[Proof of Theorem \ref{spacecontinuityfixR}]
By Theorem \ref{probabilityofsmalldifference}, For every $0<t<\infty$, $0<\gamma<\frac{1}{3}$, there exist finite positive $K(1,\frac{\gamma}{2}+\frac{1}{6})$ and $R_{mol}(t,1/2)$ such that for all $R\ge \max \Big\{\left(t^{3/2}/\pi\right)^{1/4}\,,R_{mol}(t,1/2)\Big\}$, $n \ge n_{0}:=\max\Big\{2\,,K(1,\frac{\gamma}{2}+\frac{1}{6})\,,\log_{2}\left(\pi R\right)\,,\lfloor \log_{2}\left(3\pi R^{4}/2\right)-\frac{3}{2}\log_{2}t \rfloor +1\Big\}$,
\begin{align}
&\sum_{n=n_{0}}^{\infty}\P\left( \sup_{\theta(x,x') \leq \pi 2^{-n}}\vert u_{R}(t\,, x)-u_{R}(t\,, x')\vert > \pi R2^{-\gamma n}\right) \nonumber\\
&\leq \sum_{n=n_{0}}^{\infty}\pi^{-3}\left(12288\sqrt{2h_{up}(R)}C_{\sigma_{up}}(3/2)^{1/3}\pi^{7/3}2^{\frac{\gamma}{2}+\frac{1}{6}}R^{-2/3}\right)^{n}n^{n/2}2^{(\frac{\gamma}{2}-\frac{1}{6})n^{2}}.
\end{align}
By the Borel-Cantelli Lemma, almost surely there exists a random finite $N$, $N \ge n_{0}$ such that for $n \ge N$,
\begin{equation}
\sup_{\theta(x,x') \leq \pi 2^{-n}}\vert u_{R}(t\,, x)-u_{R}(t\,, x')\vert \le \pi R2^{-\gamma n}.
\end{equation}
\end{proof}
We finish this section by demonstrating the proof of Theorem \ref{Holder Continuity} and Theorem \ref{probabilityofsmalldifference}.
\begin{proof}[Proof of Theorem \ref{Holder Continuity}]
Recall Lemma \ref{Lpspacecontinuous} which states that for any $k \ge 2$, $0<t<\infty$, $0<\epsilon_{0}<1$, there exists $0<R_{mol}(t,\epsilon_{0})<\infty$ such that for all $R>R_{mol}(t,\epsilon_{0})$ and $x\,,x' \in S_{R}^{2}$ such that $\theta(x\,,x')<\frac{t^{3/2}}{(1+\epsilon_{0})R^{4}}$, 
\begin{equation}\label{diff6}
\Big\|u_{R}(t\,,x)-u_{R}(t\,,x')\Big\|_{k}^{k} \leq \left(4\sqrt{2}C_{\sigma_{up}}\sqrt{kh_{up}(R)}(1+\epsilon_{0})^{1/3}R^{4/3}\theta(x\,,x')^{1/3}\right)^{k}.
\end{equation}
%From now on, we are working under the assumption that $R>R_{mol}(t,\epsilon_{0})$ without stating it explicitly everytime.
By \eqref{diff6} and Fubini's theorem we have local integrability for $u_{R}(t\,, \cdot)$ so for a fixed $k\ge 2$, we can define 
\begin{equation}\label{avgu}
\bar{u}_{R}(t\,, x)=\liminf_{n \to \infty}\frac{1}{\vert B_{R}(x,r_{n}(k)) \vert}\int_{B_{R}(x,r_{n}(k))}u_{R}(t\,, y)\d y\,,
\end{equation}
where $r_{n}(k)$ is such that $\mu_{k}(r_{n+1}(k))=\mu_{k}(r_{n}(k))/2$.\\
Define 
\begin{equation}
I_{k}=\int_{S_{R}^{2}} \d x \int_{S_{R}^{2}} \d x' \Bigg\vert \frac{u_{R}(t,x)-u_{R}(t,x')}{\mu_{k}(R\theta(x\,,x'))} \Bigg\vert^{k}.
\end{equation}
By \eqref{diff6}, for any $0<a<2$, $k \ge 2$, $0<t<\infty$, $0<\epsilon_{0}<1$, $R>R_{mol}(t,\epsilon_{0})$ and $x\,,x' \in S_{R}^{2}$ such that $\theta(x\,,x')<t^{3/2}(1+\epsilon_{0})^{-1}R^{-4}$, 
\begin{align}\label{Ik}
\E I_{k} &\leq \int_{S_{R}^{2}}\d x \int_{S_{R}^{2}}\d x' \frac{\left(4\sqrt{2}C_{\sigma_{up}}\sqrt{kh_{up}(R)}(1+\epsilon_{0})^{1/3}R^{4/3}\theta(x\,,x')^{1/3}\right)^{k}}{\left(R\theta(x,x')\right)^{a+\frac{k}{3}}}\nonumber\\
%&=(2^{13/6}C_{\sigma_{up}}\sqrt{kh_{up}(R)}(1+\epsilon_{0})^{1/3})^{k}R^{-a+k}\int_{S_{R}^{2}}\d x \int_{S_{R}^{2}}\d x'\left(\theta(x\,,x')\right)^{-a}\nonumber\\
%&=(2^{13/6}C_{\sigma_{up}}\sqrt{kh_{up}(R)}(1+\epsilon_{0})^{1/3})^{k}R^{4-a+k}\int_{S^{2}}\d x_{1} \int_{S^{2}}\d x_{2}\left(\theta(x_{1}\,,x_{2})\right)^{-a}\nonumber\\
&=\left(4\sqrt{2}C_{\sigma_{up}}\sqrt{kh_{up}(R)}(1+\epsilon_{0})^{1/3}\right)^{k}R^{4-a+k}\int_{S^{2}} \left(2\pi\int_{0}^{\pi} \theta^{-a} \sin \theta \d \theta\right) \d x_{1}\nonumber\\
&\leq \left(4\sqrt{2}C_{\sigma_{up}}\sqrt{kh_{up}(R)}(1+\epsilon_{0})^{1/3}\right)^{k}R^{4-a+k}\left(8\pi^{2}\int_{0}^{\pi} \theta^{1-a}  \d \theta\right) \nonumber\\
&= \left((2-a)^{-1}\pi^{4-a}32\sqrt{2}C_{\sigma_{up}}\sqrt{kh_{up}(R)}(1+\epsilon_{0})^{1/3}\right)^{k}R^{4-a+k}.
\end{align}
The identity $1-\cos\theta=2\sin^{2}(\theta/2)$ gives us
\begin{align}\label{Bk}
\int_{0}^{\pi R}\vert B_{R}(x\,,r)\vert^{-2/k}\d r
&=\int_{0}^{\pi R}\left(2\pi \int_{0}^{r/R} R^{2} \sin \theta \d \theta \right)^{-2/k}\d r\nonumber\\
&=\frac{1}{(4\pi R^{2})^{2/k}}\int_{0}^{\pi R}\frac{1}{\left(\sin(r/2R)\right)^{4/k}}\d r.
\end{align}
For $k>4$, \eqref{Bk} is finite.\\
Apply Lemma \ref{GarsiaLemma1.1} to $u_{R}(t,x)$ to get for any $k>4$, $0<t\,,R<\infty$ and almost all $x \in S_{R}^{2}$, 
\begin{equation}
\bar{u}_{R}(t\,, x)=\lim_{n \to \infty}\frac{1}{\vert B_{R}(x,r_{n}(k)) \vert}\int_{B_{R}(x,r_{n}(k))}u_{R}(t\,, y)\d y\,,
\end{equation}
and is spatial-continuous.\\ By Fatou's lemma, Fubini's theorem, and \eqref{diff6}, for every $m>0$, $k>4$, $0<\epsilon_{0}<1$,$0<t<\infty$, $R>R_{mol}(t,\epsilon_{0})$,  
\begin{align*}
&\P\left(\vert \bar{u}_{R}(t,x)-u_{R}(t,x)\vert>2^{-m}\right)\nonumber\\
%&= \P\left(\Bigg\vert\liminf_{n \to \infty}\frac{\int_{B_{R}(x,r_{n}(k))}\left(u_{R}(t,y)-u_{R}(t,x)\right)\d y}{\vert B_{R}(x\,,r_{n}(k))\vert}\Bigg\vert >2^{-m}\right)\nonumber\\
%&\leq \P\left(\liminf_{n \to \infty}\frac{\int_{B_{R}(x,r_{n}(k))}\vert u_{R}(t,y)-u_{R}(t,x)\vert \d y}{\vert B_{R}(x\,,r_{n}(k))\vert}>2^{-m}\right)\nonumber\\
&\leq 2^{m}\E\left(\liminf_{n \to \infty}\frac{\int_{B_{R}(x,r_{n}(k))}\vert u_{R}(t,y)-u_{R}(t,x)\vert \d y}{\vert B_{R}(x\,,r_{n}(k)\vert}\right)\nonumber\\
%&\leq 2^{m}\liminf_{n \to \infty} \left(\frac{\int_{B_{R}(x,r_{n}(k))}\E\big[\vert u_{R}(t,y)-u_{R}(t,x)\vert\big] \d y}{\vert B_{R}(x\,,r_{n}(k))\vert}\right)\nonumber\\
&\leq 2^{m}\liminf_{n \to \infty} \left(\frac{\int_{B_{R}(x,r_{n}(k))}\E\big[\vert u_{R}(t,y)-u_{R}(t,x)\vert^{k}\big] \d y}{\vert B_{R}(x\,,r_{n}(k))\vert}\right)\nonumber\\
&\leq 2^{m} \lim_{n \to \infty}\left(4\sqrt{2}C_{\sigma_{up}}\sqrt{kh_{up}(R)}(1+\epsilon_{0})^{1/3}R^{4/3} r_{n}(k)^{1/3}\right)^{k}\nonumber\\
&=0.
\end{align*}
Let $m \to \infty$, then for any $x \in S_{R}^{2}$, $\bar{u}_{R}(t,x)=u_{R}(t,x)$ on a full probability subset $\Omega_{x}$ of $\Omega$. By Doob's separability theory, $\sup_{x \in S_{R}^{2}}\vert \bar{u}_{R}(t,x)-u_{R}(t,x) \vert =\sup_{x \in D_{R}}\vert \bar{u}_{R}(t,x)-u_{R}(t,x) \vert$ on a full probability subset $\Omega_{0}$ of $\Omega$ where $D_{R}$ is a countable subset of $S_{R}^{2}$. Then on the full probability subset $\displaystyle \Omega_{0} \cup\big(\cup_{x \in D_{R}} \Omega_{x}\big)$, $\sup_{x \in S_{R}^{2}}\vert \bar{u}_{R}(t,x)-u_{R}(t,x) \vert=0$.
This shows for each $0<t<\infty$, $\bar{u}_{R}(t\,,\cdot)$ is an a.s.-continuous modification of $u_{R}(t\,, \cdot)$, independent of the spatial variable.
For any fixed $0<a<2$, $k \ge 2$, take
\begin{equation}\label{mu}
\mu_{k}(r)=r^{\frac{1}{3}+\frac{a}{k}}.
\end{equation} 
Then 
\begin{equation}
C_{\mu_{k}}=2^{\frac{1}{3}+\frac{a}{k}}.
\end{equation}
By Theorem \ref{GarsiaTheorem1.2}, for every $0<t<\infty$, $k\ge 2$ and $x \,,x' \in S_{R}^{2}$ such that $R\theta(x\,,x')\leq 1$,
\begin{align}\label{ResultOfGarsia}
&\vert \bar{u}_{R}(t\,, x)-\bar{u}_{R}(t\,, x')\vert  \nonumber\\
&\leq4\cdot 2^{1/3+a/k} \left(2+2^{1/3+a/k}\right)I_{k}^{1/k}\int_{0}^{R\theta(x,x')}\vert B_{R}(r)\vert^{-2/k} \d (r^{1/3+a/k}).
\end{align}
Using the identity $\cos\theta=1-2\sin^{2}(\theta/2)$, we have for $k$ such that $\frac{3a+k}{3a-12+k}\left(4/\pi\right)^{2/k}\le 2$, every $0<t<\infty$, every $0<\epsilon_{0}<1$, and  $R\ge \max \big\{\left(3(2\pi)^{-1}(1+\epsilon_{0})^{-1}t^{3/2}\right)^{1/4}\,,R_{mol}(t,\epsilon_{0})\big\}$ and $x\,,x' \in S_{R}^{2}$ such that $\theta(x\,,x')\le t^{3/2}(1+\epsilon_{0})^{-1}R^{-4}$, 
\begin{align}
&\int_{0}^{R\theta(x,x')}\vert B_{R}(r)\vert^{-2/k} \d (r^{1/3+a/k}) \nonumber\\
%&=\int_{0}^{R\theta(x,x')}\left(2\pi \int_{0}^{r/R}R^{2}\sin\theta \d \theta \right)^{-2/k}(\frac{1}{3}+\frac{a}{k})r^{a/k-2/3}\d r\nonumber\\
%&=\left(\frac{1}{3}+\frac{a}{k}\right)(4\pi)^{-2/k}R^{-4/k}\int_{0}^{R\theta(x,x')}\frac{r^{\frac{a}{k}-\frac{2}{3}}}{\left(\sin (\frac{r}{2R})\right)^{4/k}}\d r\nonumber\\
&=\left(\frac{1}{3}+\frac{a}{k}\right)(4\pi)^{-2/k}R^{(a-4)/k+1/3}\int_{0}^{\theta(x,x')}u^{a/k-2/3}\left(\sin(u/2)\right)^{-4/k} \d u\nonumber\\
&\leq \left(\frac{1}{3}+\frac{a}{k}\right)(4\pi)^{-2/k}R^{(a-4)/k+1/3}\int_{0}^{\theta(x,x')}u^{a/k-2/3}\left(u/4\right)^{-4/k} \d u\nonumber\\
%&= \frac{3a+k}{3a-12+k}4^{-k/4}\pi^{-k/2}R^{(a-4)/k+1/3}\theta(x\,,x')^{(a-4)/k+1/3}\nonumber\\
&\le 2R^{(a-4)/k+1/3}\theta(x\,,x')^{(a-4)/k+1/3}.
\end{align}
Along these lines, it is required that $\theta(x,x')\le 2\pi/3$  in order to imply that $\sin(u/2)\ge u/4$ for all $0\le u \le \theta(x,x')$. That $\theta(x\,,x') \le \theta(x\,,x')\le t^{3/2}(1+\epsilon_{0})^{-1}R^{-4}$ and $R\ge \left(3(2\pi)^{-1}(1+\epsilon_{0})^{-1}t^{3/2}\right)^{1/4}$ will suffice for this purpose.
\newline Define $K_{0}=\inf \{k \ge 0: \frac{3a+k}{3a-12+k}(4/\pi)^{2/k}\le 2\}$ then by \eqref{ResultOfGarsia} for every $0<t<\infty$, every $0<\epsilon_{0}<1$, $k \ge  K_{0}$, $R\ge \max \big\{\left(3(2\pi)^{-1}(1+\epsilon_{0})^{-1}t^{3/2}\right)^{1/4}\,,R_{mol}(t,\epsilon_{0})\big\}$ and $x\,,x' \in S_{R}^{2}$ such that $\theta(x\,,x')\le \theta(x\,,x')\le t^{3/2}(1+\epsilon_{0})^{-1}R^{4}$ and $R\theta(x\,,x')\le 1$,
\begin{equation}
\vert \bar{u}_{R}(t\,, x)-\bar{u}_{R}(t\,, x')\vert \leq8\cdot 2^{1/3+a/k} \left(2+2^{1/3+a/k}\right)I_{k}^{1/k}R^{(a-4)/k+1/3}\theta(x\,,x')^{(a-4)/k+1/3}.
\end{equation}
This implies (using the estimate $2^{1/3 + a/k} < 6$) for every $0<t<\infty$, $0<\epsilon_{0}<1$, $0<\epsilon \le 1$, $0<a<2$, $k \ge \max\{2\,,K_{0}\}$, $R\ge \max \big\{\left(3(2\pi)^{-1}(1+\epsilon_{0})^{-1}t^{3/2}\right)^{1/4}\,,R_{mol}(t,\epsilon_{0})\big\}$ and $x\,,x' \in S_{R}^{2}$ such that $\theta(x\,,x')\le \theta(x\,,x')\le t^{3/2}(1+\epsilon_{0})^{-1}R^{-4}$ and $R\theta(x\,,x')\le \epsilon$, 
\begin{align}\label{kthsquare}
\vert \bar{u}_{R}(t\,, x)-\bar{u}_{R}(t\,, x')\vert^{k} &\leq 384^{k}I_{k}\epsilon^{k/3+a-4}.
\end{align}
\newpage
By \eqref{Ik}, \eqref{kthsquare}, for every $0<t<\infty$, $0<\epsilon_{0}< 1$, $0<a<2$, $k \ge \max\{2\,,K_{0}\}$, $R\ge \max \big\{\left(3(2\pi)^{-1}(1+\epsilon_{0})^{-1}t^{3/2}\right)^{1/4}\,,R_{mol}(t,\epsilon_{0})\big\}$, $n$ such that $\pi R 2^{-n} \leq 1$ and $\pi 2^{-n}<\theta(x\,,x')\le t^{3/2}(1+\epsilon_{0})^{-1}R^{-4}$, and $0<q<\frac{1}{3}$,
\begin{align}\label{fromnon}
&\E\left[\sup_{\pi 2^{-(n+1)} \leq \theta(x,x') \leq \pi 2^{-n}}\Bigg\vert \frac{\bar{u}_{R}(t\,, x)-\bar{u}_{R}(t\,, x')}{\left(R\theta(x,x')\right)^{q}}\Bigg\vert^{k}\right] \\
&\leq \E\left[\sup_{\theta(x,x') \leq \pi 2^{-n}}\Bigg\vert \frac{ \bar{u}_{R}(t\,, x)-\bar{u}_{R}(t\,, x')}{\left(\pi R2^{-(n+1)}\right)^{q}}\Bigg\vert^{k}\right]\nonumber\\
&\leq 384^{k}\E I_{k} \left(\pi R 2^{-n}\right)^{k/3+a-4}\frac{1}{\left(\pi R 2^{-(n+1)}\right)^{qk}}\nonumber\\
%&\leq 384^{k} \left(32\sqrt{2}(2-a)^{-1}\pi^{4-a}C_{\sigma_{up}}\sqrt{kh_{up}(R)}(1+\epsilon_{0})^{1/3}\right)^{k}R^{4-a+k} \left(\pi R 2^{-n}\right)^{k/3+a-4}\frac{1}{\left(\pi R 2^{-(n+1)}\right)^{qk}}\nonumber\\
&\leq \pi^{a-4}\left(12288\sqrt{2h_{up}(R)}C_{\sigma_{up}}(2-a)^{-1}(1+\epsilon_{0})^{1/3}\pi^{13/3-a-q}2^{q}R^{1/3-q}\right)^{k}k^{k/2}2^{-n((\frac{1}{3}-q)k+a-4)}.\nonumber
\end{align}
\par
Define $K_{1} =\inf\{k \ge 0:2^{-((1/3-q)k+a-4)}<1\}$. 
Summing from $n$ to $\infty$ in \eqref{fromnon} to get for every $0<t<\infty$, $0<\epsilon_{0}<1$, $0<a<2$, $0<q<\frac{1}{3}$, $k \ge \max\{2\,,K_{0}\,,K_{1}\}$, $R\ge \max \big\{\left(3(2\pi)^{-1}(1+\epsilon_{0})^{-1}t^{3/2}\right)^{1/4}\,,R_{mol}(t,\epsilon_{0})\big\}$, $n$ such that $\pi R 2^{-n} \leq 1$ and $\pi 2^{-n}<t^{3/2}(1+\epsilon_{0})^{-1}R^{-4}$, 
\begin{align}
&\E\left[\sup_{0<\theta(x,x') \leq \pi 2^{-n}}\Bigg\vert \frac{ \bar{u}_{R}(t\,, x)-\bar{u}_{R}(t\,, x')}{\left(R\theta(x,x')\right)^{q}}\Bigg\vert^{k}\right] \nonumber\\
&\leq \pi^{a-4}\left(12288\sqrt{2h_{up}(R)}C_{\sigma_{up}}(2-a)^{-1}(1+\epsilon_{0})^{1/3}\pi^{13/3-a-q}2^{q}R^{1/3-q}\right)^{k}
k^{k/2}\frac{2^{-n((\frac{1}{3}-q)k+a-4)}}{1-2^{-(k(\frac{1}{3}-q)+a-4)}}\nonumber\\
&\leq \pi^{a-4}\left(12288\sqrt{2h_{up}(R)}C_{\sigma_{up}}(2-a)^{-1}(1+\epsilon_{0})^{1/3}\pi^{13/3-a-q}2^{q}R^{1/3-q}\right)^{k}
k^{k/2}.
\end{align}
This finishes the proof since $\bar{u}_{R}(t\,,\cdot)$ is a version of $u_{R}(t\,,\cdot)$ with the modification uniform in the spatial variable.
\end{proof}

\begin{proof}[Proof of Theorem \ref{probabilityofsmalldifference}]
By Markov's inequality, and Theorem \ref{Holder Continuity}, for every $0<t\,,\gamma<\infty$, $0<\epsilon_{0}<1$, $0<a<2$, $0<q<\frac{1}{3}$, there exist finite positive $K(a,q)$ and $R_{mol}(t,\epsilon_{0})$ such that for all $k \ge \max\{2\,,K(a,q)\}$, $R\ge \max \big\{\left(3(2\pi)^{-1}(1+\epsilon_{0})^{-1}t^{3/2}\right)^{1/4}\,,R_{mol}(t,\epsilon_{0})\big\}$ and $n$ such that $\pi R 2^{-n} \leq 1$ and $\pi 2^{-n}<t^{3/2}(1+\epsilon_{0})^{-1}R^{-4}$,
\begin{align}\label{probabilitydifferencefixR}
&\P\left( \sup_{\theta(x,x') \leq \pi 2^{-n}}\vert u_{R}(t\,, x)-u_{R}(t\,, x')\vert > \pi R2^{-\gamma n}\right)\nonumber\\
&\leq  \E \left[ \sup_{\theta(x,x') \leq \pi 2^{-n}}\vert u_{R}(t\,, x)-u_{R}(t\,, x')\vert^{k} \right]\frac{1}{\left(\pi R 2^{-\gamma n}\right)^{k}}\nonumber\\
&\leq \pi^{a-4}\left(12288\sqrt{2h_{up}(R)}C_{\sigma_{up}}(2-a)^{-1}(1+\epsilon_{0})^{1/3}\pi^{10/3-a}2^{q}R^{-2/3}\right)^{k}k^{k/2}2^{(\gamma-q)nk}.
\end{align}
Choose $k=n$ to finish the proof.
\end{proof}

\section{An asymptotic upper bound of the supremum of the mild solution}
Following the idea of \cite{ConusJosephDavar}, we show in this section that for some fixed positive constant $C(t)$ which depends on a fixed finite positive $t$, $\sup_{x \in S_{R}^{2}} \vert u_{R}(t\,,x) \vert \ge C(t) \sqrt{\log R}$ asymptotically as $R \to \infty$ with high probability. The goal of this section is to prove the following main theorem.
\begin{theorem}\label{BigR}
Assume $\sup_{R>0}\sup_{x \in S_{R}^{2}} \vert u_{R}(t\,,x) \vert \le U<\infty$. For every $0<t<\infty$, $0<\epsilon_{0}<1$, there exists a finite positive constant $C(t,\epsilon_{0})$ such that 
\begin{equation}
\lim_{R \to \infty} \P\left( \exists x \in S_{R}^{2}: \vert u_{R}(t\,,x)\vert \ge C(t,\epsilon_{0})\left(\log R\right)^{1/4+C_{h_{lo}}/4-C_{h_{up}}/8}\right) =1.
\end{equation}
Moreover, for all $0<t<\infty$, $0<\epsilon_{0}<1$, finite positive constant $C$, there exist finite positive constants $C(t,\epsilon_{0})$ and $R(t,\epsilon_{0},C)$ such that for $R \ge R(t,\epsilon_{0},C)$,
\begin{equation}
\P\left(\exists x \in S_{R}^{2}: \vert u_{R}(t\,,x)\vert \ge  C(t,\epsilon_{0})\left(\log R\right)^{1/4+C_{h_{lo}}/4-C_{h_{up}}/8}\right)\ge1-R^{-C\pi \e^{-2}\left(1/2-C_{h_{up}}/4\right)}.
\end{equation}
\end{theorem}
We begin with some important definitions and lemmas that lead to the proof of Theorem \ref{BigR}.
\begin{definition}\label{utrb}
Define the ``space-truncated'' coupling process by
\begin{align}\label{utrb}
U_{t,R}^{(\beta)}(x)&=\int_{S_{R}^{2}}p_{R}(t\,,\theta(x\,,y))u_{R,0}(y)\d y \nonumber\\
&\hskip 1in \qquad +\int_{(0\,,t)\times B_{R}(x\,,\sqrt{\beta t})}p_{R}(t-s\,, \theta(x\,,y)) \sigma \big(U_{t,R}^{(\beta)}(y)\big) W(\d s \,, \d y).
\end{align}
\end{definition}
\begin{definition}\label{utrbn}
Define the $n$-th step Picard iteration of the ``space-truncated'' coupling process by
\begin{align*}
U_{t,R}^{(\beta \,,0)}(x)=u_{R\,,0}(x)
\end{align*}
and 
\begin{align*}
U_{t,R}^{(\beta \,,n)}(x)&=\int_{S_{R}^{2}}p_{R}(t\,,\theta(x\,,y))U_{0\,,R}^{(\beta \,,(n-1) )}(y)\d y \\ 
&\hskip 0.6in \qquad +\int_{(0\,,t)\times B_{R}(x\,,\sqrt{\beta t})}p_{R}(t-s\,, \theta(x\,,y)) \sigma \big(U_{t,R}^{(\beta\,,(n-1))}(y)\big) W(\d s \,, \d y).
\end{align*}
\end{definition} 
As in Section 3,  we can show the mild solution of \ref{utrb} exists as the unique $\P$-limit of its Picard iterations, is jointly measurable in time, space and probability and has spatial continuity up to a modification by Doob's separability theory.
By the same argument as in \cite{ConusJosephDavar}, we have the following independence result.
\begin{lemma}\label{independence}
For every $0<\beta \,, t\,, R\,, n <\infty$, and $x_{1}\,, x_{2}\,, \cdots \in S_{R}^{2}$ such that $d(x_{i}\,,x_{j})>2n\sqrt{\beta t}$ whenever $i \ne j$, $\{U_{t,R}^{(\beta,n)}(x_{j})\}_{j=1\,,2\,,\cdots}$ is a collection of $i.i.d$ random variables.
\end{lemma}
As a first step, we find upper and lower bounds of the moments of $U_{t,R}^{(\beta,n)}(x)$ to give a tail probability estimate of $U_{t,R}^{(\beta,n)}(x)$. 
The following gives lower bounds of the moments of $U_{t,R}^{(\beta,n)}(x)$.
\begin{lemma}\label{lowerbound of utrbn}
For every $0<t\,,\beta<\infty$, $0<\epsilon_{0}<1$, there exists a finite positive $R_{mol}(t\,,\pi/4\,,\epsilon_{0})$ such that for all positive integers $k$, and all $R\ge \max\big\{R_{mol}(t\,,\pi/4\,,\epsilon_{0})\,,4\sqrt{\beta t}/\pi\big\}$,
\begin{equation}
\E\left[U_{t,R}^{(\beta,n)}(x)^{2k}\right] \ge \frac{2\sqrt{\pi}}{\e}\left(\frac{4\pi^{2}h_{lo}(R)tC_{\sigma_{lo}}^{2}(1-\epsilon_{0})^{2}(1-\e^{-\beta/2})^{2}k}{\e}\right)^{k}.
\end{equation}
\end{lemma}
\begin{proof}
Take $t=0$ in Definition \ref{utrbn} to get for every $0<\beta\,,R<\infty$, positive integer $n$, $x \in S_{R}^{2}$,
\begin{equation}
U_{0,R}^{(\beta,n)}(x)=\int_{S_{R}^{2}}p_{R}(0\,,\theta(x\,,y))U_{0,R}^{(\beta,n-1)}(y)\d y \nonumber\\
\end{equation}
As in Section 3, we get by induction
\begin{equation}\label{u0rbn}
U_{0,R}^{(\beta,n)}(x)=u_{R\,,0}(x)\,,
\end{equation}
and
\begin{align}
U_{t,R}^{(\beta,n)}(x)&=\int_{S_{R}^{2}}p_{R}(t\,,\theta(x,y))u_{R\,,0}(y)\d y\nonumber\\
&\hskip 0.5in \qquad +\int_{[0\,,t] \times B_{R}(x \,, \sqrt{\beta t})}p_{R}(t-s\,, \theta(x\,,y))\sigma(U_{t,R}^{(\beta,n-1)}(y)) W(\d s \,,\d y).
\end{align}
Define a martingale $\{M(u)\}_{0 \le u \le t}$ by
\begin{align}
M(u)&=\int_{S_{R}^{2}}p_{R}(t\,,\theta(x,y))u_{R\,,0}(y)\d y\nonumber\\
&\hskip 0.7in \qquad +\int_{[0\,,u] \times B_{R}(x \,, \sqrt{\beta t})}p_{R}(t-s\,, \theta(x\,,y))\sigma(U_{s,R}^{(\beta,n-1)}(y)) W(\d s \,,\d y).
\end{align}
By Ito's formula, for all $k\ge 1$,
\begin{align}
M(u)^{2k}&=\left(\int_{S_{R}^{2}}p_{R}(t\,,\theta(x,y))u_{R\,,0}(y)\d y\right)^{2k}\nonumber\\
&\qquad +2k\int_{0}^{u}M(s)^{(2k-1)}\d M(s) +\frac{2k(2k-1)}{2}\int_{0}^{u} M(s)^{(2k-2)} \d \left< M\,,M\right>_{s}.
\end{align}
Let $u=t$ and take expectation to get
\begin{align}\label{ito formula}
\E[M(t)^{2k}]=\left(\int_{S_{R}^{2}}p_{R}(t\,,\theta(x,y))u_{R\,,0}(y)\d y\right)^{2k} +\frac{2k(2k-1)}{2}\E \left[\int_{0}^{t} M(s)^{(2k-2)} \d \left< M\,,M\right>_{s}\right].
\end{align}
For notational bervity, denote
\begin{equation}
g(t\,,R\,,n\,,x\,,y_{1}\,,y_{2})= p_{R}(t-s\,,\theta(x\,,y_{1}))p_{R}(t-s\,,\theta(x\,,y_{2}))
\sigma(U_{t,R}^{(\beta,n-1)}(y_{1}))\sigma(U_{t,R}^{(\beta,n-1)}(y_{2}))\,,
\end{equation}
and 
\begin{equation}
I(t\,,R\,,\beta)=\int_{B_{R}(x \,, \sqrt{\beta t}) \times B_{R}(x \,, \sqrt{\beta t})}
p_{R}(t-s\,,\theta(x\,,y_{1}))p_{R}(t-s\,,\theta(x\,,y_{2}))\d y_{1} \d y_{2}.
\end{equation}
Then
\begin{align}\label{quadratic moment}
&\E\left[\int_{0}^{t} M(s)^{(2k-2)} \d \left< M\,,M\right>_{s}\right]\nonumber\\
&\hskip 0.5in =E\Big[\int_{0}^{t} M(s)^{(2k-2)} \d s
\int_{B_{R}(x \,, \sqrt{\beta t}) \times B_{R}(x \,, \sqrt{\beta t})}  g(t\,,R\,,n\,,x\,,y_{1}\,,y_{2}) h_{R}(y_{1},y_{2})\d y_{1} \d y_{2}\Big]\nonumber\\
&\hskip 0.5in \ge h_{lo}(R)C_{\sigma_{lo}}^{2}\int_{0}^{t} \E \left[ M(s)^{(2k-2)} \right] \d s I(t\,,R\,,\beta).
\end{align}
Define
\begin{equation}\label{new measure}
\mu_{R,\beta}(t\,, \d s):=h_{lo}(R) \d s\int_{B_{R}(x \,, \sqrt{\beta t}) \times B_{R}(x \,, \sqrt{\beta t})} p_{R}(t-s\,,\theta(x\,,y_{1}))p_{R}(t-s\,,\theta(x\,,y_{2}))\d y_{1} \d y_{2}
\end{equation}
then \eqref{quadratic moment} can be written as
\begin{equation}\label{quadratic moment 2}
\E\left[\int_{0}^{t} M(s)^{(2k-2)} \d \left< M\,,M\right>_{s}\right] \ge C_{\sigma_{lo}}^{2}\int_{0}^{t} \E \left[ M(s)^{(2k-2)} \right] \mu_{R,\beta}(t\,, \d s).
\end{equation}
Use \eqref{quadratic moment 2} in \eqref{ito formula} to get
\begin{equation}
\E[M(t)^{2k}] \ge \frac{2k(2k-1)C_{\sigma_{lo}}^{2}}{2}\int_{0}^{t} \E \left[ M(s)^{(2k-2)} \right] \mu_{R,\beta}(t\,, \d s).
\end{equation}
By induction,
\begin{align}
\E[M(t)^{2k}]& \ge \frac{(2k)! C_{\sigma_{lo}}^{2k}}{2^{k}}\int_{0}^{t}\mu_{R,\beta}(t\,,\d s_{1})\int_{0}^{s_{1}}\mu_{R,\beta}(s_{1}\,, \d s_{2}) \cdots \int_{0}^{s_{1}}\mu_{R}(s_{k-1}\,,\d s_{k})\nonumber\\
&=\frac{(2k)! C_{\sigma_{lo}}^{2k}}{2^{k}k!}\left(\int_{0}^{t}\mu_{R,\beta}(t\,, \d s)\right)^{k}\nonumber\\
&=\frac{(2k)! C_{\sigma_{lo}}^{2k}}{2^{k}k!}\left(f_{\e,\beta}(0\,,R\,,t)\right)^{k}.
\end{align}
By Stirling's approximation that for all positive integer $n$,
\begin{equation}
\sqrt{2\pi}n^{n+\tfrac12}\e^{-n} \le n! \le \e n^{n+\tfrac12} \e^{-n}.
\end{equation}
This together with Lemma \ref{lowerbound of feb0r} implies that for every $0<t\,,\beta<\infty$, $0<\epsilon_{0}<1$, there exists a finite positive $R_{mol}(t\,,\pi/4\,,\epsilon_{0})$ such that for all positive integers $k\,,n$, and all $R\ge \max\big\{R_{mol}(t\,,\pi/4\,,\epsilon_{0})\,,4\sqrt{\beta t}/\pi \big\}$,
\begin{equation}
\E\left[U_{t,R}^{(\beta,n)}(x)^{2k}\right] \ge \frac{2\sqrt{\pi}}{\e}\left(\frac{4\pi^{2}h_{lo}(R)tC_{\sigma_{lo}}^{2}(1-\epsilon_{0})^{2}(1-\e^{-\beta/2})^{2}k}{\e}\right)^{k}.
\end{equation}
\end{proof}
The next lemma gives upper bounds of the moments of $U_{t,R}^{(\beta,n)}(x)$.
\begin{lemma}\label{upperbound of utrbn}
Assume for each finite positive $R$, $\sup_{x \in S_{R}^{2}}\vert u_{R,0}(x)\vert \le U_{R}<\infty$. Then for every $0< T\,,\alpha\,,R<\infty$, $0<\beta<\pi^{2}R^{2}/T$, $k$ such that $\L_{\sigma} \sqrt{2kh_{up}(R)/\alpha}<1$, and every positive integer $n$,
\begin{equation}
\sup_{0 \le t \le T}\sup_{x \in S_{R}^{2}}\e^{-\alpha t}\big\| U_{t,R}^{(\beta,n)}(x)\big\|_{k} \leq \frac{U_{R}+\vert \sigma(0) \vert\sqrt{2kh_{up}(R)/\alpha}}{1-\L_{\sigma} \sqrt{2kh_{up}(R)/\alpha}}.
\end{equation}
\end{lemma}
\begin{proof}
As in the proof of Theorem \ref{existenceuniqueness} we can apply Carlen's bound \cite{Carleen} on Burkholder-Gundy-Davis inequality, Lemma \ref{upperbound of fear} and a similar argument in \cite{FoondunDavar}, to get for $0 \le t \le T<\infty$, $R>0$, $0<\beta<\pi^{2}R^{2}/T$, $\alpha>0$ and $k$ such that $\L_{\sigma}\sqrt{2kh_{up}(R)/\alpha}<1$, 
\begin{align}
&\e^{-\alpha t}\big\|U_{t,R}^{(\beta,n)}(x)\big\|_{k}\nonumber\\ 
&\hskip 1in \leq \e^{-\alpha t}\Bigg\vert \int_{S_{R}^{2}}p_{R}(t\,,\theta(x,y))u_{R\,,0}(y)\d y\Bigg\vert \nonumber\\
&\hskip 1in \qquad +2\sqrt{kf_{\e}(\alpha\,,R\,,t)} \sup_{0 \le t \le T}\sup_{x \in S_{R}^{2}}\Big\|\e^{-\alpha t}\left(\vert \sigma (0)\vert +\L_{\sigma}\Big\vert U_{t,R}^{(\beta,n-1)}(x)\Big\vert\right)\Big\|_{k}\nonumber\\
%&\leq \e^{-\alpha t}\sup_{x \in S_{R}^{2}}\vert u_{R,0}(x)\vert+2\sqrt{kf_{\e}(\alpha\,,R)}
%\left(\vert\sigma(0)\vert+\L_{\sigma}\sup_{t_{0} \le t \le T}\sup_{x \in S_{R}^{2}}\e^{-\alpha t}\Big\| U_{t,R}^{(\beta,n-1)}(x)\Big\|_{k}\right)\nonumber\\
%&\leq U_{0}+2\sqrt{kf_{\e}(\alpha\,,R\,,t)}
%\left(\vert\sigma(0)\vert+\L_{\sigma}\sup_{t_{0} \le t \le T}\sup_{x \in S_{R}^{2}}\e^{-\alpha t}\Big\| U_{t,R}^{(\beta,n-1)}(x)\Big\|_{k}\right)\nonumber\\
&\hskip 1in \leq U_{R}+\sqrt{2kh_{up}(R)/\alpha}
\left(\vert\sigma(0)\vert+\L_{\sigma}\sup_{0 \le t \le T}\sup_{x \in S_{R}^{2}}\e^{-\alpha t}\Big\| U_{t,R}^{(\beta,n-1)}(x)\Big\|_{k}\right).
\end{align}
By induction and a little algebra, we have
\begin{equation}
\sup_{0 \le t \le T}\sup_{x \in S_{R}^{2}}\e^{-\alpha t}\Big\| U_{t,R}^{(\beta,n)}(x)\Big\|_{k} \leq \frac{U_{R}+\vert \sigma(0) \vert\sqrt{2kh_{up}(R)/\alpha}}{1-\L_{\sigma} \sqrt{2kh_{up}(R)/\alpha}}.
\end{equation}
\end{proof}
With $U_{t,R}^{(\beta)}(x)$ replacing the role of $U_{t,R}^{(\beta,n)}(x)$ and $U_{t,R}^{(\beta,n-1)}(x)$ in the proof of Lemma \ref{upperbound of utrbn}, we will get
\begin{lemma}\label{upperbound of utrb}
Assume for each finite positive $R$, $\sup_{x \in S_{R}^{2}}\vert u_{R,0}(x)\vert \le U_{R}<\infty$. Then for every $0< T\,,\alpha\,,R<\infty$, $0<\beta<\pi^{2}R^{2}/T$, $k$ such that $\L_{\sigma} \sqrt{2kh_{up}(R)/\alpha}<1$, 
\begin{equation}
\sup_{0 \le t \le T}\sup_{x \in S_{R}^{2}}\e^{-\alpha t}\big\| U_{t,R}^{(\beta)}(x)\big\|_{k} \leq \frac{U_{R}+\vert \sigma(0) \vert\sqrt{2kh_{up}(R)/\alpha}}{1-\L_{\sigma} \sqrt{2kh_{up}(R)/\alpha}}.
\end{equation}
\end{lemma}
The following lemma gives a tail probability estimate based on the previous lemmas.
\begin{lemma}\label{tail probability of utrbn}
Assume for each finite positive $R$, $\sup_{x \in S_{R}^{2}}\vert u_{R,0}(x)\vert \le U_{R}<\infty$. For every $0<t\,,\alpha\,,\beta<\infty$, $0<\epsilon_{0}<1$, there exists a finite positive $R_{mol}(t\,,\pi/4\,,\epsilon_{0})$ such that for all positive integer $n$, and all $R\ge \max\big\{R_{mol}(t\,,\pi/4\,,\epsilon_{0})\,,4\sqrt{\beta t}/\pi \big\}$, positive integer $k$ such that $\L_{\sigma} \sqrt{2kh_{up}(R)/\alpha}<1$, all $\lambda$ such that $0<\lambda<\pi \sqrt{h_{lo}(R)t/\e}C_{\sigma_{lo}}(1-\epsilon_{0})(1-\e^{-\beta/2})\sqrt{k}$, 
\begin{align}
&\P\left(\big\vert U_{t,R}^{(\beta,n)}(x) \big\vert \ge \lambda \right) \nonumber\\
&\ge \pi \e^{-(4\alpha t +2)k-2}\left(4\pi^{2}h_{lo}(R)tC_{\sigma_{lo}}^{2}(1-\epsilon_{0})^{2}(1-\e^{-\beta/2})^{2}k\right)^{2k}
\left(\frac{U_{R}+2\vert \sigma(0) \vert\sqrt{\frac{2kh_{up}(R)}{\alpha}}}{1-2\L_{\sigma} \sqrt{\frac{2kh_{up}(R)}{\alpha}}}\right)^{-4k}.
\end{align}
\end{lemma}
\begin{proof}
By Paley-Zygmund inequality, Lemma \ref{lowerbound of utrbn} and Lemma \ref{upperbound of utrbn}, we have for every $0<t\,,\alpha\,,\beta<\infty$, $0<\epsilon_{0}<1$, there exists a finite positive $R_{mol}(t\,,\pi/4\,,\epsilon_{0})$ such that for all positive integer $n\ge 2$, and all $R\ge \max\big\{R_{mol}(t\,,\pi/4\,,\epsilon_{0})\,,4\sqrt{\beta t}/\pi \big\}$, positive integer $k$ such that $\L_{\sigma} \sqrt{2kh_{up}(R)/\alpha}<1$, all $\lambda$ such that $0<\lambda<\pi \sqrt{h_{lo}(R)t/\e}C_{\sigma_{lo}}(1-\epsilon_{0})(1-\e^{-\beta/2})\sqrt{k}$, 
\begin{align*}
&\P\left(\big\vert U_{t,R}^{(\beta,n)}(x) \big\vert \ge \lambda \right) \\
&\ge \P\left(\big\vert U_{t,R}^{(\beta,n)}(x)\big\vert \ge\tfrac12 \Big\| U_{t,R}^{(\beta,n)}(x)\Big\|_{k}\right)\\
%&\ge \frac{\left(1-\frac{1}{2^{2k}}\right)^{2}\left(\E \left[\vert U_{t,R}^{(\beta,n)}(x) \vert^{2k} \right]\right)^{2}}{\E \left[\vert U_{t,R}^{(\beta,n)}(x) \vert^{4k}\right]}\\
&\ge \frac{\left(\E \left[\big\vert U_{t,R}^{(\beta,n-1)}(x) \big\vert^{2k} \right]\right)^{2}}{4\E \left[\big\vert U_{t,R}^{(\beta,n)}(x) \big\vert^{4k}\right]}\\
&\ge \pi \e^{-(4\alpha t +2)k-2}\left(4\pi^{2}h_{lo}(R)tC_{\sigma_{lo}}^{2}(1-\epsilon_{0})^{2}(1-\e^{-\beta/2})^{2}k\right)^{2k}
\left(\frac{U_{R}+2\vert \sigma(0) \vert\sqrt{\frac{2kh_{up}(R)}{\alpha}}}{1-2\L_{\sigma} \sqrt{\frac{2kh_{up}(R)}{\alpha}}}\right)^{-4k}.
\end{align*}
\end{proof}
Now we have obtained a tail probability estimate of $U_{t,R}^{(\beta,n)}(x)$. Based on the approximation to $u_{t\,,R}(x)$ by $U_{t,R}^{(\beta,n)}(x)$, we can achieve the goal of finding a tail probability estimate of $u_{t\,,R}(x)$. The accuracy of the approximation is described in the following three lemmas.
\begin{lemma}\label{approximation by utrb}
Assume for each finite positive $R$, $\sup_{x \in S_{R}^{2}}\vert u_{R,0}(x)\vert \le U_{R}<\infty$ and that $\alpha \asymp_{R} \beta \asymp_{R} \left(\log R\right)^{c}$ where $0<c<1$ is a constant. Then for every $0<t<\infty$, $R>R_{mol}(t)$ where $R_{mol}(t)$ is as in Lemma \ref{upperbound of tfebar} and all $k \ge 2$ such that $\L_{\sigma}\sqrt{2kh_{up}(R)/\alpha}<1$, 
\begin{equation}
\sup_{x \in S_{R}^{2}}\e^{-\alpha t}\Big\|u_{R}(t,x)-U_{t,R}^{(\beta)}(x)\Big\|_{k} \leq \frac{2\sqrt{2} t^{1/2}\sqrt{kh_{up}(R)}\e^{-\sqrt{\alpha \beta t}}\left(\vert \sigma(0) \vert + \L_{\sigma} \frac{U_{R}+\vert \sigma(0) \vert \sqrt{\frac{2kh_{up}(R)}{\alpha}}}{1-\L_{\sigma}\sqrt{\frac{2kh_{up}(R)}{\alpha}}}\right)}{1-\L_{\sigma}\sqrt{2h_{up}(R)k/\alpha}}.
\end{equation}
\end{lemma}
\begin{proof}
Recall from Definition \ref{utrb} that
\begin{align*}
U_{t,R}^{(\beta)}(x)&=\int_{S_{R}^{2}}p_{R}(t\,,\theta(x,y))u_{R\,,0}(y)\d y\\
&\hskip 1in \qquad +\int_{(0\,,t)\times B_{R}(x\,,\sqrt{\beta t})}p_{R}(t-s\,, \theta(x\,,y)) \sigma \big(U_{t,R}^{(\beta)}(y)\big) W(\d s \,, \d y).
\end{align*}
Define a coupling process by
\begin{align*}
V_{t,R}(x)&=\int_{S_{R}^{2}}p_{R}(t\,,\theta(x,y))u_{R\,,0}(y)\d y\\
&\hskip 1.2in \qquad +\int_{(0\,,t)\times S_{R}^{2}}p_{R}(t-s\,, \theta(x\,,y)) \sigma \big(U_{t,R}^{(\beta)}(y)\big) W(\d s \,, \d y).
\end{align*}
Denote 
\begin{equation}
S_{t\,,R\,,\beta}=[0\,,t]\times S_{R}^{2} \setminus B_{R}\big(x\,,\sqrt{\beta t}\big)\times S_{R}^{2} \setminus B_{R}\big(x\,,\sqrt{\beta t}\big)\,,
\end{equation}
and 
\begin{equation}
g_{1}(s\,,R\,,\beta\,,x\,,y_{1}\,,y_{2})=\e^{-2 \alpha (t-s)}p_{R}\big(t-s\,,\theta(x\,,y_{1})\big)p_{R}\big(t-s\,,\theta(x\,,y_{2})\big)
\sigma\big(U_{s,R}^{(\beta)}(y_{1})\big)\sigma\big(U_{s,R}^{(\beta)}(y_{2})\big)\,,
\end{equation}
and 
\begin{align}
g_{2}(s\,,R\,,\beta\,,x\,,y_{1}\,,y_{2})&=\e^{-2 \alpha (t-s)}p_{R}\big(t-s\,,\theta(x\,,y_{1})\big)p_{R}\big(t-s\,,\theta(x\,,y_{2})\big)\nonumber\\
&\qquad \cdot \left\vert \sigma\big(u_{s,R}(y_{1})\big)-\sigma\big(U_{s,R}^{(\beta)}(y_{1})\big)\right\vert  
\cdot \left\vert \sigma\big(u_{s,R}(y_{2})\big)-\sigma\big(U_{s,R}^{(\beta)}(y_{2})\big)\right\vert
\end{align}
for notational brevity.\\
Then, as in the proof of Theorem \ref{existenceuniqueness} we can apply  Carlen-Kr\'ee's bound on Burkholder-Gundy-Davis inequality \cite{Carleen}, Lemma \ref{upperbound of tfebar}, a similar argument in \cite{FoondunDavar}, and Lemma \ref{upperbound of utrb} to get under the assumption that $\alpha \asymp_{R} \beta \asymp_{R} \left(\log R\right)^{c}$ where $0<c<1$ is a constant, for every $0<t<\infty$ there exists a finite positive $R_{mol}(t)$ such that for all $R>R_{mol}(t)$, $k \ge 2$ such that $\L_{\sigma}\sqrt{2h_{up}(R)k/\alpha}<1$, 
\begin{align}\label{differ1}
&\e^{-\alpha t}\Big\|U_{t,R}^{(\beta)}(x)-V_{t,R}(x)\Big\|_{k}\\
&\hskip 1in \qquad \leq 2\sqrt{k}\Bigg\|\sqrt{\int_{S_{t,R,\beta}}
h_{R}(y_{1},y_{2}) \e^{-2 \alpha s}g_{1}(s\,,R\,,\beta\,,x\,,y_{1}\,,y_{2})\d s \d y_{1} \d y_{2}} \Bigg\|_{k}\nonumber\\
&\hskip 1in \qquad \leq 2\sqrt{k\tilde{f}_{\e,\beta}(\alpha\,,R\,,t)}
\sup_{t \ge 0}\sup_{x \in S_{R}^{2}}\e^{-\alpha t}\left(\vert \sigma(0) \vert+\L_{\sigma}\Big\|U_{t,R}^{(\beta)}(x)\Big\|_{k}\right)\nonumber\\
&\hskip 1in \qquad \leq 2\sqrt{2} t^{1/2}\sqrt{kh_{up}(R)}\e^{-\sqrt{\alpha \beta t}}
\left(\vert \sigma(0) \vert + \L_{\sigma} \frac{U_{R}+\vert \sigma(0) \vert \sqrt{2kh_{up}(R)/\alpha}}{1-\L_{\sigma}\sqrt{2kh_{up}(R)/\alpha}}\right)\nonumber.
\end{align}
By Lemma \ref{upperbound of fear}, for every $0<t\,,R\,,\alpha<\infty$, $0<\beta<\pi^{2}R^{2}/t$, $k \ge 2$,
\begin{align}\label{differ2}
&\e^{-\alpha t}\left\|u_{R}(t,x)-V_{t,R}(x)\right\|_{k}\nonumber\\
&\hskip 1.2in \qquad \leq 2\sqrt{k}\Bigg\| \sqrt{\int_{S_{t,R,\beta}}
h_{up}(R)\e^{-2\alpha s}g_{2}(s\,,R\,,\beta\,,x\,,y_{1}\,,y_{2})\d s \d y_{1} \d y_{2}}\Bigg\|_{k}\nonumber\\
&\hskip 1.2in \qquad \leq 2\L_{\sigma}\sqrt{kf_{\e}(\alpha\,,R\,,t)}
\sup_{x \in S_{R}^{2}}\e^{-\alpha t}\Big\|u_{R}(t,x)-U_{t,R}^{(\beta)}(x)\Big\|_{k}\nonumber\\
&\hskip 1.2in \qquad \leq \L_{\sigma}\sqrt{2h_{up}(R)k/\alpha}
\sup_{x \in S_{R}^{2}}\e^{-\alpha t}\Big\|u_{R}(t,x)-U_{t,R}^{(\beta)}(x)\Big\|_{k}.
\end{align}
From \eqref{differ1} and \eqref{differ2}, we get for every $0<t<\infty$, $R>R_{mol}(t)$, and $0<\beta<\pi^{2}R^{2}/16t$, $k \ge 2$ and $\alpha>0$ such that $\L_{\sigma}\sqrt{2h_{up}(R)k/\alpha}<1$, 
\begin{align}\label{differ3}
&\sup_{x \in S_{R}^{2}}\e^{-\alpha t}\Big\|u_{R}(t,x)-U_{t,R}^{(\beta)}(x)\Big\|_{k} \nonumber\\
&\hskip 1in \qquad \leq 2\sqrt{2} t^{1/2}\sqrt{kh_{up}(R)}\e^{-\sqrt{\alpha \beta t}}
\left(\vert \sigma(0) \vert + \L_{\sigma} \frac{U_{R}+\vert \sigma(0) \vert \sqrt{2h_{up}(R)k/\alpha}}{1-\L_{\sigma}\sqrt{2h_{up}(R)k/\alpha}}\right)\nonumber\\
&\hskip 1.2in  \qquad +\L_{\sigma}\sqrt{2h_{up}(R)k/\alpha}\sup_{x \in S_{R}^{2}}\e^{-\alpha t}\Big\|u_{R}(t,x)-U_{t,R}^{(\beta)}(x)\Big\|_{k}.
\end{align}
By the same argument, we can also get for every $0<t<\infty$, $R>R_{mol}(t)$, and $0<\beta<\pi^{2}R^{2}/16t$, $k \ge 2$ and $\alpha>0$ such that $\L_{\sigma}\sqrt{2h_{up}(R)k/\alpha}<1$, and every positive integer $n$,
\begin{align}
&\sup_{x \in S_{R}^{2}}\e^{-\alpha t}\Big\|u_{t,R}^{(n)}(x)-U_{t,R}^{(\beta,n)}(x)\Big\|_{k} \nonumber\\
&\hskip 1in \qquad \leq 2\sqrt{2} t^{1/2}\sqrt{kh_{up}(R)}\e^{-\sqrt{\alpha \beta t}}
\left(\vert \sigma(0) \vert + \L_{\sigma} \frac{U_{R}+\vert \sigma(0) \vert \sqrt{2kh_{up}(R)/\alpha}}{1-\L_{\sigma}\sqrt{2kh_{up}(R)/\alpha}}\right)\nonumber\\
&\hskip 1.2in \qquad +\L_{\sigma}\sqrt{2h_{up}(R)k/\alpha}\sup_{x \in S_{R}^{2}}\e^{-\alpha t}\Big\|u_{t,R}^{(n-1)}(x)-U_{t,R}^{(\beta,n-1)}(x)\Big\|_{k}.
\end{align}
This rules out the possibility of 
\begin{align}
\sup_{x \in S_{R}^{2}}\e^{-\alpha t}\Big\|u_{R}(t,x)-U_{t,R}^{(\beta)}(x)\Big\|_{k} = \infty\,,
\end{align}
since by assumption 
\begin{equation}
\L_{\sigma}\sqrt{2h_{up}(R)k/\alpha}<1.
\end{equation}
After a little algebra in \eqref{differ3}, we arrive at the inequality
\begin{equation}
\sup_{x \in S_{R}^{2}}\e^{-\alpha t}\Big\|u_{R}(t,x)-U_{t,R}^{(\beta)}(x)\Big\|_{k} \leq \frac{2\sqrt{2} t^{1/2}\sqrt{kh_{up}(R)}\e^{-\sqrt{\alpha \beta t}}\left(\vert \sigma(0) \vert + \L_{\sigma} \frac{U_{R}+\vert \sigma(0) \vert \sqrt{\frac{2kh_{up}(R)}{\alpha}}}{1-\L_{\sigma}\sqrt{\frac{2kh_{up}(R)}{\alpha}}}\right)}{1-\L_{\sigma}\sqrt{2h_{up}(R)k/\alpha}}.
\end{equation}
\end{proof}
\begin{lemma}\label{approximation by utrbn}
Assume for each finite positive $R$, $\sup_{x \in S_{R}^{2}}\vert u_{R,0}(x)\vert \le U_{R}<\infty$. For every $0<T\,,R<\infty$, $0<\beta<\pi^{2}R^{2}/T$, $k \ge 2$, $\alpha>0$ such that $\L_{\sigma}\sqrt{2h_{up}(R)k/\alpha}<1$ and $\sqrt{2h_{up}(R)k/\alpha}<1$ and every positive integer $n$,
\begin{align*}
\sup_{0 \le t \le T}\sup_{x \in S_{R}^{2}}\e^{-\alpha t}\Big\|U_{t,R}^{(\beta)}(x)-U_{t,R}^{(\beta,n)}(x)\Big\|_{k}
\leq (C_{\sigma_{up}}+2U_{R})\frac{\left(\L_{\sigma}\sqrt{2h_{up}(R)k/\alpha}\right)^{n}}{U_{R}-\L_{\sigma}\sqrt{2h_{up}(R)k/\alpha}}.
\end{align*}
\end{lemma}
\begin{proof}
%Recall from Definition \ref{utrbn} and \eqref{u0rbn} that
%\begin{align*}
%U_{t,R}^{(\beta \,,n)}(x)&=\int_{S_{R}^{2}}p_{R}(t\,,\theta(x,y))u_{R,0}(y)\d y\\
%&\qquad +\int_{(0\,,t)\times B_{R}(x\,,\sqrt{\beta t})}p_{R}(t-s\,, \theta(x\,,y)) \sigma (U_{t,R}^{(\beta\,,(n-1))}(y)) W(\d s \,, \d y).
%\end{align*}
As in the proof of Theorem \ref{existenceuniqueness} we can apply Carlen-Kr\'ee's bound on Burkholder-Gundy-Davis inequality \cite{Carleen}, Lemma \ref{upperbound of fear}, a similar argument in \cite{FoondunDavar}, and Lemma \ref{upperbound of utrb} to get for every $0<T\,,R<\infty$, $0<\beta<\pi^{2}R^{2}/T$, $k \ge 2$, $\alpha>0$ such that $\L_{\sigma}\sqrt{2h_{up}(R)k/\alpha}<1$ and $\sqrt{2h_{up}(R)k/\alpha}<1$ and every positive integer $n$,
\begin{align*}
&\sup_{0 \le t \le T}\sup_{x \in S_{R}^{2}}\e^{-\alpha t}\Big\|U_{t,R}^{(\beta,n+1)}(x)-U_{t,R}^{(\beta,n)}(x)\Big\|_{k}\\
%&\leq 2\L_{\sigma}\sqrt{kf_{\e,\beta}(\alpha\,,R\,,t)}
%\sup_{0 \le t \le T}\sup_{x \in S_{R}^{2}}\e^{-\alpha t}\Big\|U_{t,R}^{(\beta,n)}(x)-U_{t,R}^{(\beta,n-1)}(x)\Big\|_{k}\\
&\hskip 1.5in \qquad \leq 2\L_{\sigma}\sqrt{kf_{\e}(\alpha\,,R\,,t)}
\sup_{0 \le t \le T}\sup_{x \in S_{R}^{2}}\e^{-\alpha t}\Big\|U_{t,R}^{(\beta,n)}(x)-U_{t,R}^{(\beta,n-1)}(x)\Big\|_{k}\\
&\hskip 1.5in \qquad \leq \L_{\sigma}\sqrt{2h_{up}(R)k/\alpha}\sup_{0\le t \le T}\sup_{x \in S_{R}^{2}}\e^{-\alpha t}\Big\|U_{t,R}^{(\beta,n)}(x)-U_{t,R}^{(\beta,n-1)}(x)\Big\|_{k}.
\end{align*}
By induction, for $m>n$,
\begin{align}\label{induction}
&\sup_{0 \le t \le T}\sup_{x \in S_{R}^{2}}\e^{-\alpha t}\Big\|U_{t,R}^{(\beta,m)}(x)-U_{t,R}^{(\beta,n)}(x)\Big\|_{k}\\
&\hskip 1.5in \qquad \leq \frac{\left(\L_{\sigma}\sqrt{2h_{up}(R)k/\alpha}\right)^{n}}{1-\L_{\sigma}\sqrt{2h_{up}(R)k/\alpha}}\sup_{t \ge 0}\sup_{x \in S_{R}^{2}}\e^{-\alpha t}\Big\|U_{t,R}^{(\beta,1)}(x)-U_{t,R}^{(\beta,0)}(x)\Big\|_{k}\nonumber.
\end{align}
Note that
\begin{align*}
U_{t,R}^{(\beta,1)}(x)-U_{t,R}^{(\beta,0)}(x)&=\int_{S_{R}^{2}}p_{R}\big(t\,,\theta(x\,,y)\big)u_{R,0}(y)\d y-u_{R,0}(x)\\
&\qquad +\int_{[0,t]\times B_{R}(x\,,\sqrt{\beta t})}p_{R}\big(t-s\,,\theta(x\,,y)\big)\sigma\big(u_{R,0}(y)\big)W(\d s\,,\d y).
\end{align*}
For every $0<t\,,R<\infty$,
\begin{align*}
&\e^{-\alpha t}\Bigg\vert \int_{S_{R}^{2}}p_{R}\big(t\,,\theta(x\,,y)\big)u_{R,0}(y)\d y-u_{R,0}(x)\Bigg\vert\\
&\hskip 2in \qquad \leq \e^{-\alpha t} \sup_{x \in S_{R}^{2}}\vert u_{R,0}(x)\vert
\left(1+\int_{S_{R}^{2}}p_{R}\big(t\,,\theta(x\,,y)\big)\d y\right) \\ 
&\hskip 2in \qquad\leq 2U_{R}.
\end{align*}
Since $\sqrt{2h_{up}(R)k/\alpha}<1$, by Carlen-Kr\'ee's bound on the Burkholder-Gundy-Davis inequality,
\begin{align*}
\e^{-\alpha t}\Bigg\|\int_{[0,t]\times B_{R}(x\,,\sqrt{\beta t})}p_{R}(t-s\,,\theta(x\,,y))\sigma(u_{R,0}(y))W(\d s\,,\d y)\Bigg\|_{k}
%&\leq 2C_{\sigma_{up}}\sqrt{kf_{\e,\beta}(\alpha\,,R)}\\
%\le 2C_{\sigma_{up}}\sqrt{kf_{\e}(\alpha\,,R\,,t)} 
%&\leq C_{\sigma_{up}}\sqrt{2h_{up}(R)k/\alpha}\\
\le C_{\sigma_{up}}.
\end{align*}
Let $m \rightarrow \infty$ in \eqref{induction} to get
\begin{align*}
\sup_{0 \le t \le T}\sup_{x \in S_{R}^{2}}\e^{-\alpha t}\Big\|U_{t,R}^{(\beta)}(x)-U_{t,R}^{(\beta,n)}(x)\Big\|_{k}
\leq (C_{\sigma_{up}}+2U_{R})\frac{\left(\L_{\sigma}\sqrt{2h_{up}(R)k/\alpha}\right)^{n}}{1-\L_{\sigma}\sqrt{2h_{up}(R)k/\alpha}}.
\end{align*}
\end{proof}
\begin{lemma}\label{close in probability}
Assume for each finite positive $R$, $\sup_{x \in S_{R}^{2}}\vert u_{R,0}(x)\vert \le U_{R}<\infty$ and that $\alpha \asymp_{R} \beta \asymp_{R} \left(\log R\right)^{c}$ where $0<c<1$ is a constant. Then for every $0<t<\infty$, $R>R_{mol}(t)$ where $R_{mol}(t)$ is as in Lemma \ref{upperbound of tfebar}, $k \ge 2$ such that $\L_{\sigma}\sqrt{2h_{up}(R)k/\alpha}<1$ and $\sqrt{2h_{up}(R)k/\alpha}<1$, every positive integer $n$, $\lambda>0$, $N>1$ points $x_{1}\,, \cdots \,,x_{N}\in S_{R}^{2}$,
\begin{align*}
&\P\left(\max_{1 \le j \le N}\Big\vert U_{t\,,R}^{(\beta,n)}(x_{j})-u_{t,R}(x_{j})\Big\vert >\lambda \right) \\
&\leq \frac{N}{2}(\lambda/2)^{-k}\left(\frac{2\sqrt{2} t^{1/2}\sqrt{kh_{up}(R)}\e^{\alpha t-\sqrt{\alpha \beta t}}\left(\vert \sigma(0) \vert + \L_{\sigma} \frac{U_{R}+\vert \sigma(0) \vert \sqrt{2kh_{up}(R)/\alpha}}{1-\L_{\sigma}\sqrt{2kh_{up}(R)/\alpha}}\right)}{1-\L_{\sigma}\sqrt{2h_{up}(R)k/\alpha}}\right)^{k}\\
&\qquad +\frac{N}{2}(\lambda/2)^{-k}\left((C_{\sigma_{up}}+2U_{R})\frac{\left(\L_{\sigma}\sqrt{2h_{up}(R)k/\alpha}\right)^{n}}{1-\L_{\sigma}\sqrt{2h_{up}(R)k/\alpha}}\right)^{k}.
\end{align*}
\end{lemma}
\begin{proof}
By Lemma \ref{approximation by utrb}, Lemma \ref{approximation by utrbn}, Markov's inequality and Jensen's inequaltiy,
\begin{align*}
&\P\left(\max_{1 \le j \le N}\vert U_{t\,,R}^{(\beta,n)}(x_{j})-u_{t,R}(x_{j})\vert >\lambda \right)\\
%&\leq \sum_{j=1}^{N}\P\left(\vert U_{t\,,R}^{(\beta,n)}(x_{j})-u_{t,R}(x_{j})\vert >\lambda \right)\\
&\leq N\lambda^{-k}\sup_{t \ge 0}\sup_{x \in S_{R}^{2}}\E\left(\left[U_{t,R}^{(\beta,n)}(x)-u_{R}(t,x)\right]^{k}\right)\\
%&\leq N\lambda^{-k}\sup_{t \ge 0}\sup_{x \in S_{R}^{2}}\Big(\|U_{t,R}^{(\beta)}(x)-U_{t,R}^{(\beta,n)}(x)\|_{k}
%+\|u_{R}(t,x)-U_{t,R}^{(\beta)}(x)\|_{k}\Big)^{k}\\
&\leq N(\lambda/2)^{-k}\cdot \frac12\sup_{t \ge 0}\sup_{x \in S_{R}^{2}}\Big(\Big\|U_{t,R}^{(\beta)}(x)-U_{t,R}^{(\beta,n)}(x)\Big\|_{k}^{k}+\Big\|u_{R}(t,x)-U_{t,R}^{(\beta)}(x)\Big\|_{k}^{k}\Big)\\
&\leq \frac{N}{2}(\lambda/2)^{-k}\left(\frac{2\sqrt{2} t^{1/2}\sqrt{kh_{up}(R)}\e^{\alpha t-\sqrt{\alpha \beta t}}\left(\vert \sigma(0) \vert + \L_{\sigma} \frac{U_{R}+\vert \sigma(0) \vert \sqrt{2kh_{up}(R)/\alpha}}{1-\L_{\sigma}\sqrt{2kh_{up}(R)/\alpha}}\right)}{1-\L_{\sigma}\sqrt{2h_{up}(R)k/\alpha}}\right)^{k}\\
&\qquad +\frac{N}{2}(\lambda/2)^{-k}\left((C_{\sigma_{up}}+2U_{R})\frac{\left(\L_{\sigma}\sqrt{2h_{up}(R)k/\alpha}\right)^{n}}{1-\L_{\sigma}\sqrt{2h_{up}(R)k/\alpha}}\right)^{k}.
\end{align*}
\end{proof}
We are now ready to prove the main theorem of this section that gives the aymptotic lower bound of $\sup_{x \in S_{R}^{2}}\vert u_{R}(t\,,x)\vert$.
\begin{proof}[Proof of Theorem \ref{BigR}]
Assume throughout the proof, $\alpha \asymp_{R} \beta  \asymp_{R} \left(\log R\right)^{1/2+C_{h_{up}}/4}$, \\$n \asymp_{R} \log R$, $ k \asymp_{R} \left(\log R\right)^{1/2-C_{h_{up}}/4}$, $\lambda \asymp_{R} \left(\log R\right)^{1/4+C_{h_{lo}}/4-C_{h_{up}}/8}$, %with $\asymp$ meaning the variables are asymptotically equivalent as $R \to \infty$, more precisely, the quotient of any two of the variables as $R \to \infty$ tends to some finite positive constant.
and 
\begin{equation}
\sup_{R>0}U_{R} \le U <\infty\,,
\end{equation}
and
\begin{equation}\label{kalpha}
\L_{\sigma}\sqrt{2h_{up}(R)k/\alpha}<1\,,
\end{equation}
and 
\begin{equation}\label{kalpha2}
\sqrt{2h_{up}(R)k/\alpha}<1\,,
\end{equation}
%(the above can be replaced by $\sqrt{2h_{up}(R)k/\alpha}<1$ in order for Lemma \ref{approximation by utrbn} to apply since $C_{u}>1$)\\ 
and
\begin{equation}\label{klambda}
0<\lambda<\pi \sqrt{h_{lo}(R)t/\e}C_{\sigma_{lo}}(1-\epsilon_{0})(1-\e^{-\beta/2})\sqrt{k}/2.
\end{equation}
Whenever a statement/an equality/an inequality involves the above variables, it is assumed the involved varibales are subjected to the above estimates. More accurate estimations will be given along the proof. By Lemma \ref{independence}, for all $0<t\,,R<\infty$ and every positive integer $n\,,N$  such that
\begin{equation}\label{independence condition}
2n\sqrt{\beta t}N<2\pi R\,,
\end{equation} 
there exist $N$ points $x_{1}\,,\cdots\,,x_{N}$ such that 
$U_{t,R}^{(\beta,n)}(x_{1})\,,\cdots\,,U_{t,R}^{(\beta,n)}(x_{N})$ are i.i.d. random variables. By Lemma \ref{tail probability of utrbn} and Lemma \ref{close in probability}, for every $0<t<\infty$ and $0<\epsilon_{0}<1$, there exists a finite positive $R(t\,,\epsilon_{0})$ such that for all $R \ge R(t\,,\epsilon_{0})$, % positive integer $k \ge 2$ such that $\L_{\sigma}\sqrt{\frac{2h_{up}(R)k}{\alpha}}<1$ and $\sqrt{2h_{up}(R)k/\alpha}<1$,
\begin{align} \label{station1}
&\P\left(\max_{1 \le j \le N}\vert u_{R}(t,x_{j})\vert<\lambda\right)\nonumber\\
&\leq \P\left(\max_{1 \le j \le N}\vert U_{t,R}^{(\beta,n)}(x_{j})\vert<2\lambda\right)
+\P\left(\max_{1 \le j \le N}\vert U_{t,R}^{(\beta,n)}(x_{j})-u_{R}(t,x_{j})\vert \ge \lambda\right)\nonumber\\
&=\prod_{j=1}^{N}\left[1-\P\left(|U_{t,R}^{(\beta,n)}(x_{j})| \geq 2\lambda\right)\right]
+\P\left(\max_{1 \leq j \leq N}|U_{t,R}^{(\beta,n)}(x_{j})-u_{R}(t,x_{j})| \geq \lambda\right)\nonumber\\
&\le \left(1-\pi \e^{-(4\alpha t +2)k-2}\Big(M(t,C_{\sigma_{lo}},\epsilon_{0},\beta)h_{lo}(R)k\Big)^{2k}
\left(\frac{U+2\vert \sigma(0) \vert\sqrt{2kh_{up}(R)/\alpha}}{1-2\L_{\sigma} \sqrt{2kh_{up}(R)/\alpha}}\right)^{-4k}\right)^{N}\nonumber\\
&\qquad +\frac{N}{2}(\lambda/2)^{-k}\left(\frac{2\sqrt{2} t^{1/2}\sqrt{kh_{up}(R)}\e^{\alpha t-\sqrt{\alpha \beta t}}\left(\vert \sigma(0) \vert + \L_{\sigma} \frac{U+\vert \sigma(0) \vert \sqrt{2kh_{up}(R)/\alpha}}{1-\L_{\sigma}\sqrt{2kh_{up}(R)/\alpha}}\right)}{1-\L_{\sigma}\sqrt{2h_{up}(R)k/\alpha}}\right)^{k}\nonumber\\
&\qquad +\frac{N}{2}(\lambda/2)^{-k}\left((C_{\sigma_{up}}+2U)\frac{\left(\L_{\sigma}\sqrt{2h_{up}(R)k/\alpha}\right)^{n}}{1-\L_{\sigma}\sqrt{2h_{up}(R)k/\alpha}}\right)^{k}\,,
\end{align}
where $M(t,C_{\sigma_{lo}},\epsilon_{0},\beta)=4\pi^{2}tC_{\sigma_{lo}}^{2}(1-\epsilon_{0})^{2}(1-\e^{-\beta/2})^{2}$.
Take 
\begin{equation}\label{BigN}
N=\big\lfloor k^{C_{k}}N(k) \big\rfloor+1\,,
\end{equation}
for some finite positive constant $C_{k}<2$, where 
\begin{align}\label{BigNk}
N(k) =  \left(\frac{\e^{\alpha t+1/2}}{\sqrt{M(t,C_{\sigma_{lo}},\epsilon_{0},\beta)h_{lo}(R)k}}\frac{U+2\vert \sigma(0) \vert \sqrt{2kh_{up}(R)/\alpha}}{1-2\L_{\sigma} \sqrt{2kh_{up}(R)/\alpha}}\right)^{4k}.
\end{align}
Then 
\begin{align}\label{first term}
\left(1-\pi \e^{-(4\alpha t +2)k-2}\Big(M(t,C_{\sigma_{lo}},\epsilon_{0},\beta)h_{lo}(R)k\Big)^{2k}
\left(\frac{U+2\vert \sigma(0) \vert\sqrt{\frac{2kh_{up}(R)}{\alpha}}}{1-2\L_{\sigma} \sqrt{\frac{2kh_{up}(R)}{\alpha}}}\right)^{-4k}\right)^{N}
%&\leq \left(1-\pi \e^{-2} \cdot \frac{1}{N/k^{C_{k}}}\right)^{(N/k^{C_{k}})k^{C_{k}}}\nonumber\\
\leq \e^{-\pi\e^{-2}k^{C_{k}}}.
\end{align}
Take 
\begin{equation}\label{alpha}
\alpha=8\pi^{2}h_{up}(R)\left(\max\{1\,,\L_{\sigma}\}\right)^{2}k\,,
\end{equation}
(so \eqref{kalpha} and \eqref{kalpha2} are satisfied)
and 
\begin{equation}\label{beta}
\beta =4\alpha t.
\end{equation}
Since $\sqrt{2kh_{up}(R)/\alpha}\leq (2 \pi L_{\sigma})^{-1}$, we have
\begin{align}\label{simple1}
&\frac{2\sqrt{2} t^{1/2}\sqrt{kh_{up}(R)}\e^{\alpha t-\sqrt{\alpha \beta t}}\left(\vert \sigma(0) \vert + \L_{\sigma} \frac{U_{R}+\vert \sigma(0) \vert \sqrt{2kh_{up}(R)/\alpha}}{1-\L_{\sigma}\sqrt{2kh_{up}(R)/\alpha}}\right)}{1-\L_{\sigma}\sqrt{2h_{up}(R)k/\alpha}}\nonumber\\
&\hskip 1.5in \qquad \leq \frac{8\sqrt{2}\pi^{2}\left(\vert \sigma(0)\vert +\L_{\sigma}U\right) t^{1/2}\sqrt{kh_{up}(R)}}{(2\pi-1)^{2}}\e^{-8\pi^{2}h_{up}(R)t\left(\max\{1,\L_{\sigma}\}\right)^{2}k}\,,
\end{align}
and 
\begin{align}\label{simple2}
(C_{\sigma_{up}}+2U)\frac{\left(\L_{\sigma}\sqrt{2h_{up}(R)k/\alpha}\right)^{n}}{1-\L_{\sigma}\sqrt{2h_{up}(R)k/\alpha}}
\leq \frac{2\pi}{2\pi-1} \cdot (C_{\sigma_{up}}+2U)(2\pi)^{-n}.
\end{align}
Take
\begin{align}\label{smalln}
n=\log R.
\end{align}
Then for every $0<t<\infty$, there exists a finite positive $R_{n}(t)$ such that  for all $R \ge R_{n}(t)$,
\begin{align}\label{second and third term}
&\frac{2\pi}{2\pi-1} \cdot (C_{\sigma_{up}}+2U)(2\pi)^{-n} \leq \frac{8\sqrt{2}\pi^{2}\left(\vert \sigma(0)\vert +\L_{\sigma}U\right) t^{1/2}\sqrt{kh_{up}(R)}}{(2\pi-1)^{2}}\e^{-8\pi^{2}h_{up}(R)t\left(\max\{1,\L_{\sigma}\}\right)^{2}k}.
\end{align}
Use \eqref{first term}, \eqref{second and third term}, \eqref{simple1}, \eqref{simple2} in \eqref{station1} to get under the constraints of \eqref{independence condition}, for all $0<t<\infty$, $0<\epsilon_{0}<1$ and $R \ge \max\{R_{n}(t),R(\epsilon_{0},t)\}$,
\begin{align}\label{station2}
&\P\left(\max_{1 \le j \le N}\vert u_{R}(t,x_{j})\vert<\lambda\right)\nonumber\\
&\leq \e^{-\pi\e^{-2}k^{C_{k}}}+N(\lambda/2)^{-k}\left(\frac{8\sqrt{2}\pi^{2}\left(\vert \sigma(0)\vert +\L_{\sigma}U\right) t^{1/2}\sqrt{kh_{up}(R)}}{(2\pi-1)^{2}}\e^{-8\pi^{2}h_{up}(R)t\left(\max\{1,\L_{\sigma}\}\right)^{2}k}\right)^{k}.
\end{align}
Take 
\begin{equation}\label{lambda}
\lambda = \sqrt{h_{lo}(R)t/\e}C_{\sigma_{lo}}(1-\epsilon_{0})(1-\e^{-\beta/2})\sqrt{k}\,,
\end{equation}
then \eqref{klambda} is satisfied.\newline
Use \eqref{lambda} in \eqref{station2} to get under the constraints of \eqref{independence condition}, for all $0<t<\infty$, $0<\epsilon_{0}<1$ and $R \ge \max\{R_{n}(t),R(\epsilon_{0},t)\}$,
\begin{align}\label{station3}
&\P\left(\max_{1 \le j \le N}\vert u_{R}(t\,,x_{j})\vert<\lambda \right)\nonumber\\
&\leq\exp(-\pi\e^{-2}k^{C_{k}})\nonumber\\
&\qquad +N\left(\frac{h_{up}(R)}{h_{lo}(R)}\right)^{k/2}M_{2}(U,C_{\sigma_{lo}},\epsilon_{0},\beta)^{k} \exp\left(-8\pi^{2}h_{up}(R)t\left(\max\{1,\L_{\sigma}\}\right)^{2}k^{2} \right).
\end{align}
where $M_{2}(U,C_{\sigma_{lo}},\epsilon_{0},\beta)=\frac{16\sqrt{2}\pi^{2}\sqrt{\e}\left(\vert \sigma (0)\vert+\L_{\sigma}U\right)}{\left(2 \pi -1\right)^{2}C_{\sigma_{lo}}(1-\epsilon_{0})(1-\e^{-\beta/2})}$. By \eqref{BigN} and \eqref{BigNk}, for every $0<t<\infty$, there exists a finite positive $R_{N}(t)$ such that for all $R\ge R_{N}(t)$,
\begin{equation}\label{inequalityN}
N \le 2 k^{(4C_{k}-2)k}h_{lo}(R)^{-2k}\e^{(4\alpha t+2)k}.
\end{equation}
Hence, by \eqref{alpha}, under the constraint of \eqref{independence condition}, for all $0<t<\infty$, $0<\epsilon_{0}<1$ and $R \ge \max\{R_{n}(t),R_{N}(t),R(\epsilon_{0},t)\}$,
\begin{align}\label{station3.5}
\P\left(\max_{1 \le j \le N}\vert u_{R}(t\,,x_{j})\vert<\lambda \right)
&\leq\exp(-\pi\e^{-2}k^{C_{k}})\nonumber\\
&\qquad +2\Big(M_{2}(t,C_{\sigma_{lo}},\epsilon_{0},\beta)^{k}k^{(4C_{k}-2)k}h_{lo}(R)^{-5k/2}h_{up}(R)^{k/2}R^{-2k}\nonumber\\ 
&\qquad\qquad\times\exp\left(-8\pi^{2}h_{up}(R)t\left(\max\{1,\L_{\sigma}\}\right)^{2}k^{2}+2k \right)\Big).
\end{align}
\eqref{beta},\eqref{smalln},\eqref{inequalityN} imply that it suffices to have
\begin{equation}\label{kcondition}
4t\left(\log R\right) \sqrt{\alpha} k^{(4C_{k}-2)k}h_{lo}(R)^{-2k} \e^{(4\alpha t +2)k}< \pi R,
\end{equation} in order for \eqref{independence condition} to hold.
Take
\begin{align}\label{k}
k&=\Bigg\lfloor\frac{ (\log R)^{1/2-C_{h_{up}}/4}}{2\sqrt{2}\pi \sqrt{(4+\epsilon_{\alpha})t}\left(\max\{1,\L_{\sigma}\}\right)}\Bigg\rfloor\,,
\end{align}
where $\epsilon_{\alpha}$ is a finite positive constant.
Then 
\begin{equation}\label{kcondition2}
\e^{8\pi^{2}(4+\epsilon_{\alpha})th_{up}(R)\left(\max\{1\,,\L_{\sigma}\}\right)^{2} k^{2}} < R.
\end{equation}
Note that \eqref{kcondition2} implies \eqref{kcondition} by \eqref{alpha}.
Hence, by choosing a larger $R_{N}(t)$ if necessary, for all $0<t<\infty$, $0<\epsilon_{0}<1$ and $R \ge \max\{R_{n}(t),R_{N}(t),R(\epsilon_{0},t)\}$, \eqref{independence condition} is satisfied. Since $C_{k}< 2$, by choosing a larger $R_{N}(t)$ (hence a larger $k$) if necessary, we have for all $0<t<\infty$, $0<\epsilon_{0}<1$ and $R \ge \max\{R_{n}(t),R_{N}(t),R(\epsilon_{0},t)\}$,
\begin{align}\label{simplifyinequality}
&2\Big(M_{2}(t,C_{\sigma_{lo}},\epsilon_{0},\beta)^{k}k^{(4C_{k}-2)k}h_{lo}(R)^{-5k/2}h_{up}(R)^{k/2}R^{-2k}\\ 
&\qquad\times\exp\left(-8\pi^{2}h_{up}(R)t\left(\max\{1,\L_{\sigma}\}\right)^{2}k^{2}+2k \right)\Big)
\le \exp\left(-\pi\e^{-2}k^{C_{k}}\right)\nonumber.
\end{align}
By \eqref{lambda}, \eqref{station3.5} and \eqref{simplifyinequality}, for all $0<t<\infty$, $0<\epsilon_{0}<1$, finite positive constant $C_{k}<2$, and $R \ge \max\{R_{n}(t),R_{N}(t),R(\epsilon_{0},t)\}$,
\begin{equation}\label{convergencerate}
\P\left(\max_{1 \le j \le N}\vert u_{R}(t\,,x_{j})\vert<\sqrt{h_{lo}(R)t/\e}C_{\sigma_{lo}}(1-\epsilon_{0})(1-\e^{-\beta/2})\sqrt{k}\right)\leq 2 \exp\left(-\pi\e^{-2}k^{C_{k}}\right).
\end{equation}
By \eqref{beta} and \eqref{k}, for every $0<t<\infty$, $0<\epsilon_{0}<1$, there exists a finite positive constant $C(t,\epsilon_{0})$ such that 
\begin{equation}
\lim_{R \to \infty} \P\left( \sup_{x \in  S_{R}^{2}} \vert u_{R}(t\,,x)\vert \ge C(t,\epsilon_{0})\left(\log R\right)^{1/4+C_{h_{lo}}/4-C_{h_{up}}/8}\right) =1.
\end{equation}
By Theorem \ref{spacecontinuityfixR}, the results can be restated as every all $0<t<\infty$, $0<\epsilon_{0}<1$, there exists a finite positive constant $C(t,\epsilon_{0})$ such that 
\begin{equation}
\lim_{R \to \infty} \P\left( \exists x \in S_{R}^{2}: \vert u_{R}(t\,,x)\vert \ge C(t,\epsilon_{0})\left(\log R\right)^{1/4+C_{h_{lo}}/4-C_{h_{up}}/8}\right) =1.
\end{equation}
Moreover, by \eqref{convergencerate}, for every $0<t<\infty$, $0<\epsilon_{0}<1$, finite positive constant $C$, there exist finite positive constants $C(t,\epsilon_{0})$ and $R(t,\epsilon_{0},C)$ such that for $R \ge R(t,\epsilon_{0},C)$,
\begin{equation}
\P\left(\exists x \in S_{R}^{2}: \vert u_{R}(t\,,x)\vert \ge  C(t,\epsilon_{0})\left(\log R\right)^{1/4+C_{h_{lo}}/4-C_{h_{up}}/8}\right)\ge1-R^{-C\pi \e^{-2}(1/2-C_{h_{up}}/4)}.
\end{equation}
\end{proof}

\section{Tail probability estimates}
Assume throughout this section that $u_{R,0}(x)=0$. Then $u_{R}(t\,,x)$ has mean zero and is subgaussian. In this section, we give a tail probability estimate of $u_{R}(t\,,x)$, which will be useful when we derive an asymptotic upper bound for $\sup_{x \in S_{R}^{2}}\vert u_{R}(t\,,x) \vert$ in the next section.
\begin{lemma}
Then for any $0<t,R<\infty$, $x\in S_{R}^{2}$, 
\begin{equation}
\Var(u_{R}(t,x)) \le  h_{up}(R)tC_{\sigma_{up}}^{2}.
\end{equation}
\end{lemma}
\begin{proof}
By Lemma \ref{kernelaspdf},
\begin{align}
\Var(u_{R}(t,x))&\le C_{\sigma_{up}}^{2}h_{up}(R)\int_{0}^{t}\d s\int_{S_{R}^{2} \times S_{R}^{2}}p_{R}\big(s\,,\theta(x\,,y_{1})\big)p_{R}\big(s\,,\theta(x\,,y_{2})\big)\d y_{1} \d y_{2}\nonumber\\
&= h_{up}(R)tC_{\sigma_{up}}^{2}.
\end{align} 
\end{proof}

\begin{lemma}\label{moment bound}
For any $0<t,R<\infty$, $x\in S_{R}^{2}$, any integer $k \ge 2$,
\begin{equation}
\E \big[\vert u_{R}(t\,,x)\vert^{k}\big] \leq \left(2C_{\sigma_{up}}\sqrt{h_{up}(R)t}\right)^{k}k^{k/2}.
\end{equation}
\end{lemma}
\begin{proof}
By Lemma \ref{kernelaspdf}, Carlen-Kr\'ee's optimal bound on the Burkholder-Gundy-Davis inquality \cite{Carleen},
\begin{align}
\E \big[\vert u_{R}(t\,,x)\vert^{k}\big] \leq \big(2\sqrt{k}\big)^{k}\big\|\sqrt{\Var(u_{R}((t\,,x))}\big\|_{k}^{k}\leq \left(2C_{\sigma_{up}}\sqrt{h_{up}(R)tk}\right)^{k}.
\end{align}
\end{proof}

\begin{lemma}\label{expmoment}
For all $0<t\,,R\,,\lambda<\infty$, $x\in S_{R}^{2}$, 
\begin{equation}
\E \big[\exp\big(\lambda u_{R}(t\,,x)\big)\big] < \left(1+2\lambda C_{\sigma_{up}}\sqrt{h_{up}(R)t\e}\right) \exp\left(4C_{\sigma_{up}}^{2}h_{up}(R)t\e \lambda^{2}\right).
\end{equation}
\end{lemma}
\begin{proof}
By Lemma \ref{moment bound},
\begin{align}
\E\big[\exp\big(\lambda u_{R}(t\,,x)\big)\big]&=\E\left[\sum_{k=0}^{\infty}\frac{\lambda^{k}\big(u_{R}(t\,,x)\big)^{k}}{k!}\right]\le1+\sum_{k=2}^{\infty}\frac{\left(2\lambda C_{\sigma_{up}}\sqrt{h_{up}(R)t}\right)^{k}k^{k/2}}{k!}.
\end{align}
By Sterling's estimation,
\begin{equation}
\sqrt{2\pi}k^{k+\frac{1}{2}}\e^{-k} \leq k!.
\end{equation}
This implies 
\begin{equation}
k^{k/2} \leq \frac{\e^{k/2}\sqrt{k!}k^{-1/4}}{(2\pi)^{1/4}}.
\end{equation}
So now
\begin{align}
\E\big[\exp\big(\lambda u_{R}(t\,,x)\big)\big]&\leq 1+\sum_{k=2}^{\infty}\frac{\left(2\lambda C_{\sigma_{up}}\sqrt{h_{up}(R)t}\right)^{k}\e^{k/2}k^{-1/4}}{(2\pi)^{1/4}\sqrt{k!}}\nonumber\\
&<1+\sum_{k=2}^{\infty}\frac{\left(2\lambda C_{\sigma_{up}}\sqrt{h_{up}(R)t\e}\right)^{k}}{\sqrt{k!}}\nonumber\\
%&=1+\sum_{n=1}^{\infty}\frac{\left(2C_{\sigma_{up}}\sqrt{h_{up}(R)t\e}\lambda\right)^{2n}}{\sqrt{(2n)!}}+\sum_{n=1}^{\infty}\frac{\left(2C_{\sigma_{up}}\sqrt{h_{up}(R)t\e}\lambda\right)^{2n+1}}{\sqrt{(2n+1)!}}\nonumber\\
%&<1+\left(1+\left(2C_{\sigma_{up}}\sqrt{h_{up}(R)t\e}\lambda\right)\right)\sum_{n=1}^{\infty}\frac{\left(2C_{\sigma_{up}}\sqrt{h_{up}(R)t\e}\lambda\right)^{2n}}{\sqrt{(2n)!}}\nonumber\\
%&<1+\left(1+\left(2C_{\sigma_{up}}\sqrt{h_{up}(R)t\e}\lambda\right)\right)\sum_{n=1}^{\infty}\frac{\left(4C_{\sigma_{up}}^{2}h_{up}(R)t\e \lambda^{2}\right)^{n}}{n!}\nonumber\\
&<\left(1+\left(2\lambda C_{\sigma_{up}}\sqrt{h_{up}(R)t\e}\right)\right)\exp\left(4C_{\sigma_{up}}^{2}h_{up}(R)t\e \lambda^{2}\right).
\end{align}
\end{proof}

\begin{theorem}\label{tailprobability}
For any $0<t,R<\infty$, $x\in S_{R}^{2}$, $M>4C_{\sigma_{up}}\sqrt{h_{up}(R)t\e}$, 
\begin{equation}
\sup_{x \in S_{R}^{2}}\P\big(\vert u_{R}(t\,,x)\vert> M\big)< \left(\frac{M}{2C_{\sigma_{up}}\sqrt{h_{up}(R)t\e}}\right)\exp\left(-\frac{M^{2}}{16C_{\sigma_{up}}^{2}h_{up}(R)t\e}\right).
\end{equation}
\end{theorem}
\begin{proof}
By Markov's inequality and Lemma \ref{expmoment}, for all $0<t\,,R\,,\lambda<\infty$ and $x \in S_{R}^{2}$,
\begin{align}
\P\big(\vert u_{R}(t\,,x)\vert> M\big)
&\leq \e^{-M\lambda} \E\left[\exp\left(\lambda u_{R}(t\,,x)\right)\right]\nonumber\\
&< \left(1+\left(2C_{\sigma_{up}}\sqrt{h_{up}(R)t\e}\lambda\right)\right)\exp\left(4C_{\sigma_{up}}^{2}h_{up}(R)t\e \lambda^{2}-M\lambda\right).
\end{align}
Take $\lambda=M\big(8C_{\sigma_{up}}^{2}h_{up}(R)t\e\big)^{-1}$ to get for $M>4C_{\sigma_{up}}\sqrt{h_{up}(R)t\e}$,
\begin{align}
\P\big(\vert u_{R}(t\,,x)\vert> M\big)
&< \left(1+\frac{M}{4C_{\sigma_{up}}\sqrt{h_{up}(R)t\e}}\right)\exp\left(-\frac{M^{2}}{16C_{\sigma_{up}}^{2}h_{up}(R)t\e}\right)\nonumber\\
&< \left(\frac{M}{2C_{\sigma_{up}}\sqrt{h_{up}(R)t\e}}\right)\exp\left(-\frac{M^{2}}{16C_{\sigma_{up}}^{2}h_{up}(R)t\e}\right).
\end{align}
Take the supremum on the left-hand side to finish the proof.
\end{proof}

\section{An asymptotic upper bound of the supremum of the mild solution and the proof of the main theorem}
In this section, we prove that for any fixed $t>0$, there exists a constant $C>0$ the probability of the event $\{\sup_{x \in S_{R}^{2}} \vert u_{R}(t\,,x)\vert> C\sqrt{\log R}\}$ is small as $R$ tends to $\infty$, hence obtaining an asymptotic upper bound of $\sup_{x \in S_{R}^{2}} \vert u_{R}(t\,,x)\vert$. The main result of this section is the following.
\begin{theorem}\label{maintheoremlowerbound}
For every $0<t<\infty$, there exists a positive constant $C$ such that
\begin{equation}
\lim_{R \to \infty} \P\left(\exists x \in S_{R}^{2} \text{ such that } \vert u_{R}(t,x) \vert \ge C (\log R)^{1/2+C_{h_{up}}/4}\right)=0.
\end{equation}
Moreover, for every  $2-2^{1/4}<r<1$, $\log_{2}(2-r)<\gamma<1/3$, $0<t<\infty$, there exists a finite positive $R(t,r,\gamma)$ such that for $R \ge R(t,r,\gamma)$ and $32C_{\sigma_{up}}\sqrt{t\e}\log_{2-r}(2)<C<\infty$,
\begin{align}
&\P\left(\exists x \in S_{R}^{2} \text{ such that } \vert u_{R}(t,x) \vert \ge C (\log R)^{1/2+C_{h_{up}}/4}\right) \nonumber\\
&\leq \frac{C\big(\log (2-r)\big)^{1/2}}{C_{\sigma_{up}}\sqrt{t\e}}\big(\log_{2-r}(R)\big)^{1/2}\exp\left(\big(\log_{2-r}(R)\big)\left(\log (16)-\frac{C^{2}\log (2-r)}{256C_{\sigma_{up}}^{2}t\e}\right)\right).
\end{align}
\end{theorem}
As a result, we have obtained: 
\begin{proof}[Proof of Theorem \ref{maintheorem}]
Theorem \ref{maintheorem} is obtained by combining Theorem \ref{BigR} and Theorem \ref{maintheoremlowerbound}.
\end{proof}
We begin by establishing some notations.
Suppose $C>0$ is a constant. For each $t>0$, $R>0$, $\alpha>0$, $\gamma>0$, $0<r<1$ and positive integer $k >1$, positive integer $n$, define sets as follows.
\begin{definition}
\begin{align*}
A_{t,R}=\Big\{\exists x \in S_{R}^{2} \text{ such that } \vert u_{R}(t,x) \vert \ge C\left(\log R\right)^{1/2+C_{h_{up}}/4}\Big\}\,,
\end{align*}
\begin{align*}
A_{t,R,n,\alpha}=\Big\{\exists x \in S_{R}^{2} \text{ such that } \vert u_{R}(t,x) \vert> C\left(\log R\right)^{1/2+C_{h_{up}}/4}-2^{-\alpha n}\Big\}\,,
\end{align*}
\begin{align*}
G_{R,n}&=\Big\{x \in S_{R}^{2}: x=\big(R\sin(i_{1}\pi 4^{-n})\cos(2i_{2}\pi 4^{-(n+1)})\,,R\sin(i_{1}\pi 4^{-n})\sin(2i_{2}\pi 4^{-(n+1)})\,,\\ 
&\qquad R\cos(i_{1}\pi 4^{-n})\big) \text{ for some $i_{1}\,,i_{2} \in \Z$}\Big\}\,,
\end{align*}
\begin{align*}
L_{t,R,n,\alpha}&=\Big\{\exists x \in S_{R}^{2} \text{ such that } \vert u_{R}(t,x) \vert \ge C\left(\log R\right)^{1/2+C_{h_{up}}/4}\\ &\qquad \text{ and for all } x \in G_{R,n}, \vert u_{R}(t,x) \vert \le  C\left(\log R\right)^{1/2+C_{h_{up}}/4}-2^{-\alpha n}\Big\}\,,
\end{align*}
\begin{align*}
K_{t,R,n,\alpha}=\Big\{\exists x \in G_{R,n} \text{ such that } \vert u_{R}(t,x) \vert> C\left(\log R\right)^{1/2+C_{h_{up}}/4}-2^{-\alpha n}\Big\}\,,
\end{align*}
\begin{align*}
C_{r,t,n,\gamma}=\Bigg\{\sup_{\theta(x,x') \leq \pi 2^{-n}}\vert u_{(2-r)^{n}}(t\,, x)-u_{(2-r)^{n}}(t\,, x')\vert \leq \pi \left(\frac{2-r}{2^\gamma}\right)^{n}\Bigg\}.
\end{align*}
%Note that all of the above except $G_{R,n}$ are random.
\end{definition}
We record the following result which will be used in the proof of the main theorem of this section.
\begin{lemma}\label{probabilityofsmalldifference2}
For every $0<t<\infty$, $0<q<\frac{1}{3}$, $2-2^{1/4}<r<1$, there exists a finite positive $n(t,q,r)$ such that for all $n\ge n(t,q,r)$, and $0<\gamma<\infty$,
\begin{align}
&\P\left( \sup_{\theta(x,x') \le \pi 2^{-n}}\big\vert u_{(2-r)^{n}}(t\,, x)-u_{(2-r)^{n}}(t\,, x')\big\vert \ge  \pi (2-r)^{n}2^{-n\gamma}\right)\nonumber\\
&\hskip 1.5in \qquad\leq \pi^{-3}\left(12288\sqrt{2}\big(\log\left(2-r\right)\big)^{C_{h_{up}}/4}C_{\sigma_{up}}(3/2)^{1/3}\pi^{7/3}2^{q}\right)^{n}\nonumber\\
&\hskip 1.7in \qquad \times n^{(1/2+C_{h_{up}}/4)n}\left(\frac{2^{\gamma-q}}{(2-r)^{2/3}}\right)^{n^{2}}.
\end{align}
\end{lemma}
\begin{proof}
Choose $\epsilon_{0}=1/2$, $a=1$, $R=(2-r)^{n}$ in Theorem \ref{probabilityofsmalldifference}.
\end{proof}

We are now ready to prove Theorem \ref{maintheoremlowerbound}.
\begin{proof}[Proof of Theorem \ref{maintheoremlowerbound}]
Assume throughout the proof that $2-2^{1/4}<r<1$, and $\log_{2}(2-r)<\gamma<1/3$.  \\
For every $0<t,\alpha<\infty$ and positive integer, on $L_{t,(2-r)^{n},n,\alpha}$, there exists $x \in S_{(2-r)^{n}}^{2}$ such that $\vert u_{(2-r)^{n}}(t,x) \vert \ge C\left(n\log (2-r)\right)^{1/2+\C_{h_{up}}/4}$  and for all $y \in G_{(2-r)^{n},n}$, $\vert u_{t,(2-r)^{n}}(y) \vert \le  C\left(n\log (2-r)\right)^{1/2+C_{h_{up}}/4}-2^{-\alpha n}$. For all positive integer $n$, there exists $y \in G_{(2-r)^{n},n}$ such that $\theta(x\,,y)<\pi (2-r)^{n}4^{-n}$.
Hence for every $0<t\,,\alpha<\infty$, positive integer $n$, on $L_{t,(2-r)^{n},n,\alpha} \cap C_{r,t,n,\gamma} $, there exist $x \in S_{(2-r)^{n}}^{2}$ and $y \in G_{(2-r)^{n},n}$, such that
\begin{align}
2^{-\alpha n} \le \big\vert u_{t,(2-r)^{n}}(x)\big\vert -\big\vert u_{t,(2-r)^{n}}(y)\big\vert \le \pi \left(\frac{2-r}{2^\gamma}\right)^{n}.
\end{align}
Choose any fixed $0<\alpha < \gamma-\log_{2}\left(2-r\right)$. Then for $n \ge \big\lfloor \big(\gamma-\log_{2}(2-r)-\alpha\big)(\log_{2}\pi)^{-1} \big\rfloor+1$,  
\begin{equation}\label{alphabetaRn}
2^{-\alpha n} > \pi \left(\frac{2-r}{2^\gamma}\right)^{n}.
\end{equation} 
This implies for every $0<t<\infty$, $0<\alpha < \gamma-\log_{2}\left(2-r\right)$, positive integer \\$n \ge \big\lfloor \big(\gamma-\log_{2}(2-r)-\alpha\big)(\log_{2}\pi)^{-1} \big\rfloor+1$,
\begin{equation}\label{Lsmall}
\P\left(L_{t,(2-r)^{n},n,\alpha} \cap C_{r,t,n,\alpha}\right) = \P\left( 2^{-\alpha n} \le \pi\left(\frac{2-r}{2^\gamma}\right)^{n} \right)=0.
\end{equation}
On $K_{t,(2-r)^{n},n,\alpha}$, there exists $x \in G_{(2-r)^{n},n}$ such that $\big\vert u_{t,(2-r)^{n}}(x) \big\vert \ge C\big(n\log (2-r)\big)^{1/2+C_{h_{up}}/4}\\ -2^{-\alpha n}$.
By Theorem \ref{tailprobability}, for every $0<t\,,\alpha<\infty$ and positive integer $n > \max\Bigg\{\inf \Big\{n \in \Z: C\big(n\log (2-r)\big)^{1/2+C_{h_{up}}/4} \ge 4U\Big\}\,, \inf \Big\{n \in \Z: 2^{1-\alpha n} < C \big(n\log (2-r)\big)^{1/2+C_{h_{up}}/4}\Big\}\Bigg\}$,
\begin{align}\label{ineq1}
&\P\left(K_{t,(2-r)^{n},n,\alpha} \cap C_{r,t,n,\gamma}\right) \nonumber\\
%&\leq \P\left(K_{t,(2-r)^{n},n,\alpha}\right) \nonumber\\
&\hskip 0.5in \qquad \leq 2^{4n+2} \sup_{x \in S_{(2-r)^{n}}^{2}}\P\left(\Big\vert u_{(2-r)^{n}}(t\,,x) \Big\vert > C\big(n\log (2-r)\big)^{1/2+C_{h_{up}}/4}-2^{-\alpha n} \right)\nonumber\\
%&\leq 2^{4n+2} \sup_{x \in S_{(2-r)^{n}}^{2}}\P\left(\Big\vert u_{(2-r)^{n}}(t\,,x)\Big\vert >\frac{C\left(n\log (2-r)\right)^{\frac{1}{2}+\frac{C_{h_{up}}}{4}}}{2} \right)\nonumber\\
&\hskip 0.5in\qquad \leq  2^{4n+2} \sup_{x \in S_{(2-r)^{n}}^{2}}\P\left(\Big\vert u_{(2-r)^{n}}(t\,,x)- u_{(2-r)^{n},0}(x) \Big\vert>\frac{C\left(n\log (2-r)\right)^{1/2+C_{h_{up}}/4}}{4} \right)\nonumber\\
&\hskip 0.5in\qquad\qquad + 2^{4n+2} \sup_{x \in S_{(2-r)^{n}}^{2}}\P\left(\Big\vert  u_{(2-r)^{n},0}(x) \Big\vert>\frac{C\big(n\log (2-r)\big)^{1/2+C_{h_{up}}/4}}{4} \right)\nonumber\\
%&\leq 2^{4n}\frac{C\left(n\log (2-r)\right)^{\frac{1}{2}}}{2C_{\sigma_{up}}\sqrt{t\e}}\exp\left(-\frac{C^{2}n\log (2-r)}{256C_{\sigma_{up}}^{2}te}\right)\nonumber\\
%&\qquad +2^{4n+2}\P\left(U_{0}>\frac{C\left(n\log (2-r)\right)^{\frac{1}{2}+\frac{C_{h_{up}}}{4}}}{4}\right)\nonumber\\
&\hskip 0.5in\qquad =\frac{C\big(\log (2-r)\big)^{\frac{1}{2}}}{2C_{\sigma_{up}}\sqrt{t\e}}n^{1/2}\exp\left(n\left(\log (16)-\frac{C^{2}\log (2-r)}{256C_{\sigma_{up}}^{2}t\e}\right)\right).
\end{align}
By \eqref{Lsmall}, \eqref{ineq1} and Lemma \ref{probabilityofsmalldifference2}, for every $0<t<\infty$, $0<q<1/3$, $0<\alpha < \gamma-\log_{2}\left(2-r\right)$, positive integer $n > \max\Bigg\{\inf \Big\{n \in \Z: C\big(n\log (2-r)\big)^{1/2+C_{h_{up}}/4} \ge 4U\Big\}\,, \inf \Big\{n \in \Z: 2^{1-\alpha n} < C \big(n\log (2-r)\big)^{1/2+C_{h_{up}}/4}\Big\}\,, \big\lfloor \big(\gamma-\log_{2}(2-r)-\alpha\big)(\log_{2}\pi)^{-1} \big\rfloor+1\Bigg\}$,
\begin{align}
&\P\left(A_{t,(2-r)^{n}} \ne \emptyset \right) \\
&\leq \P\big(K_{t,(2-r)^{n},n,\alpha} \cap C_{r,t,n,\gamma}\big) + \P\big(L_{t,(2-r)^{n},n,\alpha} \cap C_{r,t,n,\gamma}\big) + \P\big(C_{r,t,n,\gamma}^{c}\big)\nonumber\\
&\leq \P\big(K_{t,(2-r)^{n},n,\alpha} \cap C_{r,t,n,\gamma}\big) + \P\big(C_{r,t,n,\gamma}^{c}\big)\nonumber\\
&\leq \frac{C\big(\log (2-r)\big)^{\frac{1}{2}}}{2C_{\sigma_{up}}\sqrt{t\e}}n^{1/2}\exp\left(n\left(\log (16)-\frac{C^{2}\log (2-r)}{256C_{\sigma_{up}}^{2}te}\right)\right)\nonumber\\
&\qquad +\pi^{-3}\left(12288\sqrt{2}\left(\log\big(2-r\big)\right)^{C_{h_{up}}/4}C_{\sigma_{up}}(3/2)^{1/3}\pi^{7/3}2^{q}\right)^{n} n^{(1/2+C_{h_{up}}/4)n}\left(\frac{2^{\gamma-q}}{(2-r)^{2/3}}\right)^{n^{2}}.\nonumber
\end{align}
Choose and fix $q \in (\gamma\,,1/3)$, $0<\alpha < \gamma-\log_{2}\left(2-r\right)$. Then for every $0<t<\infty$, there exists a finite positive $n(t)$ such that for all positive integer $n \ge n(t)$,
\begin{align}
\P\left(A_{t,(2-r)^{n}} \ne \emptyset \right) \leq \frac{C\big(\log (2-r)\big)^{\frac{1}{2}}}{C_{\sigma_{up}}\sqrt{t\e}}n^{1/2}\exp\left(n\left(\log (16)-\frac{C^{2}\log (2-r)}{256C_{\sigma_{up}}^{2}t\e}\right)\right).
\end{align}
In terms of $R$, the above can be restated as: for every $0<t<\infty$, there exists a finite positive $R(t)$ such that for $R \ge R(t)$, 
\begin{align}
&\P\left(A_{t,R} \ne \emptyset \right) \nonumber\\
&\leq \frac{C\big(\log (2-r)\big)^{\frac{1}{2}}}{C_{\sigma_{up}}\sqrt{t\e}}\big(\log_{2-r}(R)\big)^{1/2}\exp\left(\left(\log_{2-r}(R)\right)\left(\log (16)-\frac{C^{2}\log (2-r)}{256C_{\sigma_{up}}^{2}te}\right)\right).
\end{align}
This implies that for every $0<t<\infty$, $32C_{\sigma_{up}}\sqrt{t\e}\log_{2-r}(2)<C<\infty$,
\begin{equation}
\lim_{R \to \infty} \P\left( \exists x \in S_{R}^{2} \text{ such that } \vert u_{R}(t,x) \vert \ge C (\log R)^{1/2+C_{h_{up}}/4}\right)=0.
\end{equation}
\end{proof}

\break

\appendix
\section{Appendix: Garsia's theorem}
We follow the arguments in \cite{DavarCBMS} to give the proof of the Garsia's theorem in the spherical context. Relevant notations and symbols are defined in Section 4.
\begin{proof}[Proof of Lemma \ref{GarsiaLemma1.1}]
By Jensen's inequality,
\begin{align}
&\big\vert \bar{f}_{l+1,k}(x) - \bar{f}_{l,k}(x) \big\vert^{k}\nonumber\\
&\hskip 0.5in \qquad=\Bigg\vert \frac{1}{\big\vert B_{R}(r_{l+1}(k))\big\vert} \int_{B_{R}(r_{l+1}(k))}f(x\tilde{+}z) \d z  - \frac{1}{\big\vert B_{R}(r_{l}(k))\big\vert} \int_{B_{R}(r_{l}(k))}f(x\tilde{+}z) \d z\Bigg\vert ^{k}\nonumber\\
&\hskip 0.5in \qquad=\Bigg\vert \frac{1}{\big\vert B_{R}(r_{l+1}(k))\big\vert \cdot \vert B_{R}(r_{l}(k))\vert} \int_{B_{R}(r_{l+1}(k))} \d z \int_{B_{R}(r_{l}(k))} \d y 
\big(f(x\tilde{+}z)-f(x\tilde{+}y)\big)\Bigg\vert^{k}\nonumber\\
&\hskip 0.5in \qquad\leq \frac{1}{\big\vert B_{R}(r_{l+1}(k))\big\vert^{2}} \int_{B_{R}(r_{l+1}(k))} \d z \int_{B_{R}(r_{l}(k))} \d y 
\big\vert f(x\tilde{+}z)-f(x\tilde{+}y) \big\vert^{k}.
\end{align}
For $\alpha> \sup_{z \in B_{R}(r_{l+1}(k))}\sup_{y \in B_{R}(r_{l}(k))} \mu_{k}(R\theta(z\,,y))$,
\begin{align}
&\int_{B_{R}(r_{l+1}(k))} \d z \int_{B_{R}(r_{l}(k))} \d y \big\vert f(x\tilde{+}z)-f(x\tilde{+}y) \big\vert^{k}\nonumber\\
&\hskip 2in \qquad \leq \alpha^{k}\int_{B_{R}(r_{l+1}(k))} \d z \int_{B_{R}(r_{l}(k))} \d y 
\frac{\big\vert f(x\tilde{+}z)-f(x\tilde{+}y) \big\vert^{k}}{\big\vert \mu_{k}(R\theta(z\,,y))\big\vert^{k}}\nonumber\\
&\hskip 2in \qquad \leq \alpha^{k}I_{k}.
\end{align}
Hence,
\begin{align}
\big\vert \bar{f}_{l+1,k}(x) - \bar{f}_{l,k}(x) \big\vert^{k} \leq \frac{\alpha^{k}I_{k}}{\big\vert B_{R}(r_{l+1}(k))\big\vert^{2}}.
\end{align}
Let $\alpha$ converges to  $\sup_{z \in B_{R}(r_{l+1}(k))}\sup_{y \in B_{R}(r_{l}(k))} \mu_{k}(R\theta(z\,,y))$. Then
\begin{align}
\big\vert \bar{f}_{l+1,k}(x) - \bar{f}_{l,k}(x) \big\vert \leq \frac{\left(\sup_{z \in B_{R}(r_{l+1}(k))}\sup_{y \in B_{R}(r_{l}(k))} \mu_{k}\big(R\theta(z\,,y)\big)\right)I_{k}^{1/k}}{\big\vert B_{R}(r_{l+1}(k))\big\vert^{2/k}}.
\end{align}
Note that 
\begin{align}
\sup_{z \in B_{R}(r_{l+1}(k))}\sup_{y \in B_{R}(r_{l}(k))} \mu_{k}(R\theta(z\,,y)) &\leq \sup_{z \in B_{R}(r_{l+1}(k))}\sup_{y \in B_{R}(r_{l}(k))} \mu_{k}\big(R\theta(z\,,N)+R\theta(y\,,N)\big)\nonumber\\
&\leq \mu_{k}\left( r_{l+1}+r_{l}\right)\nonumber\\
&\leq \mu_{k}(2r_{l}).
\end{align}
So now 
\begin{equation}
\big\vert \bar{f}_{l+1,k}(x) - \bar{f}_{l,k}(x) \big\vert \leq \frac{\mu_{k}(2r_{l})I_{k}^{1/k}}{\big\vert B_{R}(r_{l+1}(k))\big\vert^{2/k}}.
\end{equation}
For any positive integer $L$,
\begin{align}
\big\vert \bar{f}_{l+L,k}(x) - \bar{f}_{l,k}(x) \big\vert &\leq \sum_{n=l}^{l+L-1}\big\vert \bar{f}_{n+1,k}(x) - \bar{f}_{n,k}(x) \big\vert \nonumber\\
&\leq  I_{k}^{1/k}\sum_{n=l}^{\infty}\frac{\mu_{k}(2r_{n})}{\big\vert B_{R}(r_{n+1}(k))\big\vert^{2/k}}\nonumber\\
&\leq  C_{\mu_{k}}I_{k}^{1/k}\sum_{n=l}^{\infty}\frac{\mu_{k}(r_{n})}{\big\vert B_{R}(r_{n+1}(k))\big\vert^{2/k}}.
\end{align}
Since 
\begin{align}
\mu_{k}(r_{n})=2\big(\mu_{k}(r_{n})-\mu_{k}(r_{n+1})\big)=4\big(\mu_{k}(r_{n+1})-\mu_{k}(r_{n+2})\big)\,,
\end{align}
we can continue to get
\begin{align}
\big\vert \bar{f}_{l+L,k}(x) - \bar{f}_{l,k}(x) \big\vert &\leq 4C_{\mu_{k}}I_{k}^{1/k}\sum_{n=l}^{\infty}\frac{\mu_{k}(r_{n+1})-\mu_{k}(r_{n+2})}{\big\vert B_{R}(r_{n+1}(k))\big\vert^{2/k}}\nonumber\\
&\leq 4C_{\mu_{k}}I_{k}^{1/k}\sum_{n=l}^{\infty}\int_{r_{n+2}(k)}^{r_{n+1}(k)}\frac{\d \mu_{k}(r)}{\big\vert B_{R}(r_{n+1}(k))\big\vert^{2/k}}\nonumber\\
&\leq 4C_{\mu_{k}}I_{k}^{1/k}\int_{0}^{r_{l+1}(k)}\frac{\d \mu_{k}(r)}{\big\vert B_{R}(r)\big\vert^{2/k}}.
\end{align}
By letting $L \to \infty$, we have that
\begin{equation}
\bar{f}_{k}=\lim_{n \to \infty}\bar{f}_{n,k}
\end{equation}
exists and for each integer $l \ge 0$,
\begin{equation}
\sup_{x \in S_{R}^{2}}\big\vert \bar{f}_{k}(x) - \bar{f}_{l,k}(x) \big\vert \leq 4C_{\mu_{k}} I_{k}^{1/k} \int_{0}^{r_{l+1}(k)} \vert B_{R}(r)\vert^{-2/k}\d \mu_{k}(r).
\end{equation}
To prove the last statement, let $\phi : S_{R}^{2} \rightarrow \R$ be a continuous function. By \eqref{diffavg},
\begin{align}
\int_{S_{R}^{2}}\phi(x)\bar{f}_{k}(x)\d x &=\lim_{n \to \infty}\int_{S_{R}^{2}}\phi(x)\bar{f}_{n,k}(x)\d x \nonumber\\
%&=\lim_{n \to \infty}\int_{S_{R}^{2}}\frac{\phi(x)}{B_{R}(x\,,r_{n}(k))}\int_{B_{R}(x,r_{n}(k))}f(z)\d z \d x \nonumber\\
%&=\lim_{n \to \infty} \frac{1}{B_{R}(r_{n}(k))}\int_{S_{R}^{2}}\int_{S_{R}^{2}}\phi(x)f(z)\1_{B_{R}(x,r_{n}(k))}(z)\d z \d x \nonumber\\
%&=\lim_{n \to \infty} \frac{1}{B_{R}(r_{n}(k))}\int_{S_{R}^{2}}\int_{S_{R}^{2}}\phi(z)f(x)\1_{B_{R}(z,r_{n}(k))}(x)\d x \d z \nonumber\\
%&=\lim_{n \to \infty} \frac{1}{B_{R}(r_{n}(k))}\int_{S_{R}^{2}}\int_{S_{R}^{2}}\phi(z)f(x)\1_{\{R\theta(x,z)<r_{n}(k)\}}\d x \d z \nonumber\\
%&=\lim_{n \to \infty} \frac{1}{B_{R}(r_{n}(k))}\int_{S_{R}^{2}}\int_{S_{R}^{2}}\phi(z)f(x)\1_{B_{R}(x,r_{n}(k))}(z)\d x \d z \nonumber\\
%&=\lim_{n \to \infty}\int_{S_{R}^{2}}\frac{f(x)}{B_{R}(x\,,r_{n}(k))}\int_{B_{R}(x,r_{n}(k))}\phi(z)\d z \d x \nonumber\\
&=\lim_{n \to \infty}\int_{S_{R}^{2}}f(x)\bar{\phi}_{n,k}(x)\d x \nonumber\\
&=\int_{S_{R}^{2}}f(x)\phi(x)\d x.
\end{align}
This implies $f=\bar{f}_{k}\text{ a.e.}$ 
\end{proof}

\begin{proof}[Proof of Theorem \ref{GarsiaTheorem1.2}]
Suppose $r_{l+1}(k) \leq R\theta(x\,,x') \leq r_{l}(k)$ for some nonnegative integer $l$. Then by the triangle inequality and Lemma \ref{GarsiaLemma1.1},
\begin{align}\label{1.2.0}
\big\vert \bar{f}_{k}(x)-\bar{f}_{k}(x') \big\vert &\leq 2 \sup_{z \in S_{R}^{2}} \big\vert \bar{f}_{k}(z)-\bar{f}_{l,k}(z)\big\vert + \big\vert \bar{f}_{l,k}(x)-\bar{f}_{l,k}(x')\big\vert \nonumber\\
&\leq 8C_{\mu_{k}} I_{k}^{1/k} \int_{0}^{R\theta(x\,,x')} \vert B_{R}(r)\vert^{-2/k}\d \mu_{k}(r)+\big\vert \bar{f}_{l,k}(x)-\bar{f}_{l,k}(x')\big\vert.
\end{align}
We can use a similar argument as in the proof of the previous lemma to estimate the last term and get
\begin{align}
\big\vert \bar{f}_{l,k}(x)-\bar{f}_{l,k}(x')\big\vert^{k} %&= \Bigg\vert \frac{1}{\vert B_{R}(r_{l}(k))\vert} \int_{B_{R}(r_{l}(k))}f(x\tilde{+}z) \d z- \frac{1}{\vert B_{R}(r_{l}(k))\vert} \int_{B_{R}(r_{l}(k))}f(x' \tilde{+}z) \d z\Bigg\vert^{k} \nonumber\\
%&=\Bigg\vert \frac{1}{\vert B_{R}(r_{l}(k))\vert^{2}}\int_{B_{R}(r_{l}(k))}\d z \int_{B_{R}(r_{l}(k))}\d y \left(f(x\tilde{+}z)-f(x' \tilde{+}y)\right)\Bigg\vert^{k}\nonumber\\
%&\leq \frac{1}{\vert B_{R}(r_{l}(k))\vert^{2}}\int_{B_{R}(r_{l}(k))}\d z \int_{B_{R}(r_{l}(k))}\d y \vert f(x\tilde{+}z)-f(x'\tilde{+}y)\vert^{k}\nonumber\\
&\leq \frac{\alpha^{k}}{\vert B_{R}(r_{l}(k))\vert^{2}}\int_{B_{R}(r_{l}(k))}\d z \int_{B_{R}(r_{l}(k))}\d y \frac{\big\vert f(x\tilde{+}z)-f(x'\tilde{+}y)\big\vert^{k}}{\big\vert \mu\big(R\theta(x\tilde{+}z, x'\tilde{+}y)\big)\big\vert^{k}}\nonumber\\
&=\frac{\alpha^{k}I_{k}}{\vert B_{R}(r_{l}(k))\vert^{2}}\,,
\end{align}
for any $\alpha \ge \sup_{z \in B_{R}(r_{l}(k))}\sup_{y \in B_{R}(r_{l}(k))}\mu_{k}(R\theta(x\tilde{+}z,x'\tilde{+}y))$.\\
Let $\alpha$ converges to $\sup_{z \in B_{R}(r_{l}(k))}\sup_{y \in B_{R}(r_{l}(k))}\mu_{k}\big(R\theta(x\tilde{+}z,x' \tilde{+}y)\big)$ from above to get
\begin{equation}
\big\vert \bar{f}_{l,k}(x)-\bar{f}_{l,k}(x')\big\vert \leq \left(\sup_{z \in B_{R}(r_{l}(k))}\sup_{y \in B_{R}(r_{l}(k))}\mu_{k}\big(R\theta(x\tilde{+}z,x'\tilde{+}y)\big)\right)I_{k}^{1/k}\vert B_{R}(r_{l}(k))\vert^{-2/k}.
\end{equation}
For any $z\,,y \in B_{R}\big(r_{l}(k)\big)$, 
\begin{align}
\theta(x\tilde{+}z\,,x'\tilde{+}y)&\leq \theta(x\tilde{+}z\,,x)+\theta(x\,,x')+\theta(x'\,,x'\tilde{+}y)\nonumber\\
&\leq 3r_{l}(k)/R\nonumber\\
&<4r_{l}(k)/R.
\end{align}
Hence,
\begin{align}\label{1.2.1}
\big\vert \bar{f}_{l,k}(x)-\bar{f}_{l,k}(x')\big\vert &\leq C_{\mu_{k}}^{2}\mu_{r_{l}(k)}I_{k}^{1/k}\vert B_{R}(r_{l}(k))\vert^{-2/k}\nonumber\\
%&\leq 2C_{\mu_{k}}^{2}\left(\mu_{r_{l}(k)}-\mu_{r_{l+1}(k)}\right)I_{k}^{1/k}\vert B_{R}(r_{l}(k))\vert^{-2/k}\nonumber\\
&\leq 4C_{\mu_{k}}^{2}\left(\mu_{r_{l+1}(k)}-\mu_{r_{l+2}(k)}\right)I_{k}^{1/k}\vert B_{R}(r_{l+1}(k))\vert^{-2/k}\nonumber\\
%&=4C_{\mu_{k}}^{2}I_{k}^{1/k}\int_{r_{l+2}(k)}^{r_{l+1}(k)}\frac{\d \mu(r)}{\vert B_{R}(r_{l+1}(k))\vert^{2/k}}\nonumber\\
%&\leq 4C_{\mu_{k}}^{2}I_{k}^{1/k}\int_{0}^{r_{l+1}(k)}\frac{\d \mu(r)}{\vert B_{R}(r)\vert^{2/k}}\nonumber\\
&\leq 4C_{\mu_{k}}^{2}I_{k}^{1/k}\int_{0}^{R\theta(x,x')}\frac{\d \mu(r)}{\vert B_{R}(r)\vert^{2/k}}.
\end{align}
Use \eqref{1.2.1} in \eqref{1.2.0} to get
\begin{equation}
\big\vert \bar{f}_{k}(x)-\bar{f}_{k}(x')\big\vert \leq 4C_{\mu_{k}}(2+C_{\mu_{k}})I_{k}^{1/k}\int_{0}^{R\theta(x,x')}\vert B_{R}(r)\vert^{-2/k}\d \mu_{k}(r).
\end{equation}
\end{proof}
\break

\section*{Acknowledgement}
This research is partly supported by the NSF grant 1608575. The author would also like to thank Davar Khoshnevisan for many motivating discussions and his careful reading of the manuscript and Thomas Alberts for his helpful advice on editting of the manuscript.

\begin{small}

\vskip1cm

\noindent\textbf{Weicong Su.}
Department of Mathematics, The
University of Utah, 155 S. 1400 E. Salt Lake City,
UT 84112-0090, USA.\\
\texttt{su@math.utah.edu}\\

\end{small}
\end{document}